\documentclass[11pt]{article}
\usepackage{graphicx} % Required for inserting images
\usepackage{bbm}
\usepackage{bm}
\usepackage[margin=1in]{geometry}
\usepackage{color}
\usepackage{amsfonts}
\usepackage{algorithm}
\usepackage{algorithmic}
\usepackage{latexsym, amssymb, amsmath, amscd, amsthm, amsxtra}
\usepackage{mathtools}
\usepackage{enumerate}
\usepackage[all]{xy}
\usepackage{mathrsfs}
\usepackage{fancyhdr}
\usepackage{listings}
\usepackage{hyperref}
\usepackage{enumitem}
\usepackage{multirow}
\usepackage{tikz}
\usepackage{bbm}

\def\11{\mathbbm{1}}

\def\ER{Erd\H{o}s-R\'enyi\ }

\newcommand{\Pb}{\mathbb P}
\newcommand{\Qb}{\mathbb Q}

\newtheorem*{oldthm}{Theorem}
\newtheorem{thm}{Theorem}[section]

\newtheorem{proposition}[thm]{Proposition}
\newtheorem{question}[thm]{Question}
\newtheorem{lemma}[thm]{Lemma}

\newtheorem{claim}[thm]{Claim}

\newtheorem{remark}[thm]{Remark}
\newtheorem{DEF}[thm]{Definition}
\newtheorem{assum}[thm]{Assumption}

\numberwithin{equation}{section}

\makeatletter
\newenvironment{breakablealgorithm}
{% \begin{breakablealgorithm}
		\begin{center}
			\refstepcounter{algorithm}% New algorithm
			\hrule height.8pt depth0pt \kern2pt% \@fs@pre for \@fs@ruled
			\renewcommand{\caption}[2][\relax]{% Make a new \caption
				{\raggedright\textbf{\ALG@name~\thealgorithm} ##2\par}%
				\ifx\relax##1\relax % #1 is \relax
				\addcontentsline{loa}{algorithm}{\protect\numberline{\thealgorithm}##2}%
				\else % #1 is not \relax
				\addcontentsline{loa}{algorithm}{\protect\numberline{\thealgorithm}##1}%
				\fi
				\kern2pt\hrule\kern2pt
			}
		}{% \end{breakablealgorithm}
		\kern2pt\hrule\relax% \@fs@post for \@fs@ruled
	\end{center}
}
\makeatother

\hypersetup{colorlinks=true,linkcolor=blue}

\title{The Algorithmic Phase Transition in Correlated Spiked Models}

\usepackage{authblk}
\author[1]{Zhangsong Li\thanks{Email: \textit{ramblerlzs@pku.edu.cn}. Partially supported by the National Key R$\&$D program of China (Project No. 2023YFA1010103) and the NSFC Key Program (Project No. 12231002).}}

\affil[1]{School of Mathematical Sciences, Peking University}

\date{\today}
\begin{document}
\maketitle

\begin{abstract}
    We study the computational task of detecting and estimating correlated signals in a pair of spiked matrices. Specifically, we consider two fundamental models: 

    (1) \emph{Correlated spiked Wigner model.} We observe
    $$
        \bm{X} = \tfrac{\lambda}{\sqrt{n}} xx^{\top} + \bm{W} \,, \quad \bm{Y} = \tfrac{\mu}{\sqrt{n}} yy^{\top} + \bm{Z} \,.
    $$
    Here $x,y \in \mathbb R^n$ are signal vectors with norm $\|x\|,\|y\| \approx\sqrt{n}$ and correlation $\langle x,y \rangle \approx \rho\|x\|\|y\|$, while $\bm{W},\bm{Z}$ are independent Wigner matrices. 

    (2) \emph{Correlated spiked Wishart (covariance) model.} We observe
    $$
        \bm{X}= \tfrac{\sqrt{\lambda}}{\sqrt{n}} x \bm{u}^{\top} + \bm W \,, \quad \bm{Y} = \tfrac{\sqrt{\mu}}{\sqrt{n}} y \bm{v}^{\top} + \bm{Z} \,.
    $$
    Here $x,y \in \mathbb R^n$ are signal vectors with norm $\|x\|,\|y\| \approx\sqrt{n}$ and correlation $\langle x,y \rangle \approx \rho\|x\|\|y\|$, $\bm{u},\bm{v} \in \mathbb R^N$ are independent Gaussian vectors sampled from $\mathcal N(0,\mathbb I_N)$, and $\bm W,\bm Z \in \mathbb R^{n*N}$ are independent Gaussian noise matrices with $N=\Theta(n)$. 
    
    We propose an efficient detection and estimation algorithm based on counting a specific family of edge-decorated cycles. The algorithm's performance is governed by the function
    $$
        F(\lambda,\mu,\rho,\gamma)=\max\Big\{ \frac{ \lambda^2 }{ \gamma }, \frac{ \mu^2 }{ \gamma }, \frac{ \lambda^2 \rho^2 }{ \gamma-\lambda^2+\lambda^2 \rho^2 } + \frac{ \mu^2 \rho^2 }{ \gamma-\mu^2+\mu^2 \rho^2 } \Big\} \,.
    $$
    We prove our algorithm succeeds for the correlated spiked Wigner model whenever $F(\lambda,\mu,\rho,1)>1$, and succeeds for the correlated spiked Wishart model whenever $F(\lambda,\mu,\rho,\tfrac{n}{N})>1$. Our result shows that an algorithm can leverage the correlation between the spikes to detect and estimate the signals even in regimes where efficiently recovering either $x$ from $\bm{X}$ alone or $y$ from $\bm{Y}$ alone is believed to be computationally infeasible.

    We complement our algorithmic results with evidence for a matching computational lower bound. In particular, we prove that when $F(\lambda,\mu,\rho,1)<1$ for the correlated spiked Wigner model and when $F(\lambda,\mu,\rho,\tfrac{n}{N})<1$ for the spiked Wishart model, all algorithms based on {\em low-degree polynomials} fails to distinguish $(\bm{X},\bm{Y})$ with two independent noise matrices. This strongly suggests that $F(\lambda,\mu,\rho,1)=1$ (respectively, $F(\lambda,\mu,\rho,\tfrac{n}{N})=1$) is the precise computation threshold for the correlated spiked Wigner (respectively, Wishart) model.
\end{abstract}

\tableofcontents

\newpage

\section{Introduction}{\label{sec:intro}}

\subsection{Spiked matrix models}{\label{subsec:spiked-model}}

The challenge of recovering a low-dimensional structure from high-dimensional noise is a central theme in high-dimensional statistics and machine learning. Spiked matrix models \cite{BBP05, ZHT06, JL09, LM19} provide a canonical framework for studying this problem, exhibiting important phenomena such as sharp phase transitions and statistical-to-computational gaps. Two fundamental examples are the {\em spiked Wigner model} and the {\em spiked Wishart (covariance) model}, defined as follows.

\begin{DEF}[Spiked Wigner model]{\label{def-spiked-Wigner}}
    We say $\bm Y \in \mathbb R^{n*n}$ is a spiked Wigner matrix, if for a signal vector (also known as the ``spike'') $x \sim \pi$ (here $\pi=\pi_n$ is some probability measure over $\mathbb R^n$ normalized such that $\|x\|\approx\sqrt{n}$), we have
    \begin{align*}
        \bm Y = \tfrac{\lambda}{\sqrt{n}} xx^{\top} + \bm Z \,.  
    \end{align*}
    Here $\bm Z$ is an $n*n$ Wigner matrix\footnote{That is, $\bm{W},\bm{Z}$ are symmetric matrices with $\{ \bm{W}_{i,j} \}_{i \leq j}, \{ \bm{Z}_{i,j} \}_{i \leq j}$ are independently sampled from normal distribution with mean $0$ and variance $1+\mathbf 1_{i=j}$.} and $\lambda>0$ is the signal-to-noise ratio. 
\end{DEF} 
\begin{DEF}[Spiked Wishart model]{\label{def-spiked-covariance}}
    We say $\bm Y \in \mathbb R^{n*N}$ is a spiked Wishart matrix, if for a signal vector $x \sim \pi$ (here again $\pi=\pi_n$ is some probability measure over $\mathbb R^n$ normalized such that $\|x\|\approx\sqrt{n}$), we have
    \begin{align*}
        \bm Y = \tfrac{\sqrt{\lambda}}{\sqrt{n}} x \bm{u}^{\top} + \bm Z \,.
    \end{align*}
    Here $\bm u \in \mathbb R^{N}$ is a Gaussian vector sampled independently from $\mathcal N(0,\mathbb I_N)$, $\bm Z$ is an $n*N$ matrix with i.i.d.\ standard normal entries and $\lambda>0$ is the signal-to-noise ratio. Note that equivalently we can write
    \begin{align*}
        \bm Y=\big( \bm Y_1,\ldots, \bm Y_N \big), \mbox{ where } \bm Y_1,\ldots, \bm Y_N \overset{i.i.d.}{\sim} \mathcal N\big( 0,\mathbb I_n+\tfrac{\lambda}{n} xx^{\top} \big) \mbox{ given } x \,.
    \end{align*}
\end{DEF} 

A natural approach is to use the top eigenpair $(\varsigma_1(\bm Y\bm Y^{\top}),v_1(\bm Y \bm Y^{\top}))$ for both detection (i.e., distinguishing the model from a pure noise matrix) and recovery (i.e., estimating the spike $x$), since for large $\lambda$ one expects the rank-one deformation above to create an outlier eigenvalue in $\bm Y\bm Y^{\top}$. This intuition motivates the study of the spiked matrix models in the lens of random matrix theory \cite{AGZ10, Tao12}, and the performance of the spectral method for the spiked Wishart model is characterized by the celebrated {\em Baik–Ben Arous–P\'ech\'e (BBP) transition} \cite{Joh01, BBP05, BS06}.

\begin{oldthm}
    Suppose that $n=\gamma N$ for some $\gamma=\Theta(1)$ and $\bm Y$ are defined in Definition~\ref{def-spiked-covariance}. When $\lambda^2 \leq \gamma$, the top eigenvalue $\varsigma_1(\frac{1}{N}\bm Y\bm Y^{\top})$ remains within the bulk of the noise spectrum, concentrating at $(1+\sqrt{\gamma})^2$ (matching the top eigenvalue of $\bm Z\bm Z^{\top}$) and the top eigenvector $v_1(\frac{1}{N}\bm Y\bm Y^{\top})$ is asymptotically orthogonal with the spike $x$. 
    
    In contrast, when $\lambda^2>\gamma$, the top eigenvalue $\varsigma_1(\frac{1}{N}\bm Y\bm Y^{\top})$ detaches from the bulk, concentrating at $(1+\lambda)(1+ \tfrac{\gamma}{\lambda})>(1+\sqrt{\gamma})^2$, and the top eigenvector $v_1(\frac{1}{N}\bm Y\bm Y^{\top})$ achieves a non-vanishing correlation with the spike $x$. 
\end{oldthm}

For the spiked Wigner model, the performance of the spectral method is characterized by a variant of the BBP transition \cite{FP07, CDMF09}. 

\begin{oldthm}
    Suppose that $\bm Y$ are defined in Definition~\ref{def-spiked-Wigner}. When $0\leq \lambda\leq 1$, the top eigenvalue $\varsigma_1(\frac{1}{\sqrt{n}}\bm Y)$ remains within the bulk of the noise spectrum, concentrating at $2$ (matching the top eigenvalue of $\bm Z$) and the top eigenvector $v_1(\frac{1}{\sqrt{n}}\bm Y)$ is asymptotically orthogonal with the spike $x$. 
    
    In contrast, when $\lambda>1$, the top eigenvalue $\varsigma_1(\frac{1}{\sqrt{n}}\bm Y)$ detaches from the bulk, concentrating at $\lambda+\frac{1}{\lambda}>2$, and the top eigenvector $v_1(\frac{1}{\sqrt{n}}\bm Y)$ achieves a non-vanishing correlation with the spike $x$.
\end{oldthm}

A central question is whether this spectral approach is optimal, both statistically and computationally. It turns out that from a statistical perspective, the spectral method is not always optimal. For ``simple'' and ``dense'' priors $\pi$ (e.g., the product distribution of $n$ standard normal variables or the uniform distribution over the hypercube $\{ -1,+1 \}^n$), it is known that no statistics achieves strong detection or weak recovery below the BBP threshold \cite{PWBM18}. On the other hand, a long line of work has demonstrated that if $\pi$ is supported on vectors with $pn$ non-zero entries for sufficiently small $p$ (the setting of sparse PCA), then exhaustive searches requiring time $\exp(\Omega(n))$ over sparse vectors can succeed at detection and recovery for certain values below the BBP threshold \cite{LKZ15, KXZ16, BDM+16, BMV+18, PWBM18, LM19, EKJ20}. From a computational perspective, it is shown in \cite{KWB22, MW25} that under mild assumptions on the prior $\pi$, a large class of algorithms namely those based on {\em low-degree polynomials} cannot solve the detection and recovery problem below the BBP threshold. This provides strong evidence that the BBP transition represents a fundamental computational barrier for a broad range of efficient algorithms, and suggests the emergence of a statistical-computational gap when the prior $\pi$ is sparse.

\subsection{Correlated spiked models}{\label{subsec:correlated-spiked-model}}

Modern data analysis increasingly involves multiple, related datasets. The field of {\em multi-modal learning} \cite{NKK+11, RT17} is built on the premise that jointly analyzing such datasets can yield more powerful inferences than processing each one in isolation. This is particularly relevant when signals are observed through different sensors or modalities, each providing a complementary view.

To theoretically understand the potential and limitations of multi-modal inference, recent works \cite{KZ25, MZ25+} have introduced a natural extension of the spiked matrix model called the {\em correlated spiked matrices model}. In this model, we observe a pair of matrices
\begin{equation}{\label{eq-def-correlated-spike-general}}
    \bm{X} = \sum_{k=1}^{r} \tfrac{\lambda_k}{\sqrt{n}} x_k u_k^{\top} + \bm{W}, \quad \bm{Y} = \sum_{k=1}^{r} \tfrac{\mu_k}{\sqrt{n}} y_k v_k^{\top} + \bm{Z} \,.
\end{equation}
Here $\bm{W},\bm{Z} \in \mathbb R^{n*N}$ are noise matrices, $\{ x_k,y_k: 1 \leq k \leq r\}$ are correlated spikes in the sense that $\langle x_k, y_k \rangle \approx \rho_k \| x_k \| \| y_k \|$ for some $\rho_k>0$, and $\{ \lambda_k, \mu_k \}$ are signal-to-noise ratio. Two basic problems regarding this model are as follows: (1) the detection problem, i.e., deciding whether $(\bm{X},\bm{Y})$ is sampled from the law of correlated spiked model or is sampled from the law of pure noise matrices; (2) the recovery problem, i.e., recovering the planted (correlated) spikes $(x_k,y_k)$ from the observation $(\bm X,\bm Y)$. 

Initial studies of this model have yielded intriguing insights. The authors of \cite{KZ25} characterized the detectability threshold of the Bayes-optimal estimator and compared it empirically to the performance of standard spectral methods like Partial Least Squares (PLS) \cite{WSE01, Pir06} and Canonical Correlation Analysis (CCA) \cite{Tom00, GW19}. In a complementary line of work, \cite{MZ25+} derived the high-dimensional limits for a natural spectral method based on the sample-cross-covariance matrix $\bm{S}=\bm{X} \bm{Y}^{\top}$ in the rank-one case ($r=1$), showing that this canonical spectral approach might be sub-optimal even in the simple rank-one setting. This motivates the following natural questions:
\begin{question}{\label{Main-Question}}
    \begin{enumerate}
        \item[(1)] Can we settle the exact computational threshold for inference in this correlated model?
        \item[(2)] What is the optimal algorithm that approaches the computational threshold for this model?
    \end{enumerate}
\end{question}

\subsection{Main results and discussions}

In this paper, we focus on the following special rank-one setting of \eqref{eq-def-correlated-spike-general}, which we call the correlated spiked Wigner model and the correlated spiked Wishart model. 

\begin{DEF}[Correlated spikes distribution]{\label{def-correlated-spikes}}
    Let $x,y \in \mathbb R^n$ are random vectors such that $\{ x_i,y_i \}_{1 \leq i \leq n}$ are i.i.d.\ sampled from $\pi_*$, where $\pi_*$ is the law of two correlated random variables satisfying the following moment conditions:
    \begin{equation}{\label{eq-moment-pi}}
        \begin{aligned}
        &\mathbb E_{(X,Y)\sim\pi_*}[X]= \mathbb E_{(X,Y)\sim\pi_*}[Y]=0 \,, \\
        &\mathbb E_{(X,Y)\sim\pi_*}[X^2]= \mathbb E_{(X,Y)\sim\pi_*}[Y^2]=1 \,, \\
        &\mathbb E_{(X,Y)\sim\pi_*}[XY] = \rho \in (0,1) \,.
        \end{aligned}
    \end{equation}
    Let $\pi=\pi_*^{\otimes n}$ be the law of $(x,y)$. Of particular interest are the following canonical priors:
    \begin{itemize}
        \item The correlated Gaussian prior: we let $\pi_*$ be the law of a pair of mean-zero normal variables with variance $1$ and correlation $\rho$;
        \item The correlated Rademacher prior: we let $\pi_*$ be the law of a pair of Rademacher variables (i.e., marginally each variable is uniformly distributed in $\{ -1,+1 \}$) with correlation $\rho$.
        \item The correlated sparse Rademacher prior: we let $p \in (0,1)$ and let $\pi_*$ be the law of $(\tfrac{BX}{\sqrt{p}},\tfrac{BY}{\sqrt{p}})$, where $B$ is an independent Bernoulli variable with parameter $p$ and $(X,Y)$ are correlated Rademacher variables with correlation $\rho$.
    \end{itemize}
\end{DEF}
\begin{DEF}[Correlated spiked Wigner model]{\label{def-correlated-spike-specific}}
    A pair of matrices $(\bm X,\bm Y)$ is called correlated spiked Wigner matrices if they can be written as
    \begin{equation}{\label{eq-def-correlated-spike-specific}}
        \bm{X} = \tfrac{\lambda}{\sqrt{n}} xx^{\top} + \bm{W}, \quad \bm{Y} = \tfrac{\mu}{\sqrt{n}} yy^{\top} + \bm{Z} \,.
    \end{equation}
    Here $(x,y)$ are correlated spikes sampled from $\pi$, $\bm W,\bm Z$ are independent Wigner matrices, and $\lambda,\mu\geq 0$ are signal-to-noise ratios. Denote $\Pb=\Pb_{\lambda,\mu,\rho,n}$ be the law of $(\bm X,\bm Y)$ under \eqref{eq-def-correlated-spike-specific}. In addition, denote $\Qb=\Qb_n$ be the law of $(\bm W,\bm Z)$. 
\end{DEF}
\begin{DEF}[Correlated spiked Wishart model]{\label{def-correlated-spike-covariance-specific}}
    A pair of matrices $(\bm X,\bm Y) \in \mathbb R^{n*N} \times \mathbb R^{n*N}$ is called correlated spiked Wishart matrices if they can be written as
    \begin{equation}{\label{eq-def-correlated-spike-covariance-specific}}
        \bm{X} = \tfrac{\sqrt{\lambda}}{\sqrt{n}} x \bm{u}^{\top} + \bm{W}, \quad \bm{Y} = \tfrac{\sqrt{\mu}}{\sqrt{n}} y \bm{v}^{\top} + \bm{Z} \,.
    \end{equation}
    Here $(x,y)$ are correlated spikes sampled from $\pi$, $\bm u,\bm v$ are independent random vectors sampled from $\mathcal N(0,\mathbb I_N)$, $\bm W,\bm Z \in \mathbb R^{n*N}$ are matrices with i.i.d.\ standard normal entries, and $\lambda,\mu\geq 0$ are signal-to-noise ratios. Denote $\overline\Pb=\overline\Pb_{\lambda,\mu,\rho,n,N}$ be the law of $(\bm X,\bm Y)$ under \eqref{eq-def-correlated-spike-covariance-specific}. In addition, denote $\overline\Qb=\overline\Qb_{n,N}$ be the law of $(\bm W,\bm Z)$. Note that equivalently we can write
    \begin{align*}
        (\bm X,\bm Y)=\big( \bm X_1,\ldots,\bm X_N;\bm Y_1,\ldots, \bm Y_N \big)  \,,
    \end{align*}
    where given $x,y$ we have
    \begin{align*}
        \bm X_i \overset{i.i.d.}{\sim} \mathcal N\big( 0,\mathbb I_n+\tfrac{\lambda}{n} xx^{\top} \big) \mbox{ and } \bm Y_i \overset{i.i.d.}{\sim} \mathcal N\big( 0,\mathbb I_n+\tfrac{\mu}{n} yy^{\top} \big) \,.
    \end{align*}
\end{DEF}

We now formally define the detection and recovery problems in this model.

\begin{DEF}{\label{def-strong-detection}}
    We say an algorithm $\mathcal A$ that takes $(\bm X,\bm Y)$ as input and outputs an element in $\{ 0,1 \}$ achieves {\em strong detection} between $\Pb$ and $\Qb$, if
    \begin{align}{\label{eq-def-strong-detection}}
        \Pb\big( \mathcal A(\bm{X},\bm{Y})=0 \big) + \Qb\big( \mathcal A(\bm X,\bm Y)=1 \big) \to 0 \mbox{ as } n \to \infty \,.
    \end{align}
\end{DEF}
\begin{DEF}{\label{def-weak-recovery}}
    We say an algorithm $\mathcal A$ that takes $(\bm X,\bm Y)$ as input and outputs an estimator $(\widehat{x},\widehat{y}) \in \mathbb R^n \times \mathbb R^n$ achieves {\em weak recovery}, if
    \begin{equation}{\label{eq-def-weak-recovery}}
        \mathbb E_{\Pb}\Big[ \tfrac{ |\langle \widehat x,x \rangle| }{ \| \widehat x \| \| x \| } + \tfrac{ |\langle \widehat y,y \rangle| }{ \| \widehat y \| \| y \| } \Big] \geq \epsilon \mbox{ for some constant } \epsilon>0  \,.
    \end{equation}
\end{DEF}

In this work, we completely resolve Question~\ref{Main-Question} for the models \eqref{eq-def-correlated-spike-specific} and \eqref{eq-def-correlated-spike-covariance-specific}. Define
\begin{equation}{\label{eq-def-F(lambda,mu,rho)}}
    F(\lambda,\mu,\rho,\gamma)=\max\Big\{ \frac{ \lambda^2 }{ \gamma }, \frac{ \mu^2 }{ \gamma }, \frac{ \lambda^2 \rho^2 }{ \gamma-\lambda^2+\lambda^2 \rho^2 } + \frac{ \mu^2 \rho^2 }{ \gamma-\mu^2+\mu^2 \rho^2 } \Big\} \,.
\end{equation}
Our positive results provide efficient detection and recovery algorithms that succeed whenever $F(\lambda,\mu,\rho,1)>1$ for correlated spiked Wigner model, and $F(\lambda,\mu,\rho,\tfrac{n}{N})>1$ for correlated spiked Wishart model.

\begin{thm}[Informal]{\label{MAIN-THM-upper-bound-informal}}
    Suppose that $F(\lambda,\mu,\rho,1)>1$. Then, under mild assumptions on the prior $\pi$ (see Assumption~\ref{assum-upper-bound} for details), there exists two algorithms $\mathcal A,\mathcal A'$ with polynomial running time such that $\mathcal A$ (respectively, $\mathcal A'$) takes $\bm X,\bm Y$ defined in \eqref{eq-def-correlated-spike-specific} as input and achieves strong detection (respectively, weak recovery).
\end{thm} 
\begin{thm}[Informal]{\label{MAIN-THM-upper-bound-informal-covariance}}
    Suppose that $\tfrac{n}{N}=\gamma$ for some $\gamma=\Theta(1)$ and $F(\lambda,\mu,\rho,\gamma)>1$. Then, under mild assumptions on the prior $\pi$ (see Assumption~\ref{assum-upper-bound} for details), there exists two algorithms $\mathcal A,\mathcal A'$ with polynomial running time such that $\mathcal A$ (respectively, $\mathcal A'$) takes $\bm X,\bm Y$ defined in \eqref{eq-def-correlated-spike-covariance-specific} as input and achieves strong detection (respectively, weak recovery).
\end{thm}

Our algorithm is based on subgraph counting, a method widely used for network analysis in both theory \cite{MNS15, BDER16} and practice \cite{ADH+08, RPS+21}. For example, in the context of community detection, counting cycles of logarithmic length has been shown to achieve the optimal detection threshold for distinguishing the stochastic block model (SBM) with symmetric communities from the \ER graph model in the sparse regime \cite{MNS15, MNS24}. Similarly, in the dense regime, counting signed cycles turns out to achieve the optimal asymptotic power \cite{Ban18, BM17}. In addition, algorithms based on counting non-backtracking or self-avoiding walks \cite{Mas14, MNS18, BLM15, AS15, AS18} have been proposed to achieve sharp threshold with respect to the related community recovery problem.

Our main innovation is to generalize this approach to the two-matrix setting. Let us focus on the correlated Wigner model for simplicity. Whereas prior work counts cycles in a single graph, we count \emph{decorated cycles} (see Section~\ref{subsec:notation} for a formal definition) in the pair of matrices $(\bm X,\bm Y)$. Specifically, we consider cycles where each edge is ``decorated'' according to whether it is taken from $\bm X$ or $\bm Y$. The choice of such decorations on edges ensures that the number of decorated cycles grows exponentially faster than the number of unlabeled cycles, which enables us to surpass the BBP threshold in a single spiked model. While the concept is natural, its analysis presents a significant technical challenge since we need to control the complex correlations between counts of different decorated cycles. To overcome this, we introduce a delicate \emph{weighting scheme}, where each decorated cycle is assigned a weight based on its specific combinatorial structure (see \eqref{eq-def-f-mathcal-H} and \eqref{eq-def-Phi-i,j-mathcal-J} for details). This weighted approach is a key departure from ``unweighted'' spectral methods like Partial Least Squares (PLS) or Canonical Correlation Analysis (CCA). We believe such weighting is indeed crucial for an efficient algorithm to achieve the exact computational threshold in this model, and thus we tend to feel that the PLS/CCA method are suboptimal in our model.

We complement our algorithmic result by showing evidence for a matching computational lower bound, suggesting that reaching detection and recovery in the regime $F(\lambda,\mu,\rho)<1$ would require a rather different approach and might be fundamentally impossible for all efficient algorithms.

\begin{thm}[Informal]{\label{MAIN-THM-lower-bound-informal}}
    Suppose that $F(\lambda,\mu,\rho,1)<1$. Then, under mild assumptions on the prior $\pi$ (see Assumption~\ref{assum-lower-bound} for details), all algorithms based on low-degree polynomials fails to achieve strong detection between $\Pb$ and $\Qb$ defined in \eqref{eq-def-correlated-spike-specific}.
\end{thm}
\begin{thm}[Informal]{\label{MAIN-THM-lower-bound-informal-covariance}}
    Suppose that $\tfrac{n}{N}=\gamma$ for some $\gamma=\Theta(1)$ and $F(\lambda,\mu,\rho,\gamma)<1$. Then, under mild assumptions on the prior $\pi$ (see Assumption~\ref{assum-lower-bound} for details), all algorithms based on low-degree polynomials fails to achieve strong detection between $\overline\Pb$ and $\overline\Qb$ defined in \eqref{eq-def-correlated-spike-covariance-specific}.
\end{thm}

Our main approach to show Theorem~\ref{MAIN-THM-lower-bound-informal} is as follows. Again, let us focus on the correlated spiked Wigner model for simplicity. Firstly, using standard tools developed for the Gaussian additive model \cite{KWB22}, we can bound the optimal ``advantage'' of degree-$D$ polynomials by 
\begin{align}
    \mathbb E\Big[ \exp_{\leq D}\big( \tfrac{\lambda^2 \langle x,x' \rangle^2 + \mu^2 \langle y,y' \rangle^2 }{2n} \big) \Big] \,,  \label{eq-low-deg-adv-handwaving}
\end{align}
where $(x,y)$ and $(x',y')$ are independently sampled from $\pi$ and $\exp_{\leq D}(\cdot)$ is the truncated exponential function (see \eqref{eq-exp-leq-D}). The primary difficulty arises from the correlation between the inner products $\langle x,x' \rangle$ and $\langle y,y' \rangle$, which causes analyses designed for single spiked model (e.g., the approach in \cite{KWB22}) to break down. Our key insight is that using central limit theorem, we expect that $(\frac{\langle x,x' \rangle}{\sqrt{n}},\tfrac{ \langle y,y' \rangle }{ \sqrt{n} })$ behaves like a pair of standard normal variables $(U,V)$ with correlation $\rho^2$. This motivates us to bound \eqref{eq-low-deg-adv-handwaving} by a two-step procedure. Firstly, we bound the tractable ``Gaussian approximation'' of \eqref{eq-low-deg-adv-handwaving}, defined to be $\mathbb E[ \exp_{\leq D}( \tfrac{\lambda^2 U^2 + \mu^2 V^2}{2}) ]$. Secondly, we show that such Gaussian approximation is indeed valid for all low-order moments via a delicate Lindeberg's interpolation argument, which allows us to transfer the bound from the Gaussian setting back to the original quantity in \eqref{eq-low-deg-adv-handwaving}. We believe that our approach is not only effective for this problem but also constitutes a general and easily implementable methodology that could be applied to a wider range of problems involving correlated priors.  

\begin{remark}
    Combining Theorems~\ref{MAIN-THM-upper-bound-informal} and \ref{MAIN-THM-lower-bound-informal} (respectively, Theorems~\ref{MAIN-THM-upper-bound-informal-covariance} and \ref{MAIN-THM-lower-bound-informal-covariance}), our results suggest that model \eqref{eq-def-correlated-spike-specific} (respectively, model \eqref{eq-def-correlated-spike-covariance-specific}) exhibits a sharp computational transition at the threshold $F(\lambda,\mu,\rho)=1$ (respectively, $F(\lambda,\mu,\rho,\tfrac{n}{N})=1$). Notably, our results show that an algorithm can leverage the correlation between the spikes to detect the signals even in regimes where efficiently recovering either $x$ from $\bm{X}$ alone or recovering $y$ from $\bm{Y}$ alone is believed to be computationally infeasible. 
\end{remark}

\begin{remark}
    We now compare the computational threshold established in this work with the PLS/CCA thresholds identified in \cite{BHPZ19, MY23, BG23+, MZ25+} for the correlated spiked Wishart model \eqref{eq-def-correlated-spike-covariance-specific}. For the model \eqref{eq-def-correlated-spike-covariance-specific}, a common approach is the Canonical Correlation Analysis (CCA) \cite{BHPZ19, GW19, Yang22a, Yang22b, BG23+}, where the idea is to extract information of the spike from the MANOVA matrix 
    \begin{align*}
        (\bm{X} \bm{X}^{\top})^{-\frac{1}{2}} (\bm{Y} \bm{X}^{\top}) (\bm{Y} \bm{Y}^{\top})^{-\frac{1}{2}} \,.
    \end{align*} 
    (Note that the CCA approach requires $\bm X \bm X^{\top}$ to be invertible, restricting its applicability to the setting $n \leq N$.)
    It was shown in \cite{BHPZ19, MY23, BG23+} that the CCA approach achieves weak recovery if and only if
    \begin{align*}
        \frac{ \lambda^2 \mu^2 \rho^2 (\gamma^{-1}-1) }{ (\lambda^2+1)(\mu^2+1) } > 1 \mbox{ where } \gamma=\frac{n}{N} \,,
    \end{align*}
    This condition is strictly suboptimal compared to our computational threshold $F(\lambda,\mu,\rho,\gamma)=1$.
    
    Another widely used method for \eqref{eq-def-correlated-spike-covariance-specific} is the Partial Least Squares (PLS), which analyzes the sample-cross-covariance matrix $\bm{S}=\bm{X} \bm{Y}^{\top}$. In \cite{MZ25+} the authors shown that the PLS approach achieves weak recovery if and only if $\tau \leq \gamma$, where $\tau=\tau(\lambda,\mu,\rho)$ is the largest solution of 
    \begin{align*}
        1 + (1-\rho^2 \lambda^2 \mu^2 -\lambda^2-\mu^2) X + (\lambda^2\mu^2 -\lambda^2 -\mu^2) X^2 +\lambda^2 \mu^2 X^3 \,.
    \end{align*}
    This PLS threshold is also strictly suboptimal compared to our computational threshold $F(\lambda,\mu,\rho,\gamma)=1$. Figure~\ref{fig:phase-diagram} illustrates this comparison, depicting the exact computational threshold alongside the CCA and PLS thresholds for a specific parameter setting ($\gamma=0.25$, $\rho=0.99$).
\end{remark}

\begin{figure}[!ht]
    \centering
    \vspace{0cm}
    \includegraphics[height=7.62cm,width=10.70cm]{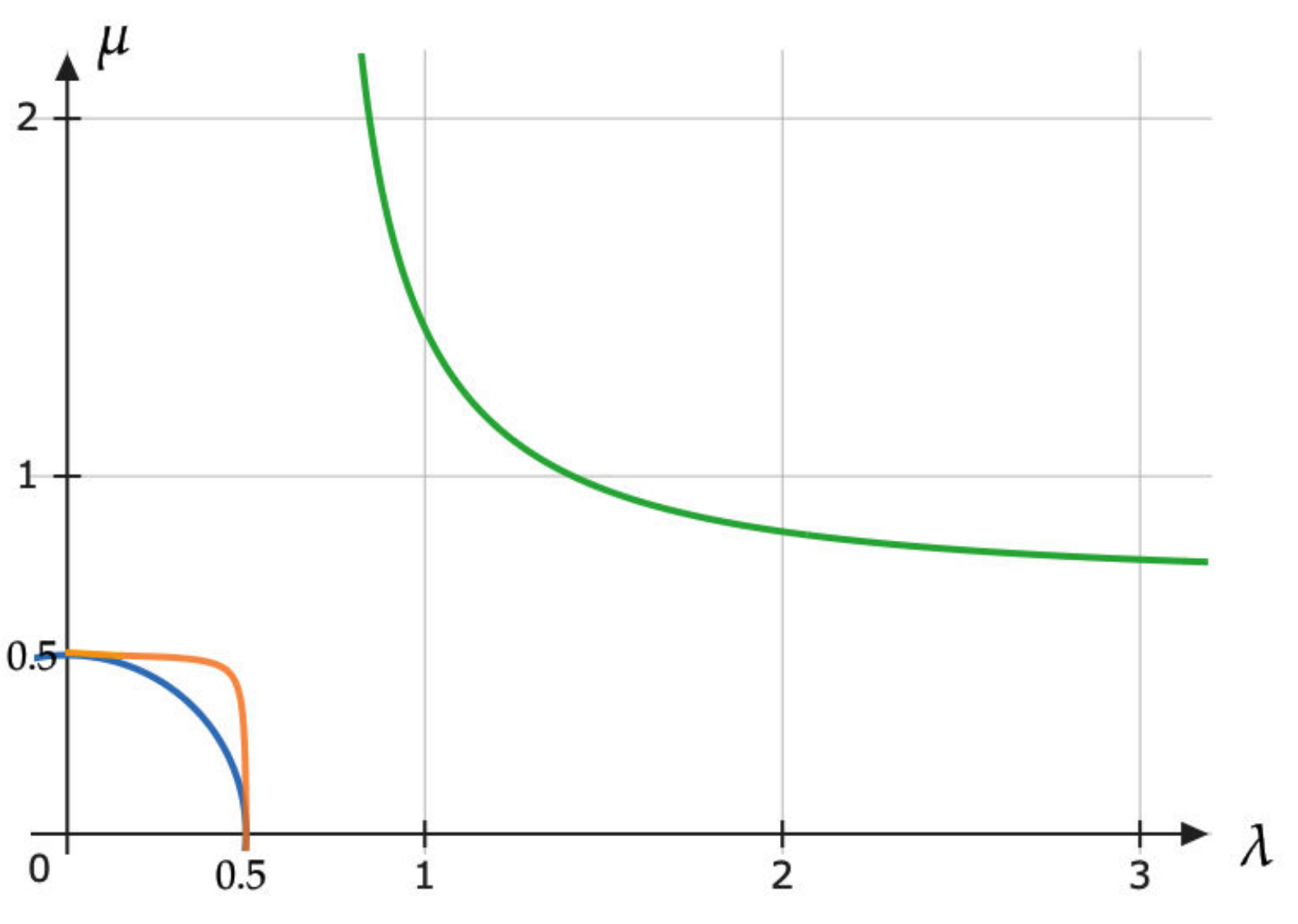}
    \caption{\noindent Phase diagram in the $(\lambda,\mu)$ plane illustrating the thresholds for the subgraph counts method (blue, this work), the PLS method (orange, \cite{MZ25+}), and the CCA method (green, \cite{BHPZ19, MY23, BG23+}). Here we take $\gamma=0.25$ and $\rho=0.99$.}
    \label{fig:phase-diagram}
\end{figure}

\begin{remark}
    For the correlated spiked Wigner model \eqref{eq-def-correlated-spike-specific}, a direct comparison with the PLS threshold or the CCA threshold is challenging, as both the PLS method and the CCA method are less understood in this setting. However, we suspect that the PLS/CCA approach might still be strictly suboptimal in our model, as we discussed following Theorem~\ref{MAIN-THM-upper-bound-informal-covariance}.
\end{remark}

\subsection{Notation and paper organization}{\label{subsec:notation}}

We record in this subsection some notation conventions. For a matrix or a vector $M$, we will use $M^{\top}$ to denote its transpose. For a $k*k$ matrix $M=(m_{ij})_{k*k}$, let $\mathsf{det}(M)$ and $\mathsf{tr}(M)$ be the determinant and trace of $M$, respectively. Denote $M \succ 0$ if $M$ is positive definite and $M \succeq 0$ if $M$ is semi-positive definite. For two $k*l$ matrices $M_1$ and $M_2$, we define their inner product to be
\begin{align*}
    \big\langle M_1,M_2 \big\rangle:=\sum_{i=1}^k \sum_{j=1}^l M_1(i,j)M_2(i,j) \,.
\end{align*}
We will use $\mathbb{I}_{k}$ to denote the $k*k$ identity matrix (and we drop the subscript if the dimension is clear from the context). We will use the following notation conventions on graphs.

{\em Labeled graphs}. Denote by $\mathsf K_n$ the complete graph with vertex set $[n]$. For any graph $H$, let $V(H)$ denote the vertex set of $H$ and let $E(H)$ denote the edge set of $H$. A graph $H = (V(H), E(H))$ is bipartite if its vertex set can be partitioned as $V(H) = V^{\mathsf a}(H) \sqcup V^{\mathsf b}(H)$ such that every edge in $E(H)$ has one endpoint in $V^{\mathsf a}(H)$ and the other in $V^{\mathsf b}(H)$. We denote by $\mathsf K_{n,N}$ the complete bipartite graph with vertex parts $[n]$ and $[N]$. We say $H$ is a subgraph of $G$, denoted by $H\subset G$, if $V(H) \subset V(G)$ and $E(H) \subset E(G)$. For all $v \in V(H)$, define $\mathsf{deg}_H(v)=\#\{ e \in E(H): v \in e \}$ to be the degree of $v$ in $H$. We say $v$ is an isolated vertex of $H$, if $\mathsf{deg}_H(v)=0$. Denote $\mathsf I(H)$ be the set of isolated vertices of $H$. We say $v$ is a leaf of $H$, if $\mathsf{deg}_H(v)=1$. Denote $\mathsf{L}(H)$ be the set of leaves in $H$. For $H,S \subset \mathsf K_n$, denote by $H \cap S$ the graph with vertex set given by $V(H) \cap V(S)$ and edge set given by $E(H)\cap E(S)$, and denote by $S \cup H$ the graph with vertex set given by $V(H) \cup V(S)$ and edge set $E(H) \cup E(S)$. In addition, denote by $S \Cap H$ the graph induced by the edge set $E(S)\cap E(H)$ (in particular, this induced graph have no isolated vertices).  

{\em Decorated graphs}. We say a triple $H=(V(H),E(H);\gamma(H))$ is a decorated graph, if the pair $(V(H),E(H))$ forms a graph and $\gamma(H):E(H) \to \{ \bullet,\circ \}$ is a function that assigns a ``decoration'' to each edge. Denote
\begin{equation}{\label{eq-def-E-1-E-2-decorated}}
    E_{\bullet}(H)=\{ e \in E(H): \gamma(e)=\bullet \} \mbox{ and } E_{\circ}(H)=\{ e \in E(H): \gamma(e)=\circ \}
\end{equation}
In addition, define
\begin{equation}{\label{eq-def-V-1-V-2-decorated}}
    \begin{aligned}
        &V_{\bullet}(H) = \{ v \in V(H): (v,u) \in E_{\bullet}(H) \mbox{ for some } u \in V(H) \} \,;  \\
        &V_{\circ}(H) = \{ v \in V(H): (v,u) \in E_{\circ}(H) \mbox{ for some } u \in V(H) \} \,.
    \end{aligned}
\end{equation}
If $H$ is a bipartite graph, define $V^{\mathsf a}_{\bullet},V^{\mathsf b}_{\circ},V^{\mathsf b}_{\bullet},V^{\mathsf b}_{\circ}$ with respect to $V^{\mathsf a},V^{\mathsf b}$ in the similar manner. Also define $H_{\bullet},H_{\circ}$ to the the subgraph of $H$ with
\begin{equation}{\label{eq-def-H-1-H-2}}
    V(H_{\bullet}) = V_{\bullet}(H), \ E(H_{\bullet})=E_{\bullet}(H); \quad V(H_{\circ})=V_{\circ}(H), \ E(H_{\circ})=E_{\circ}(H) \,.
\end{equation}
Finally, define
\begin{equation}{\label{eq-def-Dif(H)}}
    \mathsf{diff}(H) = V_{\bullet}(H) \cap V_{\circ}(H) \,.
\end{equation}

{\em Graph isomorphisms and unlabeled graphs.} Two graphs $H$ and $H'$ are isomorphic, denoted by $H\cong H'$, if there exists a bijection $\sigma:V(H) \to V(H')$ such that $(\sigma(u),\sigma(v)) \in E(H')$ if and only if $(u,v)\in E(H)$. Two decorated graphs $H$ and $H'$ are isomorphic, if there exists a bijection $\sigma:V(H) \to V(H')$ such that $\sigma$ maps $E_{\bullet}(H)$ to $E_{\bullet}(H')$ and maps $E_{\circ}(H)$ to $E_{\circ}(H')$. Denote by $[H]$ the isomorphism class of $H$; it is customary to refer to these isomorphic classes as unlabeled hypergraphs. Let $\mathsf{Aut}(H)$ be the number of automorphisms of $H$ (graph isomorphisms to itself). 

We use standard asymptotic notations: for two sequences $a_n$ and $b_n$ of positive numbers, we write $a_n = O(b_n)$, if $a_n<Cb_n$ for an absolute constant $C$ and for all $n$ (similarly we use the notation $O_h$ if the constant $C$ is not absolute but depends only on $h$); we write $a_n = \Omega(b_n)$, if $b_n = O(a_n)$; we write $a_n = \Theta(b_n)$, if $a_n =O(b_n)$ and $a_n = \Omega(b_n)$; we write $a_n = o(b_n)$ or $b_n = \omega(a_n)$, if $a_n/b_n \to 0$ as $n \to \infty$. For two real numbers $a$ and $b$, we let $a \vee b = \max \{ a,b \}$ and $a \wedge b = \min \{ a,b \}$. For two sets $A$ and $B$, we define $A\sqcup B$ to be the disjoint union of $A$ and $B$ (so the notation $\sqcup$ only applies when $A, B$ are disjoint). The indicator function of sets $A$ is denoted by $\mathbf{1}_{A}$. In addition, we use both $|A|$ and $\#A$ to denote the cardinality of $A$.
 
The rest of the paper is organized as follows. In Section~\ref{sec:main-alg-results}, we formally state the results in Theorem~\ref{MAIN-THM-upper-bound-informal} by presenting the construction of our detection and recovery statistics and presenting their theoretical guarantees. Section~\ref{sec:stat-analysis} provides the statistical analysis of our statistics. In Section~\ref{sec:hypergraph-color-coding}, we present efficient algorithms to approximately compute our statistics based on color coding. In Section~\ref{sec:lower-bound}, we formally state and prove the results mentioned in Theorem~\ref{MAIN-THM-lower-bound-informal}. Several auxiliary results are postponed to the appendix to ensure a smooth flow of presentation.

\section{Algorithmic results}{\label{sec:main-alg-results}}

In this section we formally state the results in Theorems~\ref{MAIN-THM-upper-bound-informal} and \ref{MAIN-THM-upper-bound-informal-covariance}. Note that for the correlated spiked Wigner model \eqref{eq-def-correlated-spike-specific}, when $\lambda>1$ (respectively, $\mu>1$), it is possible to detect and estimate the spike $x$ (respectively, $y$) using the single observation $\bm X$ (respectively, $\bm Y$).\footnote{In addition, since $x,y$ are positively correlated, we expect that an estimator $\widehat{x}=\widehat{x}(\bm X)$ that is positively correlated with $x$ should also be positively correlated with $y$.} Thus, for the model \eqref{eq-def-correlated-spike-specific} we may assume that
\begin{equation}{\label{eq-prelim-assumption-upper-bound}}
    0 \leq \lambda, \mu, \rho \leq 1 \mbox{ and } F(\lambda,\mu,\rho,1)>1+\epsilon \mbox{ for some constant } \epsilon>0 \,. 
\end{equation}
Similarly, for the model \eqref{eq-def-correlated-spike-covariance-specific} we may assume that 
\begin{equation}{\label{eq-prelim-assumption-upper-bound-cov}}
    \frac{n}{N}=\gamma,\  0 \leq \rho, \frac{\lambda^2}{\gamma}, \frac{\mu^2}{\gamma} \leq 1 \mbox{ and } F(\lambda,\mu,\rho,\gamma)>1+\epsilon \mbox{ for some constant } \epsilon>0 \,. 
\end{equation}
Now we state our assumptions of the prior $\pi$ in the upper bound, which will be assumed throughout Sections~\ref{sec:main-alg-results}, \ref{sec:stat-analysis} and \ref{sec:hypergraph-color-coding}.
\begin{assum}{\label{assum-upper-bound}}
    Suppose that $\pi=\pi_*^{\otimes n}$, where $\pi_*$ is the law of a pair of correlated variables $(\mathbb X,\mathbb Y)$ satisfying the following conditions: 
    \begin{enumerate}
        \item[(1)] $\mathbb E[\mathbb X]=\mathbb E[\mathbb Y]=0$, $\mathbb E[\mathbb X^2]=\mathbb E[\mathbb Y^2]=1$, and $\mathbb E[\mathbb X \mathbb Y]=\rho$. 
        \item[(2)] $\mathbb E[\mathbb X^4],\mathbb E[\mathbb Y^4] \leq C$ for some absolute constant $C$.
    \end{enumerate}
    Note that condition~(1) above is consistent with \eqref{eq-moment-pi}.
\end{assum}

\subsection{The detection statistics and theoretical guarantees}{\label{subsec:detect-stat}}

\subsubsection{The correlated spiked Wigner model}{\label{subsubsec:detect-stat-Wigner}}

For any decorated graph $S \subset \mathsf K_n$, define (recall \eqref{eq-def-E-1-E-2-decorated})
\begin{equation}{\label{eq-def-f-S}}
    f_S(\bm X,\bm Y) = \prod_{(i,j) \in E_{\bullet}(S)} \bm X_{i,j} \prod_{ (i,j) \in E_{\circ}(S) } \bm Y_{i,j} \,.
\end{equation}
Let $\ell\in\mathbb N$ be a parameter that will be decided later and define $\mathcal H=\mathcal H(\ell)$ to be the collection of decorated cycles with length $\ell$ (see Figure~\ref{fig:decorated-cycle} for an illustration). 
\begin{figure}[!ht]
    \centering
    \vspace{0cm}
    \includegraphics[height=5.19cm,width=11.86cm]{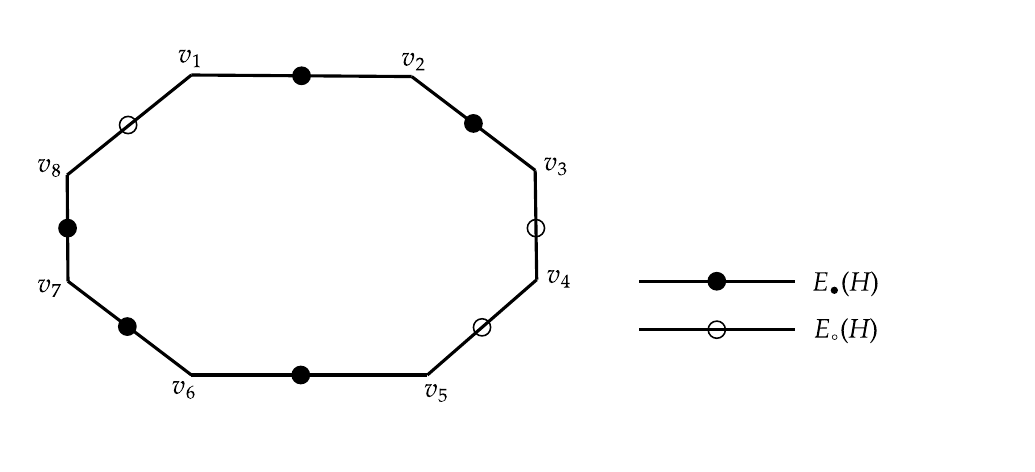}
    \caption{\noindent An unlabeled graph $[H] \in \mathcal H(\ell)$ with $\ell=8$. Here $\mathsf{diff}(H)=\{ v_1,v_3,v_5,v_8 \}$.}
    \label{fig:decorated-cycle}
\end{figure}

\noindent In addition, define
\begin{equation}{\label{eq-def-f-mathcal-H}}
    f_{\mathcal H}(\bm X,\bm Y) = \frac{1}{ \sqrt{n^{\ell} \beta_{\mathcal H}} } \sum_{ [H] \in \mathcal H } \Xi(H) \sum_{ S \subset \mathsf{K}_n, S \cong H } f_{S}(\bm X,\bm Y) \,, 
\end{equation}
where 
\begin{equation}{\label{eq-def-Xi-S}}
    \Xi(H) = \lambda^{|E_{\bullet}(H)|} \mu^{|E_{\circ}(H)|} \rho^{|\mathsf{diff}(H)|} 
\end{equation}
and 
\begin{equation}{\label{eq-def-beta-mathcal-H}}
    \beta_{\mathcal H} = \sum_{[H] \in \mathcal H} \frac{ \rho^{2\mathsf{diff}(H)} \lambda^{2|E_{\bullet}(H)|} \mu^{2|E_{\circ}(H)|} }{ |\mathsf{Aut}(H)| } = \sum_{[H] \in \mathcal H} \frac{ \Xi(H)^2 }{ |\mathsf{Aut}(H)| } \,.
\end{equation}
Algorithm~\ref{alg:detection-meta} below describes our proposed method for detection in correlated spiked Wigner model.
\begin{breakablealgorithm}{\label{alg:detection-meta}}
\caption{Detection in correlated spiked Wigner model}
    \begin{algorithmic}[1]
    \STATE {\bf Input:} Two symmetric matrices $\bm X,\bm Y$, a family $\mathcal H$ of non-isomorphic decorated graphs, and a threshold $\tau\geq 0$.
    \STATE Compute $f_{\mathcal H}(\bm X,\bm Y)$ according to \eqref{eq-def-f-mathcal-H}. 
    \STATE Let $\mathtt q=1$ if $f_{\mathcal H}(\bm X,\bm Y) \geq \tau$ and $\mathtt q=0$ if $f_{\mathcal H}(\bm X,\bm Y) < \tau$.
    \STATE {\bf Output:} $\mathtt q$.
    \end{algorithmic}
\end{breakablealgorithm}
We now state that Algorithm~\ref{alg:detection-meta} succeeds under Assumption~\ref{assum-upper-bound}, Equation~\eqref{eq-prelim-assumption-upper-bound} and an appropriate choice of the parameter $\ell$. To this end, we first show the following bound on $\beta_{\mathcal H}$.

\begin{lemma}{\label{lem-bound-beta-mathcal-H}}
    There exists a constant $D=D(\lambda,\mu,\rho)>1$ such that 
    \begin{equation}{\label{eq-bound-beta-mathcal-H}}
        \frac{D^{-1}}{\ell^2} A_+^{\ell} \leq \beta_{\mathcal H} \leq D A_+^{\ell} \,,
    \end{equation}
    where 
    \begin{equation}{\label{eq-def-A-+}}
        A_+ = \frac{ \lambda^2 + \mu^2 + \sqrt{ \lambda^4 + \mu^4 - (2-4\rho^4) \lambda^2 \mu^2 } }{ 2 } \,.
    \end{equation}
    In addition, suppose that $F(\lambda,\mu,\rho,\gamma)>1$, then there exists a constant $\delta=\delta(\lambda,\mu,\rho,\gamma)>0$ such that $A_+>(1+\delta)\gamma$. In particular, when \eqref{eq-prelim-assumption-upper-bound} is satisfied we have $A_+>1+\delta$ for some $\delta=\delta(\lambda,\mu,\rho)$. 
\end{lemma}

The proof of Lemma~\ref{lem-bound-beta-mathcal-H} is postponed to Section~\ref{subsec:proof-lem-2.2} of the appendix. Our main result for the detecting algorithm for model \eqref{eq-def-correlated-spike-specific} can be summarized as follows:

\begin{proposition}{\label{main-prop-detection}}
    Suppose that Assumption~\ref{assum-upper-bound} and Equation~\eqref{eq-prelim-assumption-upper-bound} hold, and we choose  
    \begin{align}{\label{eq-condition-strong-detection}}
        \omega(1)=\ell=o(\tfrac{\log n}{\log\log n})  \,.
    \end{align}
    Then we have
    \begin{align}{\label{eq-moment-control}}
        \mathbb E_{\Pb}\big[ f_{\mathcal H} \big]=\omega(1), \ \mathbb E_{\Qb}\big[ f_{\mathcal H}^2 \big]=1+o(1) \mbox{, and } \operatorname{Var}_{\Pb}\big[ f_{\mathcal H} \big]=o(1) \cdot \mathbb E_{\Pb}\big[ f_{\mathcal H} \big]^2 \,.
    \end{align}
\end{proposition}

Combining these variance bounds with Chebyshev’s inequality, we arrive at the following sufficient condition for the statistic $f_{\mathcal H}$ to achieve strong detection.

\begin{thm}{\label{MAIN-THM-detection}}
    Suppose that Assumption~\ref{assum-upper-bound} and Equation~\eqref{eq-prelim-assumption-upper-bound} hold and we choose $\ell$ according to \eqref{eq-condition-strong-detection}. Then the testing error satisfies
    \begin{equation}{\label{eq-testing-error}}
        \Pb\big( f_{\mathcal H}(\bm X,\bm Y)\leq\tau \big) + \Qb\big( f_{\mathcal H}(\bm X,\bm Y)\geq\tau \big) = o(1) \,,
    \end{equation}
    where the threshold is chosen as
    \begin{equation*}
        \tau = c \cdot \mathbb E_{\Pb}\big[ f_{\mathcal H}(\bm X,\bm Y) \big]
    \end{equation*}
    for any fixed constant $0<c<1$. In particular, Algorithm~\ref{alg:detection-meta} described above achieves strong detection between $\Pb$ and $\Qb$.
\end{thm}

From a computational perspective, evaluating each
\begin{align*}
    \sum_{ S \subset \mathsf K_n: S \cong H } f_S(\bm X,\bm Y)
\end{align*}
in \eqref{eq-def-f-mathcal-H} by exhaustive search takes $n^{O(\ell)}$ time which is super-polynomial when $\ell=\omega(1)$. To resolve this computational issue, in Section~\ref{subsec:approx-detection} 
we design an polynomial-time algorithm (see Algorithm~\ref{alg:cal-widetilde-f}) to compute an approximation $\widetilde{f}_{\mathcal H}(\bm X,\bm Y)$ (see \eqref{eq-def-widetilde-f-H}) for $f_{\mathcal H}(\bm X,\bm Y)$ using the strategy of color coding \cite{AYZ95, AR02, HS17, MWXY24, MWXY23, Li25+}. The following result shows that the statistic $\widetilde f_{\mathcal H}$ achieves strong detection under the same condition as in Theorem~\ref{MAIN-THM-detection}.

\begin{thm}{\label{MAIN-THM-detection-algorithmic}}
    Suppose that Assumption~\ref{assum-upper-bound} and Equation~\ref{eq-prelim-assumption-upper-bound} hold and we choose $\ell$ according to \eqref{eq-condition-strong-detection}. Then \eqref{eq-testing-error} holds with $\widetilde f_{\mathcal H}$ in place of $f_{\mathcal H}$, namely
    \begin{equation}{\label{eq-testing-error-algorithmic}}
        \Pb\big( \widetilde f_{\mathcal H}(\bm X,\bm Y)\leq\tau \big) + \Qb\big( \widetilde f_{\mathcal H}(\bm X,\bm Y)\geq\tau \big) = o(1) \,,
    \end{equation}
    Moreover, $\widetilde f_{\mathcal H}$ can be computed in time $n^{2+o(1)}$.
\end{thm}

\subsubsection{The correlated spiked Wishart model}{\label{subsubsec:detect-stat-cov}}

For any bipartite graph $S \subset \mathsf K_{n,N}$, we will always use the convention that an edge $(i,j)$ is written in the order that $i\in[n]$ and $j\in[N]$. For any decorated bipartite graph $S \subset \mathsf K_{n,N}$, define (recall \eqref{eq-def-E-1-E-2-decorated})
\begin{equation}{\label{eq-def-f-S-cov}}
    h_S(\bm X,\bm Y) = \prod_{(i,j) \in E_{\bullet}(S)} \bm X_{i,j} \prod_{ (i,j) \in E_{\circ}(S) } \bm Y_{i,j} \,.
\end{equation}
Let $\ell\in\mathbb N$ be a parameter that will be decided later and let $\mathcal G=\mathcal G(\ell)$ be the collection of all unlabeled, decorated, bipartite cycles with length $2\ell$ such that $\mathsf{diff}(H) \subset V^{\mathsf a}(H)$ for all $[H] \in\mathcal G$ (see Figure~\ref{fig:decorated-bipartite-cycle} for an illustration). 
\begin{figure}[!ht]
    \centering
    \vspace{0cm}
    \includegraphics[height=7.07cm,width=12.10cm]{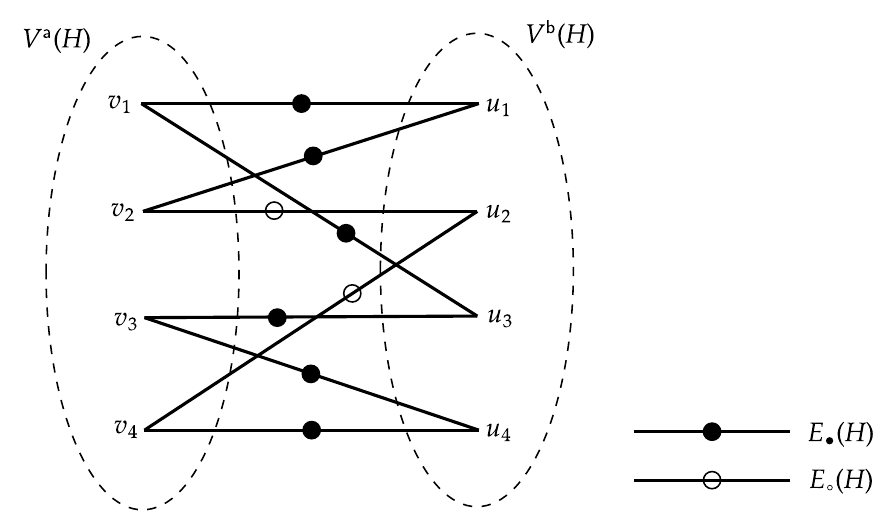}
    \caption{\noindent An unlabeled graph $[H] \in \mathcal G(\ell)$ with $\ell=4$. Here $\mathsf{diff}(H)=\{ v_2,v_4 \}$.}
    \label{fig:decorated-bipartite-cycle}
\end{figure}

\noindent Define 
\begin{equation}{\label{eq-def-f-mathcal-G}}
    h_{\mathcal G}(\bm X,\bm Y) = \frac{1}{ \sqrt{n^{\ell} N^{\ell} \beta_{\mathcal G}} } \sum_{ [H] \in \mathcal G } \Upsilon(H) \sum_{ S \subset \mathsf{K}_{n,N}, S \cong H } h_{S}(\bm X,\bm Y) \,, 
\end{equation}
where 
\begin{equation}{\label{eq-def-Upsilon-S}}
    \Upsilon(H) = \lambda^{|E_{\bullet}(H)|/2} \mu^{|E_{\circ}(H)|/2} \rho^{|\mathsf{diff}(H)|} 
\end{equation}
and 
\begin{equation}{\label{eq-def-beta-mathcal-G}}
    \beta_{\mathcal G} = \sum_{[H] \in \mathcal G} \frac{ \rho^{2\mathsf{diff}(H)} \lambda^{|E_{\bullet}(H)|} \mu^{|E_{\circ}(H)|} }{ |\mathsf{Aut}(H)| } = \sum_{[H] \in \mathcal G} \frac{ \Upsilon(H)^2 }{ |\mathsf{Aut}(H)| } \,.
\end{equation}
Algorithm~\ref{alg:detection-cov-meta} below describes our proposed method for detection in correlated spiked Wishart model.
\begin{breakablealgorithm}{\label{alg:detection-cov-meta}}
\caption{Detection in correlated spiked Wishart model}
    \begin{algorithmic}[1]
    \STATE {\bf Input:} Two matrices $\bm X,\bm Y$, a family $\mathcal G$ of non-isomorphic decorated bipartite graphs, and a threshold $\tau\geq 0$.
    \STATE Compute $h_{\mathcal G}(\bm X,\bm Y)$ according to \eqref{eq-def-f-mathcal-G}. 
    \STATE Let $\mathtt q=1$ if $h_{\mathcal G}(\bm X,\bm Y) \geq \tau$ and $\mathtt q=0$ if $h_{\mathcal G}(\bm X,\bm Y) < \tau$.
    \STATE {\bf Output:} $\mathtt q$.
    \end{algorithmic}
\end{breakablealgorithm}
We now state that Algorithm~\ref{alg:detection-cov-meta} succeeds under Assumption~\ref{assum-upper-bound}, Equation~\eqref{eq-prelim-assumption-upper-bound-cov} and an appropriate choice of the parameter $\ell$. Similarly as in Section~\ref{subsec:detect-stat}, we first show the following bound on $\beta_{\mathcal G}$.

\begin{lemma}{\label{lem-bound-beta-mathcal-G}}
    There exists a constant $D=D(\lambda,\mu,\rho,\gamma)>1$ such that (recall \eqref{eq-def-A-+}) 
    \begin{equation}{\label{eq-bound-beta-mathcal-G}}
        \frac{D^{-1}}{\ell^2} A_+^{\ell} \leq \beta_{\mathcal G} \leq D A_+^{\ell} \,,
    \end{equation}
    In addition, suppose that \eqref{eq-prelim-assumption-upper-bound-cov} holds, then there exists a constant $\delta=\delta (\lambda,\mu,\rho,\gamma)>0$ such that $A_+>(1+\delta)\gamma$.
\end{lemma}

The proof of Lemma~\ref{lem-bound-beta-mathcal-G} is incorporated in Section~\ref{subsec:proof-lem-2.6} of the appendix. Our main result for the detecting algorithm for model \eqref{eq-def-correlated-spike-covariance-specific} can be summarized as follows:

\begin{proposition}{\label{main-prop-detection-cov}}
    Suppose that Assumption~\ref{assum-upper-bound} and Equation~\eqref{eq-prelim-assumption-upper-bound-cov} hold and we choose $\ell$ according to \eqref{eq-condition-strong-detection}. Then we have
    \begin{align}{\label{eq-moment-control-cov}}
        \mathbb E_{\overline\Pb}\big[ h_{\mathcal G} \big]=\omega(1), \ \mathbb E_{\overline\Qb}\big[ h_{\mathcal G}^2 \big]=1+o(1) \mbox{, and } \operatorname{Var}_{\overline\Pb}\big[ h_{\mathcal G} \big]=o(1) \cdot \mathbb E_{\overline\Pb}\big[ h_{\mathcal G} \big]^2 \,.
    \end{align}
\end{proposition}

Combining these variance bounds with Chebyshev’s inequality, we arrive at the following sufficient condition for the statistic $h_{\mathcal G}$ to achieve strong detection.

\begin{thm}{\label{MAIN-THM-detection-cov}}
    Suppose that Assumption~\ref{assum-upper-bound} and Equation~\eqref{eq-prelim-assumption-upper-bound-cov} hold and we choose $\ell$ according to \eqref{eq-condition-strong-detection}. Then the testing error satisfies
    \begin{equation}{\label{eq-testing-error-cov}}
        \overline\Pb\big( h_{\mathcal G}(\bm X,\bm Y)\leq\tau \big) + \overline\Qb\big( h_{\mathcal G}(\bm X,\bm Y)\geq\tau \big) = o(1) \,,
    \end{equation}
    where the threshold is chosen as
    \begin{equation*}
        \tau = c \cdot \mathbb E_{\overline\Pb}\big[ h_{\mathcal G}(\bm X,\bm Y) \big]
    \end{equation*}
    for any fixed constant $0<c<1$. In particular, Algorithm~\ref{alg:detection-cov-meta} described above achieves strong detection between $\overline\Pb$ and $\overline\Qb$.
\end{thm}

Again, To resolve the computational issue, in Section~\ref{subsec:approx-detection} we design an polynomial-time algorithm (see Algorithm~\ref{alg:cal-widetilde-f-mathcal-G}) to compute an approximation $\widetilde{h}_{\mathcal G}(\bm X,\bm Y)$ (see \eqref{eq-def-widetilde-f-mathcal-G}) for $h_{\mathcal G}(\bm X,\bm Y)$ using the strategy of color coding. The following result shows that the statistic $\widetilde h_{\mathcal G}$ achieves strong detection under the same condition as in Theorem~\ref{MAIN-THM-detection-cov}.

\begin{thm}{\label{MAIN-THM-detection-algorithmic-cov}}
    Suppose that Assumption~\ref{assum-upper-bound} and Equation~\ref{eq-prelim-assumption-upper-bound} hold and we choose $\ell$ according to \eqref{eq-condition-strong-detection}. Then \eqref{eq-testing-error} holds with $\widetilde h_{\mathcal G}$ in place of $h_{\mathcal G}$, namely
    \begin{equation}{\label{eq-testing-error-algorithmic-cov}}
        \overline\Pb\big( \widetilde h_{\mathcal G}(\bm X,\bm Y)\leq\tau \big) + \overline\Qb\big( \widetilde h_{\mathcal G}(\bm X,\bm Y)\geq\tau \big) = o(1) \,,
    \end{equation}
    Moreover, $\widetilde h_{\mathcal G}$ can be computed in time $n^{2+o(1)}$.
\end{thm}

\subsection{The recovery statistics and theoretical guarantees}{\label{subsec:recovery-stat}}

\subsubsection{The correlated spiked Wigner model}{\label{subsubsec:recovery-stat-Wigner}}

We only show how to estimate $x$ and estimating $y$ can be done in the similar manner. Recall \eqref{eq-def-f-S}. Let $\ell\in\mathbb N$ be a parameter that will be decided later. In addition, define $\mathcal J_{\star}=\mathcal J_{\star}(\ell)$ to be the collection of unlabeled decorated paths with length $\ell$, and define (recall \eqref{eq-def-V-1-V-2-decorated})
\begin{equation}{\label{eq-def-mathcal-J}}
    \mathcal J=\mathcal J(\ell):= \Big\{ [H] \in \mathcal J_{\star}: \mathsf{L}(H) \subset V_{\bullet}(H) \Big\} \,.
\end{equation}
It is clear that if we label the vertices in $V(H)$ in the order $v_1,\ldots,v_{\ell+1}$ such that $(v_i,v_i+1) \in E(J)$, then $(v_1,v_2),(v_{\ell},v_{\ell+1}) \in E_{\bullet}(J)$ (see Figure~\ref{fig:decorated-path} for an illusttration). 
\begin{figure}[!ht]
    \centering
    \vspace{0cm}
    \includegraphics[height=4.24cm,width=13cm]{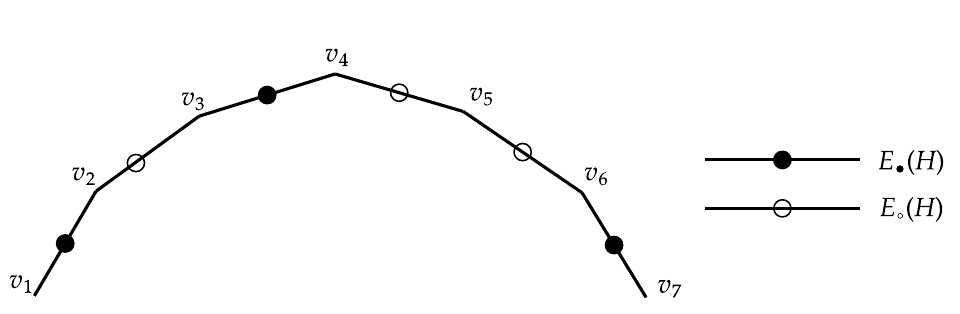}
    \caption{\noindent An unlabeled graph $[H] \in \mathcal J(\ell)$ with $\ell=6$. Here $\mathsf{diff}(H)=\{ v_2,v_3,v_4,v_6 \}$.}
    \label{fig:decorated-path}
\end{figure}
Define 
\begin{equation}{\label{eq-def-Phi-i,j-mathcal-J}}
    \Phi_{u,v}^{\mathcal J} := \frac{1}{ n^{\frac{\ell}{2}-1} \beta_{\mathcal J} } \sum_{[H] \in \mathcal J} \Xi(H) \sum_{ \substack{ S \subset \mathsf K_n: S \cong H \\ \mathsf L(S)=\{ u,v \} } } f_{S}(\bm X,\bm Y) \,. 
\end{equation}
for each $u,v \in [n]$. Here 
\begin{equation}{\label{eq-def-beta-mathcal-J}}
    \beta_{\mathcal J} = \sum_{[H] \in \mathcal J} \frac{ \lambda^{2|E_{\bullet}(H)|} \mu^{2|E_{\circ}(H)|} \rho^{2|\mathsf{diff}(H)|} }{ |\mathsf{Aut}(H)| } \overset{\eqref{eq-def-Xi-S}}{=} \sum_{[J] \in \mathcal J} \frac{ \Xi(H)^2 }{ |\mathsf{Aut}(H)| } \,.
\end{equation}
Our proposed method of recovery in correlated spiked Wigner model is as follows.
\begin{breakablealgorithm}{\label{alg:recovery-meta}}
\caption{Recovery in correlated spiked Wigner model}
    \begin{algorithmic}[1]
    \STATE {\bf Input:} Two symmetric matrices $\bm X,\bm Y$, a family $\mathcal J$ of non-isomorphic decorated graphs.
    \STATE For each pair $u,v \in [n]$, compute $\Phi_{u,v}^{\mathcal J}$ as in \eqref{eq-def-Phi-i,j-mathcal-J}. 
    \STATE Arbitrarily choose $w \in [n]$. Let $\widehat{x}_u = \Phi^{\mathcal J}_{w,u} \cdot \mathbf 1_{ \{ |\Phi^{\mathcal J}_{w,u}| \leq R^4 \} }$ for $1 \leq u \leq n$, where $R$ is given in \eqref{eq-L2-estimation-error}.
    \STATE {\bf Output:} $\widehat x$.
    \end{algorithmic}
\end{breakablealgorithm}
Similarly as in Lemma~\ref{lem-bound-beta-mathcal-H}, we need the following bound on $\beta_{\mathcal J}$.

\begin{lemma}{\label{lem-bound-beta-mathcal-J}}
    There exists a constant $D=D(\lambda,\mu,\rho)>1$ such that 
    \begin{equation}{\label{eq-bound-beta-mathcal-J}}
        D^{-1} A_+^{\ell} \leq \beta_{\mathcal J}, \beta_{\mathcal J_{\star}} \leq D A_+^{\ell} \,,
    \end{equation}
    where $A_+$ is defined in \eqref{eq-def-A-+}. In addition, suppose that $F(\lambda,\mu,\rho,\gamma)>1+\epsilon$ holds, then there exists a constant $0<\delta=\delta(\lambda,\mu,\rho)<0.1$ such that $A_+>(1+\delta)\gamma$. In particular, when \eqref{eq-prelim-assumption-upper-bound} is satisfied we have $A_+>1+\delta$ for some $\delta=\delta(\lambda,\mu,\rho) \in (0,0.1)$. 
\end{lemma}

The proof of Lemma~\ref{lem-bound-beta-mathcal-J} is incorporated in Section~\ref{subsec:proof-lem-2.10}. The key of our argument is the following result, which shows that $\Phi_{u,v}^{\mathcal J}$ is positively correlated with $x_u x_v$ and has bounded variance. 
\begin{proposition}{\label{main-prop-recovery}}
    Suppose that Assumption~\ref{assum-upper-bound} and Equation~\eqref{eq-prelim-assumption-upper-bound} hold and we choose $\ell$ such that 
    \begin{align}{\label{eq-condition-weak-recovery}}
        \ell=O_{\delta}(\log n), \quad (1+\delta)^{\ell} \geq n^2 \,.
    \end{align}
    Then there exists sufficiently large $R=O_{\lambda,\mu,\rho,\delta}(1)$ such that 
    \begin{align}
        \mathbb E_{\Pb}\Big[ \Phi^{\mathcal J}_{u,v} \cdot x_u x_v \Big] = 1+o(1) \mbox{ and }
        \mathbb E_{\Pb}\Big[ \big( \Phi^{\mathcal J}_{u,v} \big)^2 \Big], \ \mathbb E_{\Pb}\Big[ \big( \Phi^{\mathcal J}_{u,v} \big)^2 x_u^2 x_v^2 \Big] \leq R \,.  \label{eq-L2-estimation-error}
    \end{align}
\end{proposition}
Combining these variance bounds with Markov's inequality, we arrive at the following sufficient condition for Algorithm~\ref{alg:recovery-meta} to achieve weak recovery.
\begin{thm}{\label{MAIN-THM-recovery}}
    Suppose that Assumption~\ref{assum-upper-bound} and Equation~\ref{eq-prelim-assumption-upper-bound} hold, and we choose $\ell$ according to \eqref{eq-condition-weak-recovery}. Then we have (below we write $\widehat{x}$ to be the output of Algorithm~\ref{alg:recovery-meta})
    \begin{align*}
        \mathbb E_\Pb\Big[ \tfrac{ |\langle \widehat x, x \rangle| }{ \| \widehat x \| \| x \| } \Big] \geq \Omega(1) \,.
    \end{align*}
\end{thm}
Again, to resolve the computational issue of calculating $\Phi_{u,v}^{\mathcal J}$, in Section~\ref{subsec:approx-recovery}, we give a polynomial-time algorithm (see Algorithm~\ref{alg:cal-widetilde-Phi}) that computes an approximation $\widetilde{\Phi}_{u,v}^{\mathcal J}$ for $\Phi_{u,v}^{\mathcal J}$ using the strategy of color coding. The following result shows that the approximated similarity score $\widetilde{\Phi}_{u,v}^{\mathcal J}$ enjoys the same statistical guarantee under the same condition \eqref{eq-condition-weak-recovery} as Theorem~\ref{MAIN-THM-recovery}.
\begin{thm}{\label{MAIN-THM-recovery-algorithmic}}
    Proposition~\ref{main-prop-recovery} and Theorem~\ref{MAIN-THM-recovery} continues to hold with $\widetilde{\Phi}_{u,v}^{\mathcal J}$ in place of $\Phi_{u,v}^{\mathcal J}$. In addition, $\{ \widetilde{\Phi}_{u,v}^{\mathcal J}: u,v \in [n] \}$ can be computed in time $n^{T+o(1)}$ for some constant $T=T(\lambda,\mu,\rho)$.
\end{thm}

\subsubsection{The correlated spiked Wishart model}{\label{subsubsec:recovery-stat-cov}}

We only show how to estimate $x$ and estimating $y$ can be done in the similar manner. Recall \eqref{eq-def-f-S-cov}. Let $\ell\geq 0$ be a parameter that will be decided later. In addition, define $\mathcal I_{\star\star}=\mathcal I_{\star\star}(\ell)$ to be the collection of unlabeled, decorated, bipartite paths with length $2\ell-1$ or $2\ell$ such that $\mathsf{diff}(H) \subset V^{\mathsf a}(H)$ for all $H \in \mathcal I_{\star\star}$. In addition, define $\mathcal I_{\star}=\mathcal I_{\star}(\ell)$ be the collection of $[H] \in \mathcal I_{\star\star}(\ell)$ such that $\mathsf{L}(H) \subset V_{\bullet}(H) \cap V^{\mathsf a}(H)$. Finally, define (recall \eqref{eq-def-V-1-V-2-decorated})
\begin{equation}{\label{eq-def-mathcal-I}}
    \mathcal I=\mathcal I(\ell):= \Big\{ [H] \in \mathcal I_{\star}(\ell): \mathsf{L}(H) \subset V_{\bullet}(H) \Big\} \,.
\end{equation}
(see Figure~\ref{fig:decorated-bipartite-path} for an illustration.)
\begin{figure}[!ht]
    \centering
    \vspace{0cm}
    \includegraphics[height=7.15cm,width=12.58cm]{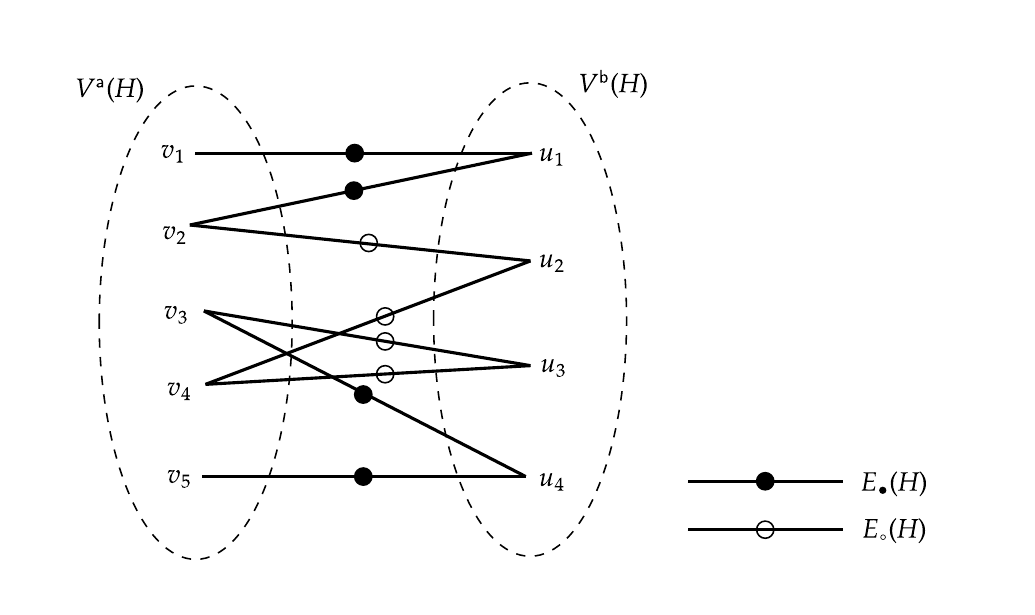}
    \caption{\noindent An unlabeled graph $[H] \in \mathcal I(\ell)$ with $\ell=4$. Here $\mathsf{diff}(H)=\{ v_2,v_4 \}$.}
    \label{fig:decorated-bipartite-path}
\end{figure}

\noindent Define 
\begin{equation}{\label{eq-def-Phi-i,j-mathcal-I}}
    \Psi_{u,v}^{\mathcal I} := \frac{1}{ N^{\ell} n^{-1} \beta_{\mathcal I} } \sum_{[H] \in \mathcal I} \Upsilon(H) \sum_{ \substack{ S \subset \mathsf K_{n,N}: S \cong H \\ \mathsf L(S)=\{ u,v \} } } h_{S}(\bm X,\bm Y) \,. 
\end{equation}
for each $u,v \in [n]$. Here 
\begin{equation}{\label{eq-def-beta-mathcal-I}}
    \beta_{\mathcal I} = \sum_{[H] \in \mathcal I} \frac{ \lambda^{|E_{\bullet}(H)|} \mu^{|E_{\circ}(H)|} \rho^{2|\mathsf{diff}(H)|} }{ |\mathsf{Aut}(H)| } \overset{\eqref{eq-def-Upsilon-S}}{=} \sum_{[H] \in \mathcal I} \frac{ \Upsilon(H)^2 }{ |\mathsf{Aut}(H)| } \,.
\end{equation}
Our proposed method of recovery in correlated spiked Wishart model is as follows.
\begin{breakablealgorithm}{\label{alg:recovery-cov-meta}}
\caption{Recovery in correlated spiked Wishart model}
    \begin{algorithmic}[1]
    \STATE {\bf Input:} Two matrices $\bm X,\bm Y$, a family $\mathcal I$ of non-isomorphic decorated, bipartite graphs.
    \STATE For each pair $u,v \in [n]$, compute $\Psi_{u,v}^{\mathcal I}$ as in \eqref{eq-def-Phi-i,j-mathcal-I}. 
    \STATE Arbitrarily choose $w \in [n]$. Let $\widehat{x}_u = \Psi^{\mathcal I}_{w,u} \cdot \mathbf 1_{ \{ |\Psi^{\mathcal I}_{w,u}| \leq R^4 \} }$ for $1 \leq u \leq n$, where $R$ is given in \eqref{eq-L2-estimation-error-cov}.
    \STATE {\bf Output:} $\widehat x$.
    \end{algorithmic}
\end{breakablealgorithm}
Similarly as in Lemma~\ref{lem-bound-beta-mathcal-J}, we need the following bound on $\beta_{\mathcal I}$.

\begin{lemma}{\label{lem-bound-beta-mathcal-I}}
    There exists a constant $D=D(\lambda,\mu,\rho,\gamma)>1$ such that 
    \begin{equation}{\label{eq-bound-beta-mathcal-I}}
        D^{-1} A_+^{\ell} \leq \beta_{\mathcal I}, \beta_{\mathcal I_{\star}}, \beta_{\mathcal I_{\star\star}} \leq D A_+^{\ell} \,,
    \end{equation}
    where $A_+$ is defined in \eqref{eq-def-A-+}. In addition, suppose that \eqref{eq-prelim-assumption-upper-bound-cov} holds, there exists a constant $0<\delta=\delta(\lambda,\mu,\rho,\gamma)<0.1$ such that $A_+>(1+\delta)\gamma$.
\end{lemma}

The proof of Lemma~\ref{lem-bound-beta-mathcal-I} is incorporated in Section~\ref{subsec:proof-lem-2.14}. The key of our argument is the following result, which shows that $\Psi_{u,v}^{\mathcal I}$ is positively correlated with $x_u x_v$ and has bounded variance. 

\begin{proposition}{\label{main-prop-recovery-cov}}
    Suppose that Assumption~\ref{assum-upper-bound} and Equation~\eqref{eq-prelim-assumption-upper-bound-cov} hold, and we choose $\ell$ according to \eqref{eq-condition-weak-recovery}. Then there exists sufficiently large $R=O_{\lambda,\mu,\rho,\gamma}(1)$ such that 
    \begin{align}
        \mathbb E_{\overline\Pb}\Big[ \Psi^{\mathcal I}_{u,v} \cdot x_u x_v \Big] = 1+o(1) \mbox{ and }
        \mathbb E_{\overline\Pb}\Big[ \big( \Psi^{\mathcal I}_{u,v} \big)^2 \Big], \ \mathbb E_{\overline\Pb}\Big[ \big( \Psi^{\mathcal I}_{u,v} \big)^2 x_u^2 x_v^2 \Big] \leq R \,.  \label{eq-L2-estimation-error-cov}
    \end{align}
\end{proposition}

Combining these variance bounds with Markov's inequality, we arrive at the following sufficient condition for Algorithm~\ref{alg:recovery-cov-meta} to achieve weak recovery.

\begin{thm}{\label{MAIN-THM-recovery-cov}}
    Suppose that Assumption~\ref{assum-upper-bound} and Equation~\ref{eq-prelim-assumption-upper-bound-cov} hold, and we choose $\ell$ according to \eqref{eq-condition-weak-recovery}. Then we have (below we write $\widehat{x}$ to be the output of Algorithm~\ref{alg:recovery-cov-meta})
    \begin{align*}
        \mathbb E_{\overline\Pb}\Big[ \tfrac{ |\langle \widehat x, x \rangle| }{ \| \widehat x \| \| x \| } \Big] \geq \Omega(1) \,.
    \end{align*}
\end{thm}

Again, to resolve the computational issue of calculating $\Psi_{u,v}^{\mathcal I}$, in Section~\ref{subsec:approx-recovery} we give a polynomial-time algorithm (see Algorithm~\ref{alg:cal-widetilde-Phi-mathcal-I}) that computes an approximation $\widetilde{\Psi}_{u,v}^{\mathcal I}$ for $\Psi_{u,v}^{\mathcal I}$ using the strategy of color coding. The following result shows that the approximated similarity score $\widetilde{\Psi}_{u,v}^{\mathcal I}$ enjoys the same statistical guarantee under the same condition \eqref{eq-condition-weak-recovery} as Theorem~\ref{MAIN-THM-recovery-cov}.

\begin{thm}{\label{MAIN-THM-recovery-algorithmic-cov}}
    Proposition~\ref{main-prop-recovery-cov} and Theorem~\ref{MAIN-THM-recovery-cov} continues to hold with $\widetilde{\Psi}_{u,v}^{\mathcal I}$ in place of $\Psi_{u,v}^{\mathcal I}$. In addition, $\{ \widetilde{\Psi}_{u,v}^{\mathcal I}: u,v \in [n] \}$ can be computed in time $n^{T+o(1)}$ for some constant $T=T(\lambda,\mu,\rho,\gamma)$.
\end{thm}

\section{Statistical analysis of the subgraph counts}{\label{sec:stat-analysis}}

\subsection{Proof of Proposition~\ref{main-prop-detection}}{\label{subsec:proof-prop-2.3}}

\begin{lemma}{\label{lem-mean-var-f-H-part-1}}
    Suppose that Assumption~\ref{assum-upper-bound} and Equation~\eqref{eq-prelim-assumption-upper-bound} hold, and we choose $\ell$ according to \eqref{eq-condition-strong-detection}. Then
    \begin{align}
        &\mathbb E_{\Qb}[ f_{\mathcal H} ] =0 \,, \label{eq-mean-Qb-f-H} \\
        &\mathbb E_{\Pb}[ f_{\mathcal H} ] =[1+o(1)] \sqrt{\beta_{\mathcal H}} = \omega(1) \,, \label{eq-mean-pb-f-H} \\
        &\mathbb E_{\Qb}[ f_{\mathcal H}^2 ] = 1+o(1) \,. \label{eq-var-Qb-f-H}
    \end{align}
\end{lemma}
\begin{proof}
    Recall \eqref{eq-def-f-mathcal-H}. Since $\mathbb E_{\Qb}[ f_S(\bm X,\bm Y) ]=0$ for all $S \subset \mathsf K_n$ and $E(S) \neq \emptyset$, we have $\mathbb E_{\Qb}[f_{\mathcal H}]=0$ by linearity. In addition, recall \eqref{eq-def-correlated-spike-specific}, we have
    \begin{align}
        \mathbb E_{\Pb}[f_S(\bm X,\bm Y)] &= \mathbb E_{x,y} \mathbb E_{\Pb}\Big[ \prod_{(i,j) \in E_{\bullet}(S)} \big( \tfrac{\lambda}{\sqrt{n}} x_i x_j + \bm W_{i,j} \big) \prod_{(i,j) \in E_{\circ}(S)} \big( \tfrac{\mu}{\sqrt{n}} y_i y_j + \bm Z_{i,j} \big) \mid x,y \Big] \nonumber \\
        &= \mathbb E_{x,y}\Big[ \prod_{(i,j) \in E_{\bullet}(S)} \big( \tfrac{\lambda}{\sqrt{n}} x_i x_j \big) \prod_{(i,j) \in E_{\circ}(S)} \big( \tfrac{\mu}{\sqrt{n}} y_i y_j \big) \Big] \nonumber \\
        &= \lambda^{|E_{\bullet}(S)|} \mu^{|E_{\circ}(S)|} n^{-\frac{|E(S)|}{2}} \mathbb E_{x,y}\Big[ \prod_{(i,j) \in E_{\bullet}(S)} \big( x_i x_j \big) \prod_{(i,j) \in E_{\circ}(S)} \big( y_i y_j \big) \Big] \nonumber \\
        &= \lambda^{|E_{\bullet}(S)|} \mu^{|E_{\circ}(S)|} \rho^{\mathsf{diff}(S)} n^{ -\frac{\ell}{2} } \overset{\eqref{eq-def-Xi-S}}{=} \Xi(S) n^{-\frac{\ell}{2}} \,,  \label{eq-mean-Pb-f-S}
    \end{align}
    where the last but one equality follows from $|E(S)|=\ell$ and (recall \eqref{eq-def-E-1-E-2-decorated}, \eqref{eq-def-V-1-V-2-decorated} and \eqref{eq-def-Dif(H)}) 
    \begin{align*}
        \prod_{(i,j) \in E_{\bullet}(S)} \big( x_i x_j \big) \prod_{(i,j) \in E_{\circ}(S)} \big( y_i y_j \big) = \prod_{i \in V_{\bullet}(S) \setminus V_{\circ}(S)} x_i^2 \prod_{i \in V_{\circ}(S) \setminus V_{\bullet}(S)} y_i^2 \prod_{ i \in \mathsf{diff}(S) } x_i y_i
    \end{align*}
    for all $[S] \in \mathcal H$. Thus, we have
    \begin{align*}
        \mathbb E_{\Pb}[f_{\mathcal H}] &\overset{\eqref{eq-def-f-mathcal-H}}{=} f_{\mathcal H}(\bm X,\bm Y) = \frac{1}{ \sqrt{\beta_{\mathcal H}} } \sum_{ [H] \in \mathcal H } \frac{ \Xi(H) }{ n^{\ell/2} } \sum_{ S \subset \mathsf{K}_n, S \cong H } \Xi(S) n^{ -\frac{\ell}{2} } \\
        &= \frac{ 1 }{ n^{\ell}\sqrt{\beta_{\mathcal H}} } \sum_{[H] \in \mathcal H} \Xi(H)^2 \cdot \#\big\{ S \cong \mathsf K_n: S \cong H \big\} \\
        & = [1+o(1)] \cdot \frac{1}{ \sqrt{\beta_{\mathcal H}} }  \sum_{[H] \in \mathcal H} \frac{ \Xi(H)^2 }{ |\mathsf{Aut}(H)| } \overset{\eqref{eq-def-beta-mathcal-H}}{=} \sqrt{\beta_{\mathcal H}} \,,
    \end{align*}
    where the third equality follows from
    \begin{equation}{\label{eq-enumerate-H-in-K-n}}
        \#\big\{ S \cong \mathsf K_n: S \cong H \big\} = \frac{ (1+o(1)) n^{|V(H)|} }{ |\mathsf{Aut}(H)| } \mbox{ for all } |V(H)|=n^{o(1)} \,.
    \end{equation}
    Also, from Lemma~\ref{lem-bound-beta-mathcal-H} and our choice of $\ell$ in \eqref{eq-condition-strong-detection} we see that $\sqrt{\beta_{\mathcal H}}=\omega(1)$. Finally, note that
    \begin{align}{\label{eq-standard-orthogonal}}
        \mathbb E_{\Qb}\big[ f_{S}(\bm X,\bm Y) f_{K}(\bm X,\bm Y) \big] = \mathbf 1_{ \{ S=K \} } \,,
    \end{align}
    we have
    \begin{align*}
        \mathbb E_{\Qb}[f_{\mathcal H}^2] &\overset{\eqref{eq-def-f-mathcal-H}}{=} \sum_{ [H],[I] \in \mathcal H } \frac{ \Xi(H)\Xi(I) }{ n^{\ell} \beta_{\mathcal H} } \sum_{ \substack{ S,K \subset \mathsf K_n \\ S \cong H, K \cong I } } \mathbf 1_{ \{ S=K \} } \\
        &= \sum_{ [H] \in \mathcal H } \frac{ \Xi(H)^2 }{ n^{\ell} \beta_{\mathcal H} } \#\big\{ S \cong \mathsf K_n: S \cong H \big\}  = \frac{ 1+o(1) }{ \beta_{\mathcal H} } \sum_{ [H] \in \mathcal H } \frac{ \Xi(H)^2 }{ |\mathsf{Aut}(H)| } \overset{\eqref{eq-def-beta-mathcal-H}}{=} 1+o(1) \,. \qedhere
    \end{align*}
\end{proof}

Now we bound the variance of $f_{\mathcal H}$ under the alternative hypothesis $\Pb$, as incorporated in the following lemma. 

\begin{lemma}{\label{lem-mean-var-f-H-part-2}}
    Assume that Assumption~\ref{assum-upper-bound} and Equation~\eqref{eq-prelim-assumption-upper-bound} hold, and we choose $\ell$ according to \eqref{eq-condition-strong-detection}. Then we have
    \begin{equation}{\label{eq-var-Pb-f-H}}
        \frac{ \operatorname{Var}_{\Pb}[f_{\mathcal H}] }{ \mathbb E_{\Pb}[f_{\mathcal H}]^2 } = 1+o(1) \,.
    \end{equation}
\end{lemma}

Assuming Lemma~\ref{lem-mean-var-f-H-part-2}, we can complete the proof of Proposition~\ref{main-prop-detection}.
\begin{proof}[Proof of Proposition~\ref{main-prop-detection} assuming Lemma~\ref{lem-mean-var-f-H-part-2}]
    It suffices to note that combining Lemmas~\ref{lem-mean-var-f-H-part-1} and \ref{lem-mean-var-f-H-part-2} yields Proposition~\ref{main-prop-detection}.
\end{proof}

The remaining part of this subsection is devoted to the proof of Lemma~\ref{lem-mean-var-f-H-part-2}. Note that
\begin{align}
    \frac{ \operatorname{Var}_{\Pb}[f_{\mathcal H}] }{ \mathbb E_{\Pb}[f_{\mathcal H}]^2 } \overset{\eqref{eq-def-f-mathcal-H},\text{Lemma~\ref{lem-mean-var-f-H-part-1}}}{=}\ & [1+o(1)] \sum_{[H],[I] \in \mathcal H} \frac{ \Xi(H) \Xi(I) }{ n^{\ell} \beta_{\mathcal H}^2 }  \sum_{ \substack{ S,K \subset \mathsf K_n \\ S \cong H, K \cong I } } \operatorname{Cov}_{\Pb}\big( f_S(\bm X,\bm Y), f_K(\bm X,\bm Y) \big) \,.  \label{eq-var-Pb-f-H-relax-1}
\end{align}
The first step of our proof is to show the following bound on $\operatorname{Cov}_{\Pb}(f_S(\bm X,\bm Y),f_K(\bm X,\bm Y))$. Define (recall \eqref{eq-def-H-1-H-2})
\begin{align}
    \mathtt M(S,K):=\ & \lambda^{ |E_{\bullet}(S) \triangle E_{\bullet}(K)| } \mu^{ |E_{\circ}(S) \triangle E_{\circ}(K)| } \rho^{|\mathsf{diff}(S) \setminus V(K)|+ |\mathsf{diff}(K) \setminus V(S)|} \nonumber  \\
    & \cdot C^{|V(S) \cap V(K)|-|V(S_{\bullet} \cap K_{\bullet}) \cup V(S_{\circ} \cap K_{\circ})|} \,.  \label{eq-def-mathtt-M}
\end{align}

\begin{lemma}{\label{lem-est-cov-f-S-f-K}}
    We have 
    \begin{align*}
        \operatorname{Cov}_{\Pb}\big( f_S,f_K \big) \leq\ & \mathbf 1_{ \{ V(S) \cap V(K) \neq \emptyset \} } \cdot n^{ -\ell+|E_{\bullet}(S)\cap E_{\bullet}(K)| +|E_{\circ}(S) \cap E_{\circ}(K)| } \cdot \mathtt M(S,K)  \,.
    \end{align*} 
\end{lemma}

The proof of Lemma~\ref{lem-est-cov-f-S-f-K} is postponed to Section~\ref{subsec:proof-lem-3.3} of the appendix. Using Lemma~\ref{lem-est-cov-f-S-f-K}, we see that the right hand side of \eqref{eq-var-Pb-f-H-relax-1} can be written as 
\begin{align}
    [1+o(1)] \sum_{ \substack{ S,K \subset \mathsf K_n: [S],[K] \in \mathcal H \\ V(S) \cap V(K) \neq \emptyset } } \frac{ \Xi(S) \Xi(K) \mathtt M(S,K) }{ n^{2\ell-|E_{\bullet}(S)\cap E_{\bullet}(K)|-|E_{\circ}(S) \cap E_{\circ}(K)|} \beta_{\mathcal H}^2 }  \,.  \label{eq-var-Pb-f-H-relax-2}
\end{align}
Now we split \eqref{eq-var-Pb-f-H-relax-2} into two parts, the first part counts the contribution from $S=K$ (in this case we have $|E_{\bullet}(S)\cap E_{\bullet}(K)|+|E_{\circ}(S) \cap E_{\circ}(K)|=|E(S)|=\ell$):
\begin{align}
    \sum_{ S \in \mathsf K_n: [S] \in \mathcal H } \frac{ \Xi(S)^2 \cdot \mathtt M(S,S) }{ n^{\ell} \beta_{\mathcal H}^2 }  \,.  \label{eq-var-Pb-f-H-relax-3-Part-1}
\end{align}
The second part counts the contribution from $V(S) \cap V(K) \neq \emptyset$ and $S \neq K$:
\begin{align}
    \sum_{ \substack{ S,K \subset \mathsf K_n: [S],[K] \in \mathcal H \\ V(S) \cap V(K) \neq \emptyset, S \neq K } } \frac{ \Xi(S) \Xi(K) \cdot \mathtt M(S,K) }{ n^{2\ell-|E_{\bullet}(S)\cap E_{\bullet}(K)|-|E_{\circ}(S) \cap E_{\circ}(K)|} \beta_{\mathcal H}^2 }  \,,  \label{eq-var-Pb-f-H-relax-3-Part-2}
\end{align}
We now bound \eqref{eq-var-Pb-f-H-relax-3-Part-1} and \eqref{eq-var-Pb-f-H-relax-3-Part-2} separately via the following lemma.

\begin{lemma}{\label{lem-detection-most-technical}}
    Suppose that Assumption~\ref{assum-upper-bound} and Equation~\eqref{eq-prelim-assumption-upper-bound} hold, and we choose $\ell$ according to \eqref{eq-condition-strong-detection}. Then we have
    \begin{align}
        & \eqref{eq-var-Pb-f-H-relax-3-Part-1} \leq \tfrac{1+o(1)}{\beta_{\mathcal H}} = o(1) \,.  \label{eq-bound-var-Pb-f-H-relax-3-Part-1} \\
        & \eqref{eq-var-Pb-f-H-relax-3-Part-2} \leq n^{-\frac{1}{2}+o(1)} = o(1) \,.  \label{eq-bound-var-Pb-f-H-relax-3-Part-2}
    \end{align}
\end{lemma}

The proof of Lemma~\ref{lem-detection-most-technical} is postponed to Section~\ref{subsec:proof-lem-3.4} in the appendix. Clearly, plugging \eqref{eq-bound-var-Pb-f-H-relax-3-Part-1} and \eqref{eq-bound-var-Pb-f-H-relax-3-Part-2} into \eqref{eq-var-Pb-f-H-relax-2} yields Lemma~\ref{lem-mean-var-f-H-part-2}.

\subsection{Proof of Proposition~\ref{main-prop-detection-cov}}{\label{subsec:proof-prop-2.7}}

\begin{lemma}{\label{lem-mean-var-f-mathcal-G-part-1}}
    Suppose that Assumption~\ref{assum-upper-bound} and Equation~\eqref{eq-condition-strong-detection} hold, and we choose $\ell$ according to \eqref{eq-condition-strong-detection}. Then
    \begin{align}
        &\mathbb E_{\overline\Qb}[ h_{\mathcal G} ] =0 \,, \label{eq-mean-Qb-f-mathcal-G} \\
        &\mathbb E_{\overline\Pb}[ h_{\mathcal G} ] =[1+o(1)] \sqrt{\gamma^{-\ell}\beta_{\mathcal G}} = \omega(1) \,, \label{eq-mean-pb-f-mathcal-G} \\
        &\mathbb E_{\overline\Qb}[ h_{\mathcal G}^2 ] = 1+o(1) \,. \label{eq-var-Qb-f-mathcal-G}
    \end{align}
\end{lemma}
\begin{proof}
    Recall \eqref{eq-def-f-mathcal-G}. Since $\mathbb E_{\overline\Qb}[ h_S(\bm X,\bm Y) ]=0$ for all $S \subset \mathsf K_{n,N}$ and $E(S) \neq \emptyset$, we have $\mathbb E_{\overline\Qb}[h_{\mathcal G}]=0$ by linearity. In addition, recall \eqref{eq-def-correlated-spike-covariance-specific}, we have
    \begin{align}
        \mathbb E_{\overline\Pb}[h_S(\bm X,\bm Y)] &= \mathbb E_{x,y,\bm u,\bm v} \mathbb E_{\overline\Pb}\Big[ \prod_{(i,j) \in E_{\bullet}(S)} \big( \tfrac{\sqrt{\lambda}}{\sqrt{n}} x_i \bm{u}_j + \bm W_{i,j} \big) \prod_{(i,j) \in E_{\circ}(S)} \big( \tfrac{\sqrt{\mu}}{\sqrt{n}} y_i \bm{v}_j + \bm Z_{i,j} \big) \mid x,y,\bm u,\bm v \Big] \nonumber \\
        &= \mathbb E_{x,y,\bm u,\bm v}\Big[ \prod_{(i,j) \in E_{\bullet}(S)} \big( \tfrac{\sqrt{\lambda}}{\sqrt{n}} x_i \bm{u}_j \big) \prod_{(i,j) \in E_{\circ}(S)} \big( \tfrac{\sqrt{\mu}}{\sqrt{n}} y_i \bm{v}_j \big) \Big] \nonumber \\
        &= \lambda^{\frac{|E_{\bullet}(S)|}{2}} \mu^{\frac{|E_{\circ}(S)|}{2}} n^{-\frac{|E(S)|}{2}} \mathbb E_{x,y,\bm u,\bm v}\Big[ \prod_{(i,j) \in E_{\bullet}(S)} \big( x_i \bm{u}_j \big) \prod_{(i,j) \in E_{\circ}(S)} \big( y_i \bm{v}_j \big) \Big] \nonumber \\
        &= \lambda^{\frac{|E_{\bullet}(S)|}{2}} \mu^{\frac{|E_{\circ}(S)|}{2}} \rho^{\mathsf{diff}(S)} n^{ -\ell } \overset{\eqref{eq-def-Xi-S}}{=} \Upsilon(S) n^{-\ell} \,,  \label{eq-mean-Pb-f-S-cov}
    \end{align}
    where the last but one equality follows from $|E(S)|=2\ell$ and (recall $\mathsf{diff}(S)\subset V^{\mathsf a}(S)$)
    \begin{align*}
        \prod_{(i,j) \in E_{\bullet}(S)} \big( x_i \bm u_j \big) \prod_{(i,j) \in E_{\circ}(S)} \big( y_i \bm v_j \big) =\ & \prod_{i \in V^{\mathsf a}_{\bullet}(S) \setminus V^{\mathsf a}_{\circ}(S)} x_i^2 \prod_{i \in V^{\mathsf a}_{\circ}(S) \setminus V^{\mathsf a}_{\bullet}(S)} y_i^2 \prod_{ i \in \mathsf{diff}(S) } x_i y_i \cdot \\
        & \prod_{j \in V^{\mathsf b}_{\bullet}(S) \setminus V^{\mathsf b}_{\circ}(S)} \bm u_j^2 \prod_{j \in V^{\mathsf b}_{\circ}(S) \setminus V^{\mathsf b}_{\bullet}(S)} \bm v_j^2
    \end{align*}
    for all $[S] \in \mathcal G$. Thus, we have
    \begin{align*}
        \mathbb E_{\overline\Pb}\big[ h_{\mathcal G}(\bm X,\bm Y) \big] &\overset{\eqref{eq-def-f-mathcal-G}}{=} \frac{1}{ \sqrt{\beta_{\mathcal G}} } \sum_{ [H] \in \mathcal H } \frac{ \Upsilon(H) }{ n^{\ell/2} N^{\ell/2} } \sum_{ S \subset \mathsf{K}_n, S \cong H } \Xi(S) n^{ -\ell } \\
        &= \frac{ 1 }{ n^{3\ell/2}N^{\ell/2} \sqrt{\beta_{\mathcal H}} } \sum_{[H] \in \mathcal H} \Upsilon(H)^2 \cdot \#\big\{ S \cong \mathsf K_n: S \cong H \big\} \\
        & = [1+o(1)] \cdot \frac{ 1 }{ \sqrt{\gamma^{\ell}\beta_{\mathcal H}} }  \sum_{[H] \in \mathcal H} \frac{ \Upsilon(H)^2 }{ |\mathsf{Aut}(H)| } \overset{\eqref{eq-def-beta-mathcal-G}}{=}  \sqrt{\gamma^{-\ell} \beta_{\mathcal G}} \,,
    \end{align*}
    where the third equality follows from $N=\gamma n$ and
    \begin{equation}{\label{eq-enumerate-H-in-K-N,n}}
        \#\big\{ S \cong \mathsf K_{n,N}: S \cong H \big\} = \frac{ (1+o(1)) n^{\ell} N^{\ell} }{ |\mathsf{Aut}(H)| } \mbox{ for all } [H] \in \mathcal G \,.
    \end{equation}
    Also, from Lemma~\ref{lem-bound-beta-mathcal-G} and our choice of $\ell$ in \eqref{eq-condition-strong-detection} we see that $\sqrt{\gamma^{-\ell}\beta_{\mathcal G}}=\omega(1)$. Finally, note that
    \begin{align}{\label{eq-standard-orthogonal-cov}}
        \mathbb E_{\overline\Qb}\big[ h_{S}(\bm X,\bm Y) h_{K}(\bm X,\bm Y) \big] = \mathbf 1_{ \{ S=K \} } \,,
    \end{align}
    we have
    \begin{align*}
        \mathbb E_{\overline\Qb}[h_{\mathcal G}^2] &\overset{\eqref{eq-def-f-mathcal-G}}{=} \sum_{ [H],[I] \in \mathcal G } \frac{ \Upsilon(H)\Upsilon(I) }{ n^{\ell} N^{\ell} \beta_{\mathcal G} } \sum_{ \substack{ S,K \subset \mathsf K_n \\ S \cong H, K \cong I } } \mathbf 1_{ \{ S=K \} } \\
        &= \sum_{ [H] \in \mathcal G } \frac{ \Upsilon(H)^2 }{ n^{\ell} N^{\ell} \beta_{\mathcal G} } \#\big\{ S \cong \mathsf K_{n,N}: S \cong H \big\}  = \frac{ 1+o(1) }{ \beta_{\mathcal G} } \sum_{ [H] \in \mathcal G } \frac{ \Upsilon(H)^2 }{ |\mathsf{Aut}(H)| } \overset{\eqref{eq-def-beta-mathcal-G}}{=} 1+o(1) \,. \qedhere
    \end{align*}
\end{proof}

Now we bound the variance of $h_{\mathcal G}$ under the alternative hypothesis $\overline\Pb$, as incorporated in the following lemma. 

\begin{lemma}{\label{lem-mean-var-f-mathcal-G-part-2}}
    Assume that Assumption~\ref{assum-upper-bound} and Equation~\eqref{eq-prelim-assumption-upper-bound} hold and we choose $\ell$ according to \eqref{eq-condition-strong-detection}. Then we have
    \begin{equation}{\label{eq-var-Pb-f-mathcal-G}}
        \frac{ \operatorname{Var}_{\overline\Pb}[h_{\mathcal G}] }{ \mathbb E_{\overline\Pb}[h_{\mathcal G}]^2 } = 1+o(1) \,.
    \end{equation}
\end{lemma}

Assuming Lemma~\ref{lem-mean-var-f-mathcal-G-part-2}, we can complete the proof of Proposition~\ref{main-prop-detection-cov}.

\begin{proof}[Proof of Proposition~\ref{main-prop-detection-cov} assuming Lemma~\ref{lem-mean-var-f-mathcal-G-part-2}]
    It suffices to note that combining Lemmas~\ref{lem-mean-var-f-mathcal-G-part-1} and \ref{lem-mean-var-f-mathcal-G-part-2} yields Proposition~\ref{main-prop-detection-cov}.
\end{proof}

The remaining part of this subsection is devoted to the proof of Lemma~\ref{lem-mean-var-f-mathcal-G-part-2}. Note that
\begin{align}
    \frac{ \operatorname{Var}_{\overline\Pb}[h_{\mathcal G}] }{ \mathbb E_{\overline\Pb}[h_{\mathcal G}]^2 } \overset{\eqref{eq-def-f-mathcal-G},\text{Lemma~\ref{lem-mean-var-f-mathcal-G-part-1}}}{=}\ & [1+o(1)] \sum_{[H],[I] \in \mathcal G} \frac{ \Upsilon(H)\Upsilon(I) }{ \gamma^{-\ell} n^{\ell} N^{\ell} \beta_{\mathcal G}^2 }  \sum_{ \substack{ S,K \subset \mathsf K_{n,N} \\ S \cong H, K \cong I } } \operatorname{Cov}_{\overline\Pb}\big( h_S(\bm X,\bm Y), h_K(\bm X,\bm Y) \big) \,.  \label{eq-var-Pb-f-mathcal-G-relax-1}
\end{align}
The first step of our proof is to show the following bound on $\operatorname{Cov}_{\overline\Pb}(h_S(\bm X,\bm Y),h_K(\bm X,\bm Y))$. Define (recall \eqref{eq-def-H-1-H-2})
\begin{align}
    \mathtt P(S,K):=\ & \lambda^{ \frac{1}{2} |E_{\bullet}(S) \triangle E_{\bullet}(K)| } \mu^{ \frac{1}{2} |E_{\circ}(S) \triangle E_{\circ}(K)| } \rho^{|\mathsf{diff}(S) \setminus V(K)|+ |\mathsf{diff}(K) \setminus V(S)|} \nonumber  \\
    & \cdot (2C)^{|V(S) \cap V(K)|-|V(S_{\bullet} \cap K_{\bullet}) \cup V(S_{\circ} \cap K_{\circ})|} \,.  \label{eq-def-mathtt-P}
\end{align}

\begin{lemma}{\label{lem-est-cov-f-S-f-K-cov}}
    We have 
    \begin{align*}
        \operatorname{Cov}_{\overline\Pb}\big( h_S,h_K \big) \leq\ & \mathbf 1_{ \{ V(S) \cap V(K) \neq \emptyset \} } \cdot n^{ -2\ell+|E_{\bullet}(S)\cap E_{\bullet}(K)| +|E_{\circ}(S) \cap E_{\circ}(K)| } \cdot \mathtt P(S,K)  \,.
    \end{align*} 
\end{lemma}

The proof of Lemma~\ref{lem-est-cov-f-S-f-K-cov} is postponed to Section~\ref{subsec:proof-lem-3.7} of the appendix. Using Lemma~\ref{lem-est-cov-f-S-f-K-cov}, we see that the right hand side of \eqref{eq-var-Pb-f-mathcal-G-relax-1} can be written as 
\begin{align}
    [1+o(1)] \sum_{ \substack{ S,K \subset \mathsf K_{n,N}: [S],[K] \in \mathcal G \\ V(S) \cap V(K) \neq \emptyset } } \frac{ \Upsilon(S) \Upsilon(K) \mathtt P(S,K) }{ \gamma^{-\ell} N^{\ell} n^{3\ell-|E_{\bullet}(S)\cap E_{\bullet}(K)|-|E_{\circ}(S) \cap E_{\circ}(K)|} \beta_{\mathcal G}^2 }  \,.  \label{eq-var-Pb-f-mathcal-G-relax-2}
\end{align}
Now we split \eqref{eq-var-Pb-f-mathcal-G-relax-2} into two parts, the first part counts the contribution from $S=K$ (in this case we have $|E_{\bullet}(S)\cap E_{\bullet}(K)|+|E_{\circ}(S) \cap E_{\circ}(K)|=|E(S)|=2\ell$):
\begin{align}
    \sum_{ S \in \mathsf K_{n,N}: [S] \in \mathcal G } \frac{ \Upsilon(S)^2 \cdot \mathtt P(S,S) }{ \gamma^{-\ell} N^{\ell} n^{\ell} \beta_{\mathcal G}^2 }  \,.  \label{eq-var-Pb-f-mathcal-G-relax-3-Part-1}
\end{align}
The second part counts the contribution from $V(S) \cap V(K) \neq \emptyset$ and $S \neq K$:
\begin{align}
    \sum_{ \substack{ S,K \subset \mathsf K_{n,N}: [S],[K] \in \mathcal G \\ V(S) \cap V(K) \neq \emptyset, S \neq K } } \frac{ \Upsilon(S) \Upsilon(K) \cdot \mathtt P(S,K) }{ \gamma^{-\ell} N^{\ell} n^{3\ell-|E_{\bullet}(S)\cap E_{\bullet}(K)|-|E_{\circ}(S) \cap E_{\circ}(K)|} \beta_{\mathcal G}^2 }  \,,  \label{eq-var-Pb-f-mathcal-G-relax-3-Part-2}
\end{align}
We now bound \eqref{eq-var-Pb-f-mathcal-G-relax-3-Part-1} and \eqref{eq-var-Pb-f-mathcal-G-relax-3-Part-2} separately via the following lemma.

\begin{lemma}{\label{lem-detection-cov-most-technical}}
    Suppose that Assumption~\ref{assum-upper-bound} and Equation~\eqref{eq-prelim-assumption-upper-bound-cov} hold, and we choose $\ell$ according to \eqref{eq-condition-strong-detection}. Then we have
    \begin{align}
        & \eqref{eq-var-Pb-f-mathcal-G-relax-3-Part-1} \leq \tfrac{1+o(1)}{\gamma^{-\ell}\beta_{\mathcal G}} = o(1) \,.  \label{eq-bound-var-Pb-f-mathcal-G-relax-3-Part-1} \\
        & \eqref{eq-var-Pb-f-mathcal-G-relax-3-Part-2} \leq n^{-\frac{1}{2}+o(1)} = o(1) \,.  \label{eq-bound-var-Pb-f-mathcal-G-relax-3-Part-2}
    \end{align}
\end{lemma}

The proof of Lemma~\ref{lem-detection-cov-most-technical} in postponed to Section~\ref{subsec:proof-lem-3.8} of the appendix. Clearly, plugging \eqref{eq-bound-var-Pb-f-mathcal-G-relax-3-Part-1} and \eqref{eq-bound-var-Pb-f-mathcal-G-relax-3-Part-2} into \eqref{eq-var-Pb-f-mathcal-G-relax-2} yields Lemma~\ref{lem-mean-var-f-mathcal-G-part-2}.

\subsection{Proof of Theorems~\ref{MAIN-THM-detection} and \ref{MAIN-THM-detection-cov}}{\label{subsec:proof-thm-2.4}}

With Propositions~\ref{main-prop-detection} and \ref{main-prop-detection-cov}, we are ready to prove Theorems~\ref{MAIN-THM-detection} and \ref{MAIN-THM-detection-cov}.
\begin{proof}[Proof of Theorems~\ref{MAIN-THM-detection} and \ref{MAIN-THM-detection-cov}]
    Under condition \eqref{eq-condition-strong-detection}, from Proposition~\ref{main-prop-detection} we get that
    \begin{align*}
        \mathbb E_{\Qb}[f_{\mathcal H}]=0, \ \mathbb E_{\Qb}[f_{\mathcal H}^2] = 1+o(1), \ \mathbb E_{\Pb}[f_{\mathcal H}] = \omega(1), \ \operatorname{Var}_{\Pb}[f_{\mathcal H}] = o(1) \cdot \mathbb E_{\Pb}[f_{\mathcal H}]^2 \,.
    \end{align*}
    Thus, for any constant $0<c<1$ we obtain
    \begin{align*}
        &\Pb\Big( f_{\mathcal H}(\bm Y) \leq c\cdot \mathbb E_{\Pb}[f_{\mathcal H}] \Big) \leq \frac{ \operatorname{Var}_{\Pb}[f_{\mathcal H}] }{ (1-c)^2 \mathbb E_{\Pb}[f_{\mathcal H}]^2 } = o(1) \,, \\
        &\Qb\Big( f_{\mathcal H}(\bm Y) \geq c\cdot \mathbb E_{\Pb}[f_{\mathcal H}] \Big) \leq \frac{ \operatorname{Var}_{\Qb}[f_{\mathcal H}] }{ c^2 \mathbb E_{\Pb}[f_{\mathcal H}]^2 } = o(1) \,. 
    \end{align*}
    We can derive Theorem~\ref{MAIN-THM-detection-cov} from Proposition~\ref{main-prop-detection-cov} in the same manner.
\end{proof}

\subsection{Proof of Proposition~\ref{main-prop-recovery}}{\label{subsec:proof-prop-2.10}}

\begin{lemma}{\label{lem-conditional-exp-given-endpoints}}
    We have 
    \begin{align*}
        \mathbb E_{\Pb}\big[ f_S(\bm X,\bm Y) \mid x_u,x_v \big] = n^{-\ell/2} \Xi(S) x_u x_v
    \end{align*}
    for all $[S] \in \mathcal J$ and $\mathsf L(S)=\{ u,v \}$.
\end{lemma}
\begin{proof}
    Denote $\pi_{u,v}$ be the law of $(x,y)$ given $x_u,x_v$. We then have
    \begin{align*}
        \mathbb E_{\Pb}\big[ f_S(\bm X,\bm Y) \mid x_u,x_v \big] = \mathbb E_{(x,y)\sim\pi_{u,v}}\Big[ \mathbb E\big[ f_S(\bm X,\bm Y) \mid x,y \big] \Big] \,.
    \end{align*}
    Recall \eqref{eq-def-f-S}, we have
    \begin{align*}
        \mathbb E\big[ f_S(\bm X,\bm Y) \mid x,y \big] &= \mathbb E\Big[ \prod_{(i,j) \in E_{\bullet}(S)} \big( \tfrac{\lambda}{\sqrt{n}} x_i x_j + \bm W_{i,j} \big) \prod_{(i,j) \in E_{\circ}(S)} \big( \tfrac{\mu}{\sqrt{n}} y_i y_j + \bm Z_{i,j} \big) \mid x,y \Big] \\
        &= n^{-\ell/2} \lambda^{|E_{\bullet}(S)|} \mu^{|E_{\circ}(S)|} \prod_{(i,j) \in E_{\bullet}(S)} \big( x_i x_j \big) \prod_{(i,j) \in E_{\circ}(S)} \big( y_i y_j \big) \,.
    \end{align*}
    In addition, recall \eqref{eq-def-V-1-V-2-decorated} and recall that for all $[S] \in \mathcal J$ with $\mathsf L(S)=\{ u,v \}$, we have $\{ u,v \} \in V_{\bullet}(S)$. Thus,
    \begin{align*}
        & \prod_{(i,j) \in E_{\bullet}(S)} \big( x_i x_j \big) \prod_{(i,j) \in E_{\circ}(S)} \big( y_i y_j \big) \\
        =\ & x_u x_v \cdot \prod_{ i \in V(S) \setminus \{ u,v \} } \big( x_i^2 \mathbf 1_{ \{ i \in V_{\bullet}(S) \setminus \mathsf{diff}(S) \} } + y_i^2 \mathbf 1_{ \{ i \in V_{\circ}(S) \setminus \mathsf{diff}(S) \} } + x_i y_i \mathbf 1_{ \{ i \in \mathsf{diff}(S) \} } \big) \,.
    \end{align*}
    This yields that
    \begin{align*}
        \mathbb E_{(x,y) \sim \pi_{u,v}}\Big[ \prod_{(i,j) \in E_{\bullet}(S)} \big( x_i x_j \big) \prod_{(i,j) \in E_{\circ}(S)} \big( y_i y_j \big) \Big] = \rho^{|\mathsf{diff}(S)|} x_u x_v \,, 
    \end{align*}
    leading to the desired result (recall \eqref{eq-def-Xi-S}).
\end{proof}

Based on Lemma~\ref{lem-conditional-exp-given-endpoints}, we see that 
\begin{align*}
    \mathbb E_{\Pb}[f_S(\bm Y)]=0 \mbox{ for all } [S] \in \mathcal J, \mathsf L(S)=\{ u,v \} \Longrightarrow \mathbb E_{\Pb}\big[ \Phi_{u,v}^{\mathcal J} \big]=0 \,.
\end{align*}
In addition, we have
\begin{align}
    &\mathbb E_{\Pb}\big[ \Phi^{\mathcal J}_{u,v} \mid x_u,x_v \big] \overset{\eqref{eq-def-Phi-i,j-mathcal-J}}{=} \frac{ 1 }{ n^{\frac{\ell}{2}-1} \beta_{\mathcal J}  } \sum_{[H] \in \mathcal J} \sum_{ \substack{ S \subset \mathsf K_n: S \cong H \\ \mathsf L(S)=\{ u,v \} } } \Xi(H)^2 n^{-\ell/2} x_u x_v \nonumber \\
    =\ & \frac{ x_u x_v }{ n^{\ell-1} \beta_{\mathcal J} } \sum_{[H] \in \mathcal J} \Xi(H)^2 x_u x_v \cdot \#\big\{ S \subset \mathsf K_n: S \cong H, \mathsf L(S)=\{ u,v \} \big\} \nonumber \\
    =\ & \frac{ x_u x_v }{ n^{\ell-1} \beta_{\mathcal J} } \sum_{[H] \in \mathcal J} \Xi(H)^2 \cdot \frac{ n^{\ell-1} }{ |\mathsf{Aut}(H)| }  \overset{\eqref{eq-def-beta-mathcal-J}}{=} (1+o(1)) x_u x_v \,.  \label{eq-conditional-mean-Phi-i,j}
\end{align}
We now bound the variance of $\Phi_{i,j}^{\mathcal J}$ via the following lemma.
\begin{lemma}{\label{lem-conditional-var-Phi-i,j}}
    Suppose that Assumption~\ref{assum-upper-bound} and Equation~\ref{eq-prelim-assumption-upper-bound} hold and we choose $\ell$ according to \eqref{eq-condition-weak-recovery}. Then
    \begin{equation}{\label{eq-conditional-var-Phi-i,j}}
        \mathbb E_{\Pb}\Big[ \big( \Phi_{u,v}^{\mathcal J} \big)^2 \Big], \ \mathbb E_{\Pb}\Big[ \big( \Phi_{u,v}^{\mathcal J} \big)^2 x_u^2 x_v^2 \Big] \leq O_{\lambda,\mu,\rho,\delta}(1)  \,.
    \end{equation}
\end{lemma}

Based on Lemma~\ref{lem-conditional-var-Phi-i,j}, we can now finish the proof of Proposition~\ref{main-prop-recovery}.
\begin{proof}[Proof of Proposition~\ref{main-prop-recovery} assuming Lemma~\ref{lem-conditional-var-Phi-i,j}]
    Note that \eqref{eq-conditional-mean-Phi-i,j} immediately implies that $\mathbb E_{\Pb}[ \Phi^{\mathcal J}_{i,j} \cdot x_i x_j] = 1+o(1)$. Combined with Lemma~\ref{lem-conditional-var-Phi-i,j}, we have shown Proposition~\ref{main-prop-recovery}.
\end{proof}

The rest part of this subsection is devoted to the proof of Lemma~\ref{lem-conditional-var-Phi-i,j}. We will only show how bound $\mathbb E_{\Pb}[ (\Phi_{u,v}^{\mathcal J})^2 x_u^2 x_v^2 ]$ and $\mathbb E_{\Pb}[ (\Phi_{u,v}^{\mathcal J})^2 ]$ can be bounded in the same manner. Recall \eqref{eq-def-Phi-i,j-mathcal-J}, we have
\begin{align}
    \mathbb E_{\Pb}\Big[ \big( \Phi_{u,v}^{\mathcal J} \big)^2 x_u^2 x_v^2 \Big] = \sum_{ \substack{ [S],[K] \in \mathcal J \\ \mathsf L(S)=\mathsf L(K)=\{ i,j \} } } \frac{ \Xi(S) \Xi(K) }{ n^{\ell-2} \beta_{\mathcal J}^2 } \mathbb E_{\Pb}\big[ f_S f_K \cdot x_u^2 x_v^2 \big] \,.  \label{eq-var-Phi-i,j-relax-1}
\end{align}

\begin{lemma}{\label{lem-est-cov-f-S-f-K-chain-case}}
    Recall \eqref{eq-def-E-1-E-2-decorated} and \eqref{eq-def-mathtt-M}. Suppose that Assumption~\ref{assum-upper-bound} holds, we have 
    \begin{align*}
        \mathbb E_{\Pb}\big[ f_S f_K \cdot x_u^2 x_v^2 \big] \leq C^2 n^{ -\ell+|E_{\bullet}(S)\cap E_{\bullet}(K)| +|E_{\circ}(S) \cap E_{\circ}(K)| } \mathtt M(S,K)  \,.
    \end{align*} 
\end{lemma}

The proof of Lemma~\ref{lem-est-cov-f-S-f-K-chain-case} is postponed to Section~\ref{subsec:proof-lem-3.11} of the appendix. To this end, plugging Lemma~\ref{lem-est-cov-f-S-f-K-chain-case} into \eqref{eq-var-Phi-i,j-relax-2}, we have
\begin{align}
    \eqref{eq-var-Phi-i,j-relax-1} &\leq \frac{ C^2 }{ n^{\ell-2} \beta_{\mathcal J}^2 } \sum_{ \substack{ [S], [K] \in \mathcal J \\ \mathsf L(S)=\mathsf L(K)=\{ i,j \} } } n^{ -\ell+|E_{\bullet}(S)\cap E_{\bullet}(K)| +|E_{\circ}(S) \cap E_{\circ}(K)| } \Xi(S)\Xi(K) \mathtt M(S,K) \nonumber  \\
    &= \frac{ C^2 }{ n^{2\ell-2} \beta_{\mathcal J}^2 } \sum_{ \substack{ [S], [K] \in \mathcal J \\ \mathsf L(S)=\mathsf L(K)=\{ i,j \} } } n^{ |E_{\bullet}(S)\cap E_{\bullet}(K)| +|E_{\circ}(S) \cap E_{\circ}(K)| } \Xi(S)\Xi(K) \mathtt M(S,K) \,.  \label{eq-var-Phi-i,j-relax-2}
\end{align}
We now introduce some notations. For $S,K \subset \mathsf K_n$ with $[S],[K] \in \mathcal J(\ell)$, denote 
\begin{equation}{\label{eq-decomposition-S}}
    V(S) = \big\{ v_1,\ldots,v_{\ell+1} \big\}, \ E(S)=\{ (v_1,v_2), \ldots, (v_{\ell},v_{\ell+1}) \} \mbox{ with } v_{1}=u, v_{\ell+1}=v 
\end{equation}
and 
\begin{equation}{\label{eq-decomposition-K}}
    V(K) = \big\{ v_1',\ldots,v_{\ell+1}' \big\}, \ E(K)=\{ (v_1',v_2'), \ldots, (v_{\ell}',v_{\ell+1}') \} \mbox{ with } v_{1}'=u, v_{\ell+1}'=v \,.
\end{equation}
In addition, define $v_{[\mathtt a,\mathtt b]}=\{ v_{\mathtt a}, \ldots, v_{\mathtt b} \}$ and define $S_{[\mathtt a,\mathtt b]}$ to be the subgraph of $S$ induced by $v_{[\mathtt a,\mathtt b]}$. Define $v'_{[\mathtt a,\mathtt b]}$ and $K_{[\mathtt a,\mathtt b]}$ in the similar manner. In addition, given two paths $(S,K)$ with length $\ell$ and $\mathsf L(S)=\mathsf L(K)$, we can uniquely determine a sequence $\mathsf{SEQ}(S,K)=(\mathtt p_1,\mathtt q_1,\ldots,\mathtt p_{\mathtt T}, \mathtt q_{\mathtt T})$ such that
\begin{equation}{\label{eq-def-SEQ(S,K)}}
    \begin{aligned}
        & 1 = \mathtt p_1 \leq \mathtt q_1 < \mathtt p_2 \leq \mathtt q_2 \ldots < \mathtt p_{\mathtt T} \leq \mathtt q_{\mathtt T} = \ell+1 \,, \\
        & V(S) \cap V(K) = v_{ [\mathtt p_1,\mathtt q_1] } \cup \ldots \cup v_{ [\mathtt p_{\mathtt T},\mathtt q_{\mathtt T}] } \,.
    \end{aligned}
\end{equation}
(if $V(S)=V(K)$ we simply let $\mathtt T=1$). Then, it is straightforward to check that there exists a unique sequence $\overline{\mathsf{SEQ}}(S,K)=(\mathtt p_1',\mathtt q_1',\ldots,\mathtt p_{\mathtt T}', \mathtt q_{\mathtt T}')$ such that
\begin{equation}{\label{eq-def-overline-SEQ(S,K)}}
    \begin{aligned}
        & \mathtt p_1'=1, \mathtt q_1'= \mathtt q_1, \mathtt p_{\mathtt T}'=\mathtt p_{\mathtt T}, \mathtt q_{\mathtt T}'=\ell+1 \,, \\
        & \{ \mathtt q_1' - \mathtt p_1', \ldots, \mathtt q_{\mathtt T}'-\mathtt p_{\mathtt T}' \} = \{ \mathtt q_1 - \mathtt p_1, \ldots, \mathtt q_{\mathtt T}-p_{\mathtt T} \} \,,   \\
        & V(S) \cap V(K) = v_{ [\mathtt p_1',\mathtt q_1'] }' \cup \ldots \cup v_{ [\mathtt p_{\mathtt T}',\mathtt q_{\mathtt T}'] }' \,.
    \end{aligned}
\end{equation} 
Finally, we define 
\begin{equation}{\label{eq-def-sigma(S,K)}}
    \begin{aligned}
        & \sigma=\sigma(S,K) \in \{ 0,1 \}^{ [\mathtt p_1,\mathtt q_1-1] \cup \ldots [\mathtt p_{\mathtt T},\mathtt q_{\mathtt T}-1] }: \\
        & \sigma(\mathtt i)= \mathbf 1_{ \{ (v_{\mathtt i},v_{\mathtt i+1}) \in E_{\bullet}(S) \cap E_{\bullet}(K) \} } + \mathbf 1_{ \{ (v_{\mathtt i},v_{\mathtt i+1}) \in E_{\circ}(S) \cap E_{\circ}(K) \} }
    \end{aligned}
\end{equation}
Clearly we have that $(S,K)$ is uniquely determined by 
\begin{align*}
    & \mathsf{SEQ}(S,K), \quad \overline{\mathsf{SEQ}}(S,K), \quad \sigma(S,K), \quad S_{\cap} = \big( S_{[\mathtt p_1,\mathtt q_1]}, \ldots, S_{[\mathtt p_{\mathtt T},\mathtt q_{\mathtt T}]} \big) \,; \\ 
    & S_{\setminus} = \big( S_{[\mathtt q_1,\mathtt p_2]}, \ldots, S_{[\mathtt q_{\mathtt T-1}, \mathtt p_{\mathtt T}]} \big), \quad K_{\setminus} = \big( K_{[\mathtt q_1',\mathtt p_2']}, \ldots, K_{[\mathtt q_{\mathtt T-1}',\mathtt p_{\mathtt T}']} \big) \,. 
\end{align*}
The next lemma shows how to bound $\Xi(S)$, $\Xi(K)$ and $\mathtt M(S,K)$ using the quantity above. Let 
\begin{equation}{\label{eq-def-aleph(sigma)}}
    \aleph(\sigma) = \#\big\{ \mathtt i: \sigma(\mathtt i)=0 \big\} \,.
\end{equation}

\begin{lemma}{\label{lem-bound-Xi-mathtt-M}}
    We have 
    \begin{align}
        & \Xi(S) \leq \rho^{-2\mathtt T} \cdot \Xi(S_{[\mathtt p_1,\mathtt q_1]}) \ldots \Xi(S_{[\mathtt p_{\mathtt T},\mathtt q_{\mathtt T}]}) \cdot \Xi( S_{[\mathtt q_1,\mathtt p_2]} ) \ldots \Xi( S_{[\mathtt q_{\mathtt T-1},\mathtt p_{\mathtt T}]} ) \,, \label{eq-bound-Xi(S)} \\
        & \Xi(K) \leq \rho^{-2\mathtt T} \cdot \Xi(K_{[\mathtt p'_0,\mathtt q_0']}) \ldots \Xi(K_{[\mathtt p_{\mathtt T}',\mathtt q_{\mathtt T}']}) \cdot \Xi( K_{[\mathtt q_1',\mathtt p_2']} ) \Xi( K_{[\mathtt q_{\mathtt T-1}',\mathtt p_{\mathtt T}']} )  \,,  \label{eq-bound-Xi(K)} \\
        & \mathtt M(S,K) \leq \rho^{-4\mathtt T} C^{\aleph(\sigma)+\mathtt T} \cdot \Xi(S_{[\mathtt q_1,\mathtt p_2]}) \Xi(K_{[\mathtt q_1',\mathtt p_2']}) \ldots \Xi(S_{[\mathtt q_{\mathtt T-1},\mathtt p_{\mathtt T}]}) \Xi(K_{[\mathtt q_{\mathtt T-1}',\mathtt p_{\mathtt T}']}) \,. \label{eq-bound-mathtt-M(S,K)}
    \end{align}
\end{lemma}

The proof of Lemma~\ref{lem-bound-Xi-mathtt-M} is postponed to Section~\ref{subsec:proof-lem-3.12} of the appendix. In addition, note that \eqref{eq-def-SEQ(S,K)} and \eqref{eq-def-overline-SEQ(S,K)} implies $\Xi(K_{[\mathtt p'_{\mathtt t}+1,\mathtt q_{\mathtt t}']})=\Xi(S_{[\mathtt p_{\mathtt t}+1,\mathtt q_{\mathtt t}]})$ for $1 \leq \mathtt t \leq \mathtt T$. Thus, using Lemma~\ref{lem-bound-Xi-mathtt-M}, we have $\Xi(S)\Xi(K)\mathtt M(S,K)$ is upper-bounded by 
\begin{align*}
     & (\rho^{-8}C)^{\mathtt T} C^{\aleph(\sigma)} \prod_{1\leq \mathtt t \leq \mathtt T} \Xi(S_{[\mathtt p_{\mathtt t},\mathtt q_{\mathtt t}]})^2 \cdot \prod_{1\leq \mathtt t \leq \mathtt T-1} \Xi(S_{[\mathtt q_{\mathtt t},\mathtt p_{\mathtt t+1}]})^2 \prod_{1\leq \mathtt t \leq \mathtt T-1} \Xi(K_{[\mathtt q'_{\mathtt t},\mathtt p_{\mathtt t+1}']})^2 \\
     =\ & (\rho^{-8}C)^{\mathtt T} C^{\aleph(\sigma)} \Xi(S_{\cap})^2 \Xi(S_{\setminus})^2 \Xi(K_{\setminus})^2 \,.
\end{align*}
In addition, note that 
\begin{equation}{\label{eq-transfrom-E_1-E_2}}
    |E_{\bullet}(S) \cap E_{\bullet}(K)| + |E_{\circ}(S) \cap E_{\circ}(K)| = \sum_{ 1 \leq \mathtt t \leq \mathtt T } (\mathtt q_{\mathtt t}-\mathtt p_{\mathtt t}) - \aleph(\sigma) \,.
\end{equation}
We can write \eqref{eq-var-Phi-i,j-relax-2} as 
\begin{align}
    \eqref{eq-var-Phi-i,j-relax-2} &\leq \sum_{\mathtt T \geq 1} \sum_{ \substack{ \mathsf{SEQ}, \overline{\mathsf{SEQ}} \\ S_{\cap}, S_{\setminus}, K_{\setminus}, \sigma } } \frac{ (\rho^{-8}C)^{\mathtt T} C^{\aleph(\sigma)} n^{ \sum_{ 1 \leq \mathtt t \leq \mathtt T } (\mathtt q_{\mathtt t}-\mathtt p_{\mathtt t}) - \aleph(\sigma) } }{ n^{2\ell-2} \beta_{\mathcal J}^2 } \cdot \Xi(S_{\cap})^2 \Xi(S_{\setminus})^2 \Xi(K_{\setminus})^2 \nonumber  \\ 
    &\leq [1+o(1)] \cdot \sum_{\mathtt T \geq 1} \sum_{ \substack{ \mathsf{SEQ}, \overline{\mathsf{SEQ}} \\ S_{\cap}, S_{\setminus}, K_{\setminus} } } \frac{ (\rho^{-8}C)^{\mathtt T} n^{ \sum_{ 1 \leq \mathtt t \leq \mathtt T } (\mathtt q_{\mathtt t}-\mathtt p_{\mathtt t}) } }{ n^{2\ell-2} \beta_{\mathcal J}^2 } \cdot \Xi(S_{\cap})^2 \Xi(S_{\setminus})^2 \Xi(K_{\setminus})^2 \,.  \label{eq-var-Phi-i,j-relax-3}
\end{align}
We will bound \eqref{eq-var-Phi-i,j-relax-3} via the following lemma, thus completing the proof of Lemma~\ref{lem-conditional-var-Phi-i,j}. The proof of the following lemma is postponed to Section~\ref{subsec:proof-lem-3.13} of the appendix.

\begin{lemma}{\label{lem-recovery-most-technical}}
    Suppose that Assumption~\ref{assum-upper-bound} and Equation~\eqref{eq-prelim-assumption-upper-bound} hold and we choose $\ell$ according to \eqref{eq-condition-weak-recovery}. Then $\eqref{eq-var-Phi-i,j-relax-3} \leq O_{\lambda,\mu,\rho,\delta}(1)$. 
\end{lemma}

\subsection{Proof of Proposition~\ref{main-prop-recovery-cov}}{\label{subsec:proof-prop-2.14}}

\begin{lemma}{\label{lem-conditional-exp-given-endpoints-cov}}
    We have 
    \begin{align*}
        \mathbb E_{\overline\Pb}\big[ h_S(\bm X,\bm Y) \mid x_u,x_v \big] = n^{-\ell} \Upsilon(S) x_u x_v
    \end{align*}
    for all $[S] \in \mathcal I$ and $\mathsf L(S)=\{ u,v \}$.
\end{lemma}
\begin{proof}
    Denote $\pi_{u,v}$ be the law of $(x,y)$ given $x_u,x_v$. We then have
    \begin{align*}
        \mathbb E_{\overline\Pb}\big[ h_S(\bm X,\bm Y) \mid x_u,x_v \big] = \mathbb E_{(x,y)\sim\pi_{u,v},(\bm u,\bm v)}\Big[ \mathbb E_{\overline{\Pb}}\big[ h_S(\bm X,\bm Y) \mid x,y,\bm u,\bm v \big] \Big] \,.
    \end{align*}
    Recall \eqref{eq-def-f-S-cov}, we have
    \begin{align*}
        \mathbb E_{\overline\Pb}\big[ h_S(\bm X,\bm Y) \mid x,y,\bm u,\bm v \big] &= \mathbb E_{\bm Z,\bm W}\Big[ \prod_{(i,j) \in E_{\bullet}(S)} \big( \tfrac{\sqrt{\lambda}}{\sqrt{n}} x_i x_j + \bm W_{i,j} \big) \prod_{(i,j) \in E_{\circ}(S)} \big( \tfrac{\sqrt{\mu}}{\sqrt{n}} y_i y_j + \bm Z_{i,j} \big) \Big] \\
        &= n^{-\ell} \lambda^{\frac{1}{2}|E_{\bullet}(S)|} \mu^{\frac{1}{2}|E_{\circ}(S)|} \prod_{(i,j) \in E_{\bullet}(S)} \big( x_i \bm u_j \big) \prod_{(i,j) \in E_{\circ}(S)} \big( y_i \bm v_j \big) \,.
    \end{align*}
    In addition, recall \eqref{eq-def-V-1-V-2-decorated} and recall that for all $[S] \in \mathcal J$ with $\mathsf L(S)=\{ u,v \}$, we have $\{ u,v \} \in V_{\bullet}(S)$ and $\mathsf{diff}(S) \subset V^{\mathsf a}(S)$. Thus,
    \begin{align*}
        & \prod_{(i,j) \in E_{\bullet}(S)} \big( x_i \bm u_j \big) \prod_{(i,j) \in E_{\circ}(S)} \big( y_i \bm v_j \big) \\
        =\ & x_u x_v \cdot \prod_{ i \in V^{\mathsf a}(S) \setminus \{ u,v \} } \big( x_i^2 \mathbf 1_{ \{ i \in V_{\bullet}(S) \setminus \mathsf{diff}(S) \} } + y_i^2 \mathbf 1_{ \{ i \in V_{\circ}(S) \setminus \mathsf{diff}(S) \} } + x_i y_i \mathbf 1_{ \{ i \in \mathsf{diff}(S) \} } \big) \\
        & \cdot \prod_{ i \in V^{\mathsf b}(S) } \big( \bm u_i^2 \mathbf 1_{ \{ i \in V_{\bullet}(S) \} } + \bm v_i^2 \mathbf 1_{ \{ i \in V_{\circ}(S) \} }  \big) \,.
    \end{align*}
    This yields that
    \begin{align*}
        \mathbb E_{(x,y) \sim \pi_{u,v}}\Big[ \prod_{(i,j) \in E_{\bullet}(S)} \big( x_i \bm u_j \big) \prod_{(i,j) \in E_{\circ}(S)} \big( y_i \bm v_j \big) \Big] = \rho^{|\mathsf{diff}(S)|} x_u x_v \,, 
    \end{align*}
    leading to the desired result (recall \eqref{eq-def-Upsilon-S}).
\end{proof}

Based on Lemma~\ref{lem-conditional-exp-given-endpoints-cov}, we see that 
\begin{align*}
    \mathbb E_{\overline\Pb}[h_S(\bm Y)]=0 \mbox{ for all } [S] \in \mathcal I, \mathsf L(S)=\{ u,v \} \Longrightarrow \mathbb E_{\overline\Pb}\big[ \Psi_{u,v}^{\mathcal I} \big]=0 \,.
\end{align*}
In addition, we have
\begin{align}
    &\mathbb E_{\overline\Pb}\big[ \Psi^{\mathcal I}_{u,v} \mid x_u,x_v \big] \overset{\eqref{eq-def-Phi-i,j-mathcal-I}}{=} \frac{ 1 }{ N^{\ell} n^{-1} \beta_{\mathcal I} } \sum_{[H] \in \mathcal I} \sum_{ \substack{ S \subset \mathsf K_{n,N}: S \cong H \\ \mathsf L(S)=\{ u,v \} } } \Upsilon(J)^2 n^{-\ell} x_u x_v \nonumber \\
    =\ & \frac{ x_u x_v }{ N^{\ell} n^{\ell-1} \beta_{\mathcal I} } \sum_{[H] \in \mathcal I} \Xi(H)^2 x_u x_v \cdot \#\big\{ S \subset \mathsf K_{n,N}: S \cong H, \mathsf L(S)=\{ u,v \} \big\} \nonumber \\
    =\ & \frac{ (1+o(1))x_u x_v }{ N^{\ell} n^{\ell-1} \beta_{\mathcal I} } \sum_{[H] \in \mathcal I} \Upsilon(J)^2 \cdot \frac{ N^{\ell} n^{\ell-1} }{ |\mathsf{Aut}(H)| }  \overset{\eqref{eq-def-beta-mathcal-I}}{=} (1+o(1)) x_u x_v \,.  \label{eq-conditional-mean-Phi-i,j-mathcal-I}
\end{align}
We now bound the variance of $\Psi_{i,j}^{\mathcal I}$ via the following lemma.
\begin{lemma}{\label{lem-conditional-var-Phi-i,j-mathcal-I}}
    Suppose that Assumption~\ref{assum-upper-bound} and Equation~\ref{eq-prelim-assumption-upper-bound} hold and we choose $\ell$ according to \eqref{eq-condition-weak-recovery}. Then
    \begin{equation}{\label{eq-conditional-var-Phi-i,j-mathcal-I}}
        \mathbb E_{\overline\Pb}\Big[ \big( \Psi_{u,v}^{\mathcal I} \big)^2 \Big], \ \mathbb E_{\overline\Pb}\Big[ \big( \Psi_{u,v}^{\mathcal I} \big)^2 x_u^2 x_v^2 \Big] \leq O_{\lambda,\mu,\rho}(1)  \,.
    \end{equation}
\end{lemma}

Based on Lemma~\ref{lem-conditional-var-Phi-i,j-mathcal-I}, we can now finish the proof of Proposition~\ref{main-prop-recovery-cov}.

\begin{proof}[Proof of Proposition~\ref{main-prop-recovery-cov} assuming Lemma~\ref{lem-conditional-var-Phi-i,j-mathcal-I}]
    Note that \eqref{eq-conditional-mean-Phi-i,j-mathcal-I} immediately implies that $\mathbb E_{\overline\Pb}[ \Psi^{\mathcal I}_{i,j} \cdot x_i x_j] = 1+o(1)$. Combined with Lemma~\ref{lem-conditional-var-Phi-i,j-mathcal-I}, we have shown Proposition~\ref{main-prop-recovery-cov}.
\end{proof}

The rest part of this subsection is devoted to the proof of Lemma~\ref{lem-conditional-var-Phi-i,j-mathcal-I}. We will only show how bound $\mathbb E_{\overline\Pb}[ (\Psi_{u,v}^{\mathcal I})^2 x_u^2 x_v^2 ]$ and $\mathbb E_{\overline\Pb}[ (\Psi_{u,v}^{\mathcal I})^2 ]$ can be bounded in the same manner. Recall \eqref{eq-def-Phi-i,j-mathcal-I}, we have
\begin{align}
    \mathbb E_{\overline\Pb}\Big[ \big( \Psi_{u,v}^{\mathcal I} \big)^2 x_u^2 x_v^2 \Big] = \sum_{ \substack{ [S],[K] \in \mathcal I \\ \mathsf L(S)=\mathsf L(K)=\{ i,j \} } } \frac{ \Upsilon(S) \Upsilon(K) }{ N^{2\ell} n^{-2} \beta_{\mathcal I}^2 } \mathbb E_{\overline\Pb}\big[ h_S h_K \cdot x_u^2 x_v^2 \big] \,.  \label{eq-var-Phi-i,j-mathcal-I-relax-1}
\end{align}

\begin{lemma}{\label{lem-est-cov-f-S-f-K-chain-case-cov}}
    Recall \eqref{eq-def-E-1-E-2-decorated} and \eqref{eq-def-mathtt-P}. Suppose that Assumption~\ref{assum-upper-bound} holds, we have 
    \begin{align*}
        \mathbb E_{\overline\Pb}\big[ h_S h_K \cdot x_u^2 x_v^2 \big] \leq C^2 n^{ -2\ell+|E_{\bullet}(S)\cap E_{\bullet}(K)| +|E_{\circ}(S) \cap E_{\circ}(K)| } \mathtt P(S,K)  \,.
    \end{align*} 
\end{lemma}

The proof of Lemma~\ref{lem-est-cov-f-S-f-K-chain-case-cov} is postponed to Section~\ref{subsec:proof-lem-3.16} of the appendix. To this end, plugging Lemma~\ref{lem-est-cov-f-S-f-K-chain-case-cov} into \eqref{eq-var-Phi-i,j-mathcal-I-relax-1}, we have
\begin{align}
    \eqref{eq-var-Phi-i,j-mathcal-I-relax-1} &\leq \frac{ C^2 }{ N^{2\ell} n^{-2} \beta_{\mathcal I}^2 } \sum_{ \substack{ [S], [K] \in \mathcal I \\ \mathsf L(S)=\mathsf L(K)=\{ i,j \} } } n^{ -2\ell+|E_{\bullet}(S)\cap E_{\bullet}(K)| +|E_{\circ}(S) \cap E_{\circ}(K)| } \Upsilon(S)\Upsilon(K) \mathtt P(S,K) \nonumber  \\
    &= \frac{ C^2 }{ N^{2\ell} n^{2\ell-2} \beta_{\mathcal I}^2 } \sum_{ \substack{ [S], [K] \in \mathcal I \\ \mathsf L(S)=\mathsf L(K)=\{ i,j \} } } n^{ |E_{\bullet}(S)\cap E_{\bullet}(K)| +|E_{\circ}(S) \cap E_{\circ}(K)| } \Upsilon(S)\Upsilon(K) \mathtt P(S,K) \,.  \label{eq-var-Phi-i,j-mathcal-I-relax-2}
\end{align}
Recall the definition of $\mathsf{SEQ},\overline{\mathsf{SEQ}},\sigma$ and $S_{\cap}, S_{\setminus}, K_{\setminus}$ we introduced in Section~\ref{subsec:proof-prop-2.10}. Again, we will show how to bound $\Upsilon(S)$, $\Upsilon(K)$ and $\mathtt P(S,K)$ using the quantity above.

\begin{lemma}{\label{lem-bound-Upsilon-mathtt-P}}
    Recall \eqref{eq-def-aleph(sigma)}. We have 
    \begin{align}
        & \Upsilon(S) \leq \rho^{-2\mathtt T} \cdot \Upsilon(S_{[\mathtt p_1,\mathtt q_1]}) \ldots \Upsilon(S_{[\mathtt p_{\mathtt T},\mathtt q_{\mathtt T}]}) \cdot \Upsilon( S_{[\mathtt q_1,\mathtt p_2]} ) \ldots \Upsilon( S_{[\mathtt q_{\mathtt T-1},\mathtt p_{\mathtt T}]} ) \,, \label{eq-bound-Upsilon(S)} \\
        & \Upsilon(K) \leq \rho^{-2\mathtt T} \cdot \Upsilon(K_{[\mathtt p'_0,\mathtt q_0']}) \ldots \Upsilon(K_{[\mathtt p_{\mathtt T}',\mathtt q_{\mathtt T}']}) \cdot \Upsilon( K_{[\mathtt q_1',\mathtt p_2']} ) \Upsilon( K_{[\mathtt q_{\mathtt T-1}',\mathtt p_{\mathtt T}']} )  \,,  \label{eq-bound-Upsilon(K)} \\
        & \mathtt P(S,K) \leq \rho^{-4\mathtt T} C^{\aleph(\sigma)+\mathtt T} \cdot \Upsilon(S_{[\mathtt q_1,\mathtt p_2]}) \Upsilon(K_{[\mathtt q_1',\mathtt p_2']}) \ldots \Upsilon(S_{[\mathtt q_{\mathtt T-1},\mathtt p_{\mathtt T}]}) \Upsilon(K_{[\mathtt q_{\mathtt T-1}',\mathtt p_{\mathtt T}']}) \,. \label{eq-bound-mathtt-P(S,K)}
    \end{align}
\end{lemma}

The proof of Lemma~\ref{lem-bound-Upsilon-mathtt-P} is almost identical to the proof of Lemma~\ref{lem-bound-Xi-mathtt-M}, so we will omit further details here. In addition, note that \eqref{eq-def-SEQ(S,K)} and \eqref{eq-def-overline-SEQ(S,K)} implies $\Upsilon(K_{[\mathtt p'_{\mathtt t}+1,\mathtt q_{\mathtt t}']})=\Upsilon(S_{[\mathtt p_{\mathtt t}+1,\mathtt q_{\mathtt t}]})$ for $1 \leq \mathtt t \leq \mathtt T$. Thus, using Lemma~\ref{lem-bound-Upsilon-mathtt-P}, we have $\Upsilon(S)\Upsilon(K)\mathtt P(S,K)$ is upper-bounded by 
\begin{align*}
     & (2\rho^{-8}C^2)^{\mathtt T} (2C)^{\aleph(\sigma)} \prod_{1\leq \mathtt t \leq \mathtt T} \Upsilon(S_{[\mathtt p_{\mathtt t},\mathtt q_{\mathtt t}]})^2 \cdot \prod_{1\leq \mathtt t \leq \mathtt T-1} \Upsilon(S_{[\mathtt q_{\mathtt t},\mathtt p_{\mathtt t+1}]})^2 \prod_{1\leq \mathtt t \leq \mathtt T-1} \Upsilon(K_{[\mathtt q'_{\mathtt t},\mathtt p_{\mathtt t+1}']})^2 \\
     =\ & (2\rho^{-8}C^2)^{\mathtt T} (2C)^{\aleph(\sigma)} \Upsilon(S_{\cap})^2 \Upsilon(S_{\setminus})^2 \Upsilon(K_{\setminus})^2 \,.
\end{align*}
In addition, recall \eqref{eq-transfrom-E_1-E_2}, we have
\begin{align}
    \eqref{eq-var-Phi-i,j-mathcal-I-relax-2} &\leq \sum_{\mathtt T \geq 1} \sum_{ \substack{ \mathsf{SEQ}, \overline{\mathsf{SEQ}} \\ S_{\cap}, S_{\setminus}, K_{\setminus}, \sigma } } \frac{ (2\rho^{-8}C^2)^{\mathtt T} (2C)^{\aleph(\sigma)} n^{ \sum_{ 1 \leq \mathtt t \leq \mathtt T } (\mathtt q_{\mathtt t}-\mathtt p_{\mathtt t}) - \aleph(\sigma) } }{ N^{2\ell} n^{2\ell-2} \beta_{\mathcal I}^2 } \cdot \Upsilon(S_{\cap})^2 \Upsilon(S_{\setminus})^2 \Upsilon(K_{\setminus})^2 \nonumber  \\ 
    &\leq [1+o(1)] \cdot \sum_{\mathtt T \geq 1} \sum_{ \substack{ \mathsf{SEQ}, \overline{\mathsf{SEQ}} \\ S_{\cap}, S_{\setminus}, K_{\setminus} } } \frac{ (2\rho^{-8}C^2)^{\mathtt T} n^{ \sum_{ 1 \leq \mathtt t \leq \mathtt T } (\mathtt q_{\mathtt t}-\mathtt p_{\mathtt t}) } }{ N^{2\ell} n^{2\ell-2} \beta_{\mathcal I}^2 } \cdot \Upsilon(S_{\cap})^2 \Upsilon(S_{\setminus})^2 \Upsilon(K_{\setminus})^2 \,.  \label{eq-var-Phi-i,j-mathcal-I-relax-3}
\end{align}

Based on \eqref{eq-var-Phi-i,j-mathcal-I-relax-3}, it suffices to show the following lemma, whose proof is incorporated in Section~\ref{subsec:proof-lem-3.18} of the appendix.

\begin{lemma}{\label{lem-recovery-most-technical-cov}}
    Suppose that Assumption~\ref{assum-upper-bound} and Equation~\eqref{eq-prelim-assumption-upper-bound} hold, and we choose $\ell$ according to \eqref{eq-condition-weak-recovery}. Then $\eqref{eq-var-Phi-i,j-mathcal-I-relax-3} \leq O_{\lambda,\mu,\rho,\gamma}(1)$. 
\end{lemma}

\subsection{Proof of Theorems~\ref{MAIN-THM-recovery} and \ref{MAIN-THM-recovery-cov}}{\label{subsec:proof-thm-2.11}}

With Propositions~\ref{main-prop-recovery} and \ref{main-prop-recovery-cov}, we are ready to prove Theorems~\ref{MAIN-THM-recovery} and \ref{MAIN-THM-recovery-cov}.
\begin{proof}[Proof of Theorems~\ref{MAIN-THM-recovery} and \ref{MAIN-THM-recovery-cov}]
    We first show Theorem~\ref{MAIN-THM-recovery}. Note by our definition of $\widehat x$ we have $|\widehat{x}_u| \leq R^4$ and thus $\| \widehat x \|^2 \leq R^8 n$. In addition, from Proposition~\ref{main-prop-recovery} and a simple Chebyshev's inequality we have 
    \begin{align*}
        \Pb\Big( \widehat x_u \neq \Phi_{u,w}^{\mathcal J} \Big) = \Pb\Big( |\Phi_{u,w}^{\mathcal J}| \geq R^4 \Big) \leq \frac{ \mathbb E_{\Pb}[ (\Phi_{u,w}^{\mathcal J})^2 ] }{ (R^4)^2 } \leq R^{-4} \,.
    \end{align*}
    Thus, we have
    \begin{align*}
        \mathbb E\big[ \widehat x_u \cdot x_u x_w \big] &= \mathbb E\big[ \Phi_{u,w}^{\mathcal J} \cdot x_u x_w \big] - \mathbb E\Big[ (\Phi_{u,w}^{\mathcal J}- \widehat x_u) \cdot x_u x_w \Big] \\
        &\overset{\text{Proposition~\ref{main-prop-recovery}}}{\geq} 1+o(1) - \mathbb E\Big[ \mathbf 1_{ \{ |\Phi_{u,w}^{\mathcal J}|>R^4 \} } \cdot \big| \Phi_{u,w}^{\mathcal J} x_u x_w \big| \Big] \\
        &\leq 1+o(1) - \Pb\Big( |\Phi_{u,w}^{\mathcal J}|>R^4 \Big)^{1/2} \cdot \mathbb E_{\Pb}\Big[ \Phi_{u,w}^{\mathcal J} x_u^2 x_w^2 \Big]^{1/2} \\
        &\overset{\text{Proposition~\ref{main-prop-recovery}}}{\geq} 1+o(1) - \big( R^{-4} \cdot R \big)^{1/2} \geq \frac{1}{2} \,,
    \end{align*}
    and thus $\mathbb E_{\Pb}[ |\langle \widehat x,x \rangle| ] \geq \Omega(1) \cdot n$. In addition, from Assumption~\ref{assum-upper-bound} we see that
    \begin{align*}
        \mathbb E_{\Pb}\big[ \| \widehat x \|^2 \| x \|^2 \big] \leq R^8 n \cdot \mathbb E_{\Pb}\big[ \| x \|^2 \big] \leq O_{\lambda,\mu,\rho,\delta}(1) \cdot n^2 \,.
    \end{align*}
    Thus, we have
    \begin{align*}
        \mathbb E_{\Pb}\Big[ \tfrac{ |\langle \widehat x,x \rangle| }{ \|\widehat x\| \|x\| } \Big] \geq \mathbb E_{\Pb}\Big[ \tfrac{ \langle \widehat x,x \rangle^2 }{ \|\widehat x\|^2 \|x\|^2 } \Big] \geq \mathbb E_{\Pb}\big[ |\langle \widehat x,x \rangle| \big]^2 \cdot \mathbb E_{\Pb}\big[ \| \widehat x \|^2 \| x \|^2 \big]^{-1} \geq \Omega_{\lambda,\mu,\rho,\delta}(1) \,,
    \end{align*}
    where the first inequality follows from $\frac{|\langle \widehat x,x \rangle|}{\|\widehat x\|\|x\|} \in (0,1)$ and the second inequality follows from Cauchy-Schwartz inequality. We can derive Theorem~\ref{MAIN-THM-recovery-cov} from Proposition~\ref{main-prop-recovery-cov} in the same manner.
\end{proof}

\section{Approximate statistics by color coding}{\label{sec:hypergraph-color-coding}}

In this section, we provide an efficient algorithm to approximately compute the detection statistic $f_{\mathcal H},h_{\mathcal G}$ given in \eqref{eq-def-f-mathcal-H}, \eqref{eq-def-f-mathcal-G} and the recovery statistic $\{ \Phi^{\mathcal J}_{u,w}\},\{\Psi^{\mathcal I}_{u,w} \}$ given in \eqref{eq-def-Phi-i,j-mathcal-J}, \eqref{eq-def-Phi-i,j-mathcal-I}, using the idea of color coding.

\subsection{Approximate the detection statistics}{\label{subsec:approx-detection}}

\subsubsection{The correlated spiked Wigner model}{\label{subsubsec:approx-detection-Wigner}}

Recall that we define $\mathcal H$ to be the collection of all decorated cycles with length $\ell$. By applying the color coding method, we first approximately count the signed subgraphs that are isomorphic to a query decorated cycle in $\mathcal H$ with $\ell$ edges and $\ell$ vertices. Specifically, recall \eqref{eq-def-correlated-spike-specific}, we generate a random coloring $\tau:[n] \to [\ell]$ that assigns a color to each vertex of $[n]$ from the color set $\ell$ independently and uniformly at random. Given any $V \subset [n]$, let $\chi_{\tau}(V)$ indicate that $\tau(V)$ is colorful, i.e., $\tau(x) \neq \tau(y)$ for any distinct $x,y \in V$. In particular, if $|V|=mp\ell$, then $\chi_{\tau}(V)=1$ with probability
\begin{equation}{\label{eq-def-r}}
    r := \frac{ \ell! }{ (\ell)^{\ell} } = e^{-\Theta(\ell)} \,.
\end{equation}
For any unlabeled decorated graph $[H]$ with $\ell$ vertices, we define
\begin{equation}{\label{eq-def-g-S}}
    \mathfrak F_{H}(\bm X,\bm Y,\tau) := \sum_{ S \subset \mathsf K_n, S \cong H } \chi_{\tau}(V(S)) \prod_{(i,j) \in E_{\bullet}(S)} \bm X_{i,j} \prod_{(i,j) \in E_{\circ}(S)} \bm Y_{i,j} \,.
\end{equation}
Then 
\begin{equation}{\label{eq-averaging-over-coloring}}
    \mathbb E\big[ \mathfrak F_{H}(\bm X,\bm Y,\tau) \mid \bm X,\bm Y \big] = r\sum_{S \cong H} f_S(\bm X,\bm Y)
\end{equation}
where $f_S(\bm X,\bm Y)$ was defined in \eqref{eq-def-f-S}. To further obtain an accurate approximation of $\sum f_S(\bm X,\bm Y)$, we average over multiple copies of $\mathfrak F_H(\bm X,\bm Y,\tau)$ by generating $t$ independent random colorings, where
\begin{align*}
    t := \lceil 1/r \rceil \,.
\end{align*}
Next, we plug in the averaged subgraph count to approximately compute $f_{\mathcal H}(\bm X,\bm Y)$. Specifically, we generate $t$ random colorings $ \{ \tau_{\mathtt k} \}_{\mathtt k=1}^{t}$ which are independent copies of $\tau$ that map $[n]$ to $[\ell]$. Then, we define
\begin{equation}{\label{eq-def-widetilde-f-H}}
    \widetilde{f}_{\mathcal H} := \frac{1}{ \sqrt{ n^{\ell}\beta_{\mathcal H} }}  \sum_{ [H] \in \mathcal H } \Xi(H) \Big( \frac{1}{tr} \sum_{\mathtt k=1}^{t} \mathfrak F_H(\bm X,\bm Y,\tau_{\mathtt k}) \Big) \,.
\end{equation}
The following result shows that $\widetilde{f}_{\mathcal H}(\bm X,\bm Y)$ well approximates $f_{\mathcal H}(\bm X,\bm Y)$ in a relative sense.

\begin{proposition}{\label{prop-approximate-detection-statistics}}
    Suppose Assumption~\ref{assum-upper-bound} and Equation~\ref{eq-prelim-assumption-upper-bound} hold and we choose $\ell$ according to \eqref{eq-condition-strong-detection}. Then under both $\Pb$ and $\Qb$
    \begin{align}{\label{eq-approximate-detection-statistics}}
        \frac{ \widetilde{f}_{\mathcal H} - f_{\mathcal H} }{ \mathbb E_{\Pb}[f_{\mathcal H}] } \overset{L_2}{\longrightarrow} 0  \,.
    \end{align}
\end{proposition}

The proof of Proposition~\ref{prop-approximate-detection-statistics} is postponed to Section~\ref{subsec:proof-prop-4.1} of the appendix. Since the convergence in $L_2$ implies the convergence in probability, as an immediate corollary of Theorem~\ref{MAIN-THM-detection} and Proposition~\ref{prop-approximate-detection-statistics}, we conclude that the approximate test statistic $\widetilde{f}_{\mathcal H}$ succeeds under the same condition as the original test statistic $f_{\mathcal H}$. 

Now we show that $\widetilde{f}_{\mathcal H}$ can be calculated efficiently, and the key is to show the following lemma.

\begin{lemma}{\label{lem-color-coding}}
    For any coloring $\tau:[n] \to [\ell]$ and any graph $[H] \in \mathcal H$, there exists an algorithm (see Algorithm~\ref{alg:dynamic-programming}) that computes $\mathfrak F_H(\bm X,\bm Y,\tau)$ in time $O(n^{2+o(1)})$.
\end{lemma}

The proof of Lemma~\ref{lem-color-coding} is postponed to Section~\ref{subsec:proof-lem-4.2} of the appendix. Based on Lemma~\ref{lem-color-coding}, we show that the approximate test statistic $\widetilde{f}$ can be computed efficiently by the following algorithm.
\begin{breakablealgorithm}{\label{alg:cal-widetilde-f}}
    \caption{Computation of $\widetilde{f}_{\mathcal H}$}
    \begin{algorithmic}[1]
    \STATE {\bf Input:} Symmetric matrices $\bm X,\bm Y$, signal-to-noise ratio $\lambda,\mu$, correlation $\rho$, and integer parameter $\ell$.
    \STATE List all non-isomorphic unlabeled decorated cycles with length $\ell$. Denote this list as $\mathcal H$.
    \STATE For each $[H] \in \mathcal H$, compute $\operatorname{Aut}(H)$. 
    \STATE Generate i.i.d.\ random colorings $\{ \tau_{\mathtt k} : 1 \leq \mathtt k \leq t \}$ that map $[n]$ to $[\ell]$ uniformly at random. 
    \FOR{each $1 \leq \mathtt k \leq t$}
    \STATE For each $[H] \in \mathcal H$, compute $\mathfrak F_{H}(\bm X,\bm Y,\tau_{\mathtt k})$ via Algorithm~\ref{alg:dynamic-programming}. 
    \ENDFOR
    \STATE Compute $\widetilde{f}_{\mathcal H}(\bm X,\bm Y)$ according to \eqref{eq-def-widetilde-f-H}.
    \STATE {\bf Output:} $\widetilde{f}_{\mathcal H}(\bm X,\bm Y)$.
    \end{algorithmic}
\end{breakablealgorithm}

\begin{proposition}{\label{prop-running-time-detection}}
    Suppose we choose $\ell$ according to \eqref{eq-condition-strong-detection}, Algorithm~\ref{alg:cal-widetilde-f} computes $\widetilde{f}_{\mathcal H}(\bm X,\bm Y)$ in time $n^{2+o(1)}$. 
\end{proposition}
\begin{proof}
    Listing all non-isomorphic, unrooted unlabeled hypergraphs in $\mathcal H$ takes time $2^{\ell}=n^{o(1)}$. Since $[H] \in \mathcal H$ are cycles, calculating $\mathsf{Aut}(H)$ for all $[H]\in \mathcal H$ takes time $|\mathcal H| \cdot O(\ell)=n^{o(1)}$. For all $1 \leq i \leq t$, calculating $\mathfrak F_{H}(\bm X,\bm Y,\tau_i)$ takes time $n^{2+o(1)}$. Thus, the total running time of Algorithm~\ref{alg:cal-widetilde-f} is 
    \begin{equation*}
        O(n^{o(1)}) + O(n^{o(1)}) + 2t \cdot O(n^{2+o(1)}) = O(n^{2+o(1)}) \,. \qedhere
    \end{equation*}
\end{proof}

We complete this subsection by pointing out that Theorem~\ref{MAIN-THM-detection-algorithmic} follows directly from Proposition~\ref{prop-approximate-detection-statistics} and Proposition~\ref{prop-running-time-detection}.

\subsubsection{The correlated spiked Wishart model}{\label{subsubsec:approx-detection-cov}}

Recall that we define $\mathcal G$ to be the collection of all decorated, bipartite cycles with length $2\ell$. Again, we first approximately count the signed subgraphs that are isomorphic to a query decorated, bipartite cycle in $\mathcal G$ with $2\ell$ edges. Specifically, recall \eqref{eq-def-correlated-spike-covariance-specific}, we generate random colorings $\tau:[n],[N] \to [2\ell]$ that assigns a color to each vertex of $[n],[N]$ from the color set $2\ell$ independently and uniformly at random. Given any $V=V^{\mathsf a} \sqcup V^{\mathsf b}$ such that $V^{\mathsf a} \subset [n],V^{\mathsf b} \subset [N]$, let $\chi_{\tau}(V)$ indicate that $\tau(V)$ is colorful, i.e., $\tau(x) \neq \tau(y)$ for any distinct $x,y \in V$. In particular, if $|V|=2\ell$, then $\chi_{\tau}(V)=1$ with probability
\begin{equation}{\label{eq-def-r-cov}}
    r := \frac{ (2\ell!) }{ (2\ell)^{2\ell} } = e^{-\Theta(\ell)} \,.
\end{equation}
For any unlabeled, decorated, bipartite graph $[H]$ with $\ell$ vertices, we define
\begin{equation}{\label{eq-def-g-S-cov}}
    \mathfrak G_{H}(\bm X,\bm Y,\tau) := \sum_{ S \subset \mathsf K_{n,N}, S \cong H } \chi_{\tau,\tau'}(V(S)) \prod_{(i,j) \in E_{\bullet}(S)} \bm X_{i,j} \prod_{(i,j) \in E_{\circ}(S)} \bm Y_{i,j} \,.
\end{equation}
Then 
\begin{equation}{\label{eq-averaging-over-coloring-cov}}
    \mathbb E\big[ \mathfrak G_{H}(\bm X,\bm Y,\tau) \mid \bm X,\bm Y \big] = r\sum_{S \cong H} h_S(\bm X,\bm Y)
\end{equation}
where $h_S(\bm X,\bm Y)$ was defined in \eqref{eq-def-f-S-cov}. To further obtain an accurate approximation of $\sum (\bm X,\bm Y)$, we average over multiple copies of $\mathfrak G_H(\bm X,\bm Y,\tau)$ by generating $t$ independent random colorings, where
\begin{align*}
    t := \lceil 1/r \rceil \,.
\end{align*}
Next, we plug in the averaged subgraph count to approximately compute $h_{\mathcal G}(\bm X,\bm Y)$. Specifically, we generate $t$ random colorings $ \{ \tau_{\mathtt k} \}_{\mathtt k=1}^{t}$ which are independent copies of $\tau$ that map $[n],[N]$ to $[2\ell]$. Then, we define
\begin{equation}{\label{eq-def-widetilde-f-mathcal-G}}
    \widetilde{h}_{\mathcal G} := \frac{1}{ \sqrt{ N^{\ell} n^{\ell} \beta_{\mathcal G} }}  \sum_{ [H] \in \mathcal G } \Upsilon(H) \Big( \frac{1}{tr} \sum_{\mathtt k=1}^{t} \mathfrak G_H(\bm X,\bm Y,\tau_{\mathtt k}) \Big) \,.
\end{equation}
The following result shows that $\widetilde{h}_{\mathcal G}(\bm X,\bm Y)$ well approximates $h_{\mathcal G}(\bm X,\bm Y)$ in a relative sense.

\begin{proposition}{\label{prop-approximate-detection-statistics-cov}}
    Suppose Assumption~\ref{assum-upper-bound} and Equation~\ref{eq-prelim-assumption-upper-bound} hold and we choose $\ell$ according to \eqref{eq-condition-strong-detection}. Then under both $\overline\Pb$ and $\overline\Qb$
    \begin{align}{\label{eq-approximate-detection-statistics-cov}}
        \frac{ \widetilde{h}_{\mathcal G} - h_{\mathcal G} }{ \mathbb E_{\overline\Pb}[h_{\mathcal G}] } \overset{L_2}{\longrightarrow} 0  \,.
    \end{align}
\end{proposition}

The proof of Proposition~\ref{prop-approximate-detection-statistics-cov} is incorporated in Section~\ref{subsec:proof-prop-4.4} of the appendix. Since the convergence in $L_2$ implies the convergence in probability, as an immediate corollary of Theorem~\ref{MAIN-THM-detection-cov} and Proposition~\ref{prop-approximate-detection-statistics-cov}, we conclude that the approximate test statistic $\widetilde{h}_{\mathcal G}$ succeeds under the same condition as the original test statistic $h_{\mathcal G}$. 

Now we show that $\widetilde{h}_{\mathcal G}$ can be calculated efficiently, and the key is to show the following lemma.

\begin{lemma}{\label{lem-color-coding-cov}}
    For any colorings $\tau:[n],[N] \to [2\ell]$ and any hypergraph $[H] \in \mathcal G$, there exists an algorithm (see Algorithm~\ref{alg:dynamic-programming-cov}) that computes $\mathfrak G_H(\bm X,\bm Y,\tau)$ in time $O(n^{2+o(1)})$.
\end{lemma}

The proof of Lemma~\ref{lem-color-coding} is postponed to Section~\ref{subsec:proof-lem-4.5} of the appendix. Based on Lemma~\ref{lem-color-coding-cov}, we show that the approximate test statistic $\widetilde{h}_{\mathcal G}$ can be computed efficiently by the following algorithm.
\begin{breakablealgorithm}{\label{alg:cal-widetilde-f-mathcal-G}}
    \caption{Computation of $\widetilde{h}_{\mathcal G}$}
    \begin{algorithmic}[1]
    \STATE {\bf Input:} $n*N$ matrices $\bm X,\bm Y$ with $n=\gamma N$, signal-to-noise ratio $\lambda,\mu$, correlation $\rho$, and integer parameter $\ell$.
    \STATE List all non-isomorphic unlabeled decorated, bipartite cycles with length $2\ell$. Denote this list as $\mathcal G$.
    \STATE For each $[H] \in \mathcal G$, compute $\operatorname{Aut}(H)$. 
    \STATE Generate i.i.d.\ random colorings $\{ \tau_{\mathtt k} : 1 \leq \mathtt k \leq t \}$ that map $[n],[N]$ to $[2\ell]$ uniformly at random. 
    \FOR{each $1 \leq \mathtt k \leq t$}
    \STATE For each $[H] \in \mathcal G$, compute $\mathfrak G_{H}(\bm X,\bm Y,\tau_{\mathtt k})$ via Algorithm~\ref{alg:dynamic-programming}. 
    \ENDFOR
    \STATE Compute $\widetilde{h}_{\mathcal G}(\bm X,\bm Y)$ according to \eqref{eq-def-widetilde-f-mathcal-G}.
    \STATE {\bf Output:} $\widetilde{h}_{\mathcal G}(\bm X,\bm Y)$.
    \end{algorithmic}
\end{breakablealgorithm}

\begin{proposition}{\label{prop-running-time-detection-cov}}
    Suppose we choose $\ell$ according to \eqref{eq-condition-strong-detection}, Algorithm~\ref{alg:cal-widetilde-f-mathcal-G} computes $\widetilde{h}_{\mathcal G}(\bm X,\bm Y)$ in time $n^{2+o(1)}$. 
\end{proposition}
\begin{proof}
    Listing all non-isomorphic, unrooted unlabeled hypergraphs in $\mathcal G$ takes time $2^{2\ell}=n^{o(1)}$. Since $[H] \in \mathcal G$ are cycles, calculating $\mathsf{Aut}(H)$ for all $[H]\in \mathcal G$ takes time $|\mathcal G| \cdot O(\ell)=n^{o(1)}$. For all $1 \leq i \leq t$, calculating $\mathfrak G_{H}(\bm X,\bm Y,\tau_i)$ takes time $n^{2+o(1)}$. Thus, the total running time of Algorithm~\ref{alg:cal-widetilde-f-mathcal-G} is 
    \begin{equation*}
        O(n^{o(1)}) + O(n^{o(1)}) + 2t \cdot O(n^{2+o(1)}) = O(n^{2+o(1)}) \,. \qedhere
    \end{equation*}
\end{proof}

We complete this subsection by pointing out that Theorem~\ref{MAIN-THM-detection-algorithmic-cov} follows directly from Proposition~\ref{prop-approximate-detection-statistics-cov} and Proposition~\ref{prop-running-time-detection-cov}.

\subsection{Approximate the recovery statistics}{\label{subsec:approx-recovery}}

\subsubsection{The correlated spiked Wigner model}{\label{subsec:approx-recovery-Wigner}}

Similar as in Section~\ref{subsec:approx-detection}, we generate a random coloring $\xi:[n] \to [\ell+1]$ that assigns a color to each vertex of $[n]$ from the color set $\ell+1$ independently and uniformly at random. Given any $V \subset [n]$, let $\chi_{\xi}(V)$ indicate that $\xi(V)$ is colorful, i.e., $\xi(x) \neq \xi(y)$ for any distinct $x,y \in V$. In particular, if $|V|=mp\ell$, then $\chi_{\xi}(V)=1$ with probability
\begin{equation}{\label{eq-def-varsigma}}
    \varkappa := \frac{ (\ell+1)! }{ (\ell+1)^{\ell+1} } \,.
\end{equation}
For any unlabeled decorated graph $[H] \in \mathcal J$ with $\ell+1$ vertices, we define
\begin{equation}{\label{eq-def-g-J-xi}}
    \mathfrak L_{H}(\bm X,\bm Y,\xi) := \sum_{ \substack{ S \subset \mathsf K_n: S \cong H \\ \mathsf{L}(S)=\{ u,v \} } } \chi_{\xi}(V(S)) \prod_{(i,j) \in E_{\bullet}(S)} \bm X_{i,j} \prod_{(i,j) \in E_{\circ}(S)} \bm Y_{i,j} \,.
\end{equation}
Then 
\begin{equation}{\label{eq-averaging-over-coloring-recovery}}
    \mathbb E\big[ \mathfrak L_{H}(\bm X,\bm Y,\xi) \mid \bm X,\bm Y \big] = \varkappa \sum_{ \substack{ S \subset \mathsf K_n: S \cong H \\ \mathsf{L}(S)=\{ u,v \} } } f_S(\bm X,\bm Y)
\end{equation}
where $f_S(\bm X,\bm Y)$ was defined in \eqref{eq-def-f-S}. Next, we define $t=\lceil \frac{1}{\varkappa} \rceil$ and generate $t$ random colorings $ \{ \xi_{\mathtt k} \}_{\mathtt k=1}^{t}$ which are independent copies of $\xi$ that map $[n]$ to $[\ell+1]$. Then, we define
\begin{equation}{\label{eq-def-widetilde-Phi-H}}
    \widetilde{\Phi}^{\mathcal J}_{u,v} := \frac{1}{ n^{\frac{\ell}{2}-1}\beta_{\mathcal J} } \sum_{[H] \in \mathcal J} \Xi(H) \Big( \frac{1}{t\varkappa} \sum_{\mathtt k=1}^{t} \mathfrak L_H(\bm X,\bm Y,\xi_{\mathtt k}) \Big) \,.
\end{equation}
Moreover, the following result bounds the approximation error under
the same conditions as those in Proposition~\ref{main-prop-recovery} for the second moment calculation.

\begin{proposition}{\label{prop-approximate-recovery-statistics}}
    Suppose Assumption~\ref{assum-upper-bound} and Equation~\ref{eq-prelim-assumption-upper-bound} hold and we choose $\ell$ according to \eqref{eq-condition-weak-recovery}. Then there exists $R=O_{\lambda,\mu,\rho,\delta}(1)$ such that
    \begin{equation}{\label{eq-approximate-recovery-statistics}}
        \mathbb E_{\Pb}\Big[ \widetilde{\Phi}_{u,v}^{\mathcal J} \cdot x_u x_v \Big] = 1+o(1) \mbox{ and } \mathbb E_{\Pb}\Big[ \big(\widetilde{\Phi}_{u,v}^{\mathcal J}\big)^2 \Big], \ \mathbb E_{\Pb}\Big[ \big(\widetilde{\Phi}_{u,v}^{\mathcal J}\big)^2 x_u^2 x_v^2 \Big] \leq R \,.
    \end{equation}
\end{proposition}

The proof of Proposition~\ref{prop-approximate-recovery-statistics} is postponed to Section~\ref{subsec:proof-prop-4.4} of the appendix. Now we show the approximate subgraph counts $\mathfrak L_{H}(\bm X,\bm Y,\xi)$ can be computed efficiently via the following algorithm. And (similar as in Section~\ref{subsec:approx-detection}) the key is to show the following lemma.

\begin{lemma}{\label{lem-color-coding-recovery}}
    For any coloring $\xi:[n] \to [\ell+1]$ and any hypergraph $[H] \in \mathcal J$, there exists an algorithm (see Algorithm~\ref{alg:dynamic-programming}) that computes $\mathfrak L_H(\bm X,\bm Y,\xi)$ in time $O(n^{T+o(1)})$ for some constant $T=O(1) \cdot (\tfrac{\ell}{\log n}+1 )$.
\end{lemma}

The proof of Lemma~\ref{lem-color-coding-recovery} is postponed to Section~\ref{subsec:proof-lem-4.8} of the appendix. Based on Lemma~\ref{lem-color-coding-recovery}, we show that the approximate test statistic $\widetilde{f}$ can be computed efficiently.
\begin{breakablealgorithm}{\label{alg:cal-widetilde-Phi}}
    \caption{Computation of $\widetilde{\Phi}^{\mathcal J}_{u,v}$}
    \begin{algorithmic}[1]
    \STATE {\bf Input:} Symmetric matrices $\bm X,\bm Y$, signal-to-noise ratio $\lambda,\mu$, correlation $\rho$, and integer parameter $\ell$.
    \STATE List all unlabeled decorated paths in $\mathcal J=\mathcal J(\ell)$. 
    \STATE For each $[H] \in \mathcal J$, compute $\operatorname{Aut}(H)$. 
    \STATE Generate i.i.d.\ random colorings $\{ \xi_{\mathtt k} : 1 \leq \mathtt k \leq t \}$ that map $[n]$ to $[\ell+1]$ uniformly at random. 
    \FOR{each $1 \leq \mathtt k \leq t$}
    \STATE For each $[H] \in \mathcal J$, compute $\mathfrak L_{H}(\bm X,\bm Y,\xi_{\mathtt k})$ via Algorithm~\ref{alg:dynamic-programming}. 
    \ENDFOR
    \STATE Compute $\widetilde{\Phi}^{\mathcal J}_{u,v}(\bm X,\bm Y)$ according to \eqref{eq-def-widetilde-Phi-H}.
    \STATE {\bf Output:} $\{ \widetilde{\Phi}^{\mathcal J}_{u,v}(\bm X,\bm Y): 1 \leq u,v \leq n \}$.
    \end{algorithmic}
\end{breakablealgorithm}

\begin{proposition}{\label{prop-running-time-recovery}}
    Suppose we choose $\ell$ according to \eqref{eq-condition-weak-recovery}, Algorithm~\ref{alg:cal-widetilde-Phi} computes $\{ \widetilde{\Phi}^{\mathcal J}_{u,v}(\bm X,\bm Y): 1 \leq u,v \leq n \}$ in time $O(n^{T+o(1)})$. 
\end{proposition}
\begin{proof}
    Listing all unlabeled decorated paths in $\mathcal J$ takes time $2^{\ell}=O(n^{T})$. Calculating $\mathsf{Aut}(H)$ for $[H]\in \mathcal J$ takes time $|\mathcal H| \cdot O(\ell) = O(n^{T+o(1)})$. For all $1 \leq \mathtt k \leq t$, calculating $\mathfrak L_{H}(\bm X,\bm Y,\xi_{\mathtt k})$ takes time $n^{T+o(1)}$. Thus, the total running time of Algorithm~\ref{alg:cal-widetilde-f} is 
    \begin{equation*}
        O(n^{T}) + O(n^{T+o(1)}) + 2t \cdot O(n^{T+o(1)}) = O(n^{T+o(1)}) \,. \qedhere
    \end{equation*}
\end{proof}
We finish this subsection by pointing out that Theorem~\ref{MAIN-THM-recovery} directly follows from Proposition~\ref{prop-approximate-recovery-statistics} and Proposition~\ref{prop-running-time-recovery}.

\subsubsection{The correlated spiked Wishart model}{\label{subsec:approx-recovery-cov}}

We similarly generate a random coloring $\xi:[n],[N] \to [2\ell+1]$ that assigns a color to each vertex of $[n],[N]$ from the color set $2\ell+1$ independently and uniformly at random. Given any $V=(V^{\mathsf a},V^{\mathsf b}) \subset [n] \times [N]$, let $\chi_{\xi}(V)$ indicate that $\xi(V)$ is colorful, i.e., $\xi(x) \neq \xi(y)$ for any distinct $x,y \in V$. In particular, if $|V|=2\ell+1$, then $\chi_{\xi}(V)=1$ with probability
\begin{equation}{\label{eq-def-varsigma-cov}}
    \varkappa := \frac{ (2\ell+1)! }{ (2\ell+1)^{2\ell+1} } \,.
\end{equation}
For any unlabeled decorated graph $[H] \in \mathcal I$ with $2\ell+1$ vertices, we define
\begin{equation}{\label{eq-def-g-J-xi-cov}}
    \mathfrak R_{H}(\bm X,\bm Y,\xi) := \sum_{ \substack{ S \subset \mathsf K_n: S \cong H \\ \mathsf{L}(S)=\{ u,v \} } } \chi_{\xi}(V(S)) \prod_{(i,j) \in E_{\bullet}(S)} \bm X_{i,j} \prod_{(i,j) \in E_{\circ}(S)} \bm Y_{i,j} \,.
\end{equation}
Then 
\begin{equation}{\label{eq-averaging-over-coloring-recovery-cov}}
    \mathbb E\big[ \mathfrak R_{H}(\bm X,\bm Y,\xi) \mid \bm X,\bm Y \big] = \varkappa \sum_{ \substack{ S \subset \mathsf K_n: S \cong H \\ \mathsf{L}(S)=\{ u,v \} } } h_S(\bm X,\bm Y)
\end{equation}
where $f_S(\bm X,\bm Y)$ was defined in \eqref{eq-def-f-S-cov}. Next, we define $t=\lceil \frac{1}{\varkappa} \rceil$ and generate $t$ random colorings $ \{ \xi_{\mathtt k} \}_{\mathtt k=1}^{t}$ which are independent copies of $\xi$ that map $[n],[N]$ to $[2\ell+1]$. Then, we define
\begin{equation}{\label{eq-def-widetilde-Phi-mathcal-J}}
    \widetilde{\Phi}^{\mathcal I}_{u,v} := \frac{1}{ N^{\ell} n^{-1} \beta_{\mathcal I} } \sum_{[H] \in \mathcal I} \Upsilon(H) \Big( \frac{1}{t\varkappa} \sum_{\mathtt k=1}^{t} \mathfrak R_H(\bm X,\bm Y,\xi_{\mathtt k}) \Big) \,.
\end{equation}
Moreover, the following result bounds the approximation error under
the same conditions as those in Proposition~\ref{main-prop-recovery-cov} for the second moment calculation.

\begin{proposition}{\label{prop-approximate-recovery-statistics-cov}}
    Suppose Assumption~\ref{assum-upper-bound} and Equation~\eqref{eq-prelim-assumption-upper-bound} hold and we choose $\ell$ according to \eqref{eq-condition-weak-recovery}. Then there exists $R=O_{\lambda,\mu,\rho,\iota}(1)$ such that
    \begin{equation}{\label{eq-approximate-recovery-statistics-cov}}
        \mathbb E_{\overline\Pb}\Big[ \widetilde{\Phi}_{u,v}^{\mathcal I} \cdot x_u x_v \Big] = 1+o(1) \mbox{ and } \mathbb E_{\overline\Pb}\Big[ \big(\widetilde{\Phi}_{u,v}^{\mathcal I}\big)^2 \Big], \ \mathbb E_{\overline\Pb}\Big[ \big(\widetilde{\Phi}_{u,v}^{\mathcal I}\big)^2 x_u^2 x_v^2 \Big] \leq R \,.
    \end{equation}
\end{proposition}

The proof of Proposition~\ref{prop-approximate-recovery-statistics-cov} is postponed to Section~\ref{subsec:proof-prop-4.10} of the appendix. Now we show the approximate subgraph counts $\mathfrak R_{H}(\bm X,\bm Y,\xi)$ can be computed efficiently via the following algorithm. And (similar as in Section~\ref{subsubsec:approx-detection-cov}) the key is to show the following lemma.

\begin{lemma}{\label{lem-color-coding-recovery-cov}}
    For any coloring $\xi:[n] \to [2\ell+1]$ and any hypergraph $[H] \in \mathcal I$, there exists an algorithm (see Algorithm~\ref{alg:dynamic-programming}) that computes $\mathfrak R_H(\bm X,\bm Y,\xi)$ in time $O(n^{T+o(1)})$ for some constant $T=O(1) \cdot (\tfrac{\ell}{\log n}+1 )$.
\end{lemma}

The proof of Lemma~\ref{lem-color-coding-recovery-cov} is postponed to Section~\ref{subsec:proof-lem-4.11} of the appendix. 
Based on Lemma~\ref{lem-color-coding-recovery-cov}, we show that the approximate test statistic $\widetilde{\Phi}$ can be computed efficiently.
\begin{breakablealgorithm}{\label{alg:cal-widetilde-Phi-mathcal-I}}
    \caption{Computation of $\widetilde{\Phi}^{\mathcal I}_{u,v}$}
    \begin{algorithmic}[1]
    \STATE {\bf Input:} Symmetric matrices $\bm X,\bm Y$, signal-to-noise ratio $\lambda,\mu$, correlation $\rho$, and integer parameter $\ell$.
    \STATE List all unlabeled decorated paths in $\mathcal I=\mathcal I(\ell)$. 
    \STATE For each $[H] \in \mathcal I$, compute $\operatorname{Aut}(H)$. 
    \STATE Generate i.i.d.\ random colorings $\{ \xi_{\mathtt k} : 1 \leq \mathtt k \leq t \}$ that map $[n],[N]$ to $[2\ell+1]$ uniformly at random. 
    \FOR{each $1 \leq \mathtt k \leq t$}
    \STATE For each $[H] \in \mathcal I$, compute $\mathfrak R_{H}(\bm X,\bm Y,\xi_{\mathtt k})$ via Algorithm~\ref{alg:dynamic-programming-recovery-cov}. 
    \ENDFOR
    \STATE Compute $\widetilde{\Phi}^{\mathcal I}_{u,v}(\bm X,\bm Y)$ according to \eqref{eq-def-widetilde-Phi-mathcal-J}.
    \STATE {\bf Output:} $\{ \widetilde{\Phi}^{\mathcal I}_{u,v}(\bm X,\bm Y): 1 \leq u,v \leq n \}$.
    \end{algorithmic}
\end{breakablealgorithm}

\begin{proposition}{\label{prop-running-time-recovery-cov}}
    Suppose we choose $\ell$ according to \eqref{eq-condition-weak-recovery}, Algorithm~\ref{alg:cal-widetilde-Phi-mathcal-I} computes $\{ \widetilde{\Phi}^{\mathcal I}_{u,v}(\bm X,\bm Y): 1 \leq u,v \leq n \}$ in time $O(n^{T+o(1)})$. 
\end{proposition}
\begin{proof}
    Listing all unlabeled decorated paths in $\mathcal I$ takes time $2^{\ell}=O(n^{T})$. Calculating $\mathsf{Aut}(H)$ for $[H]\in \mathcal I$ takes time $|\mathcal I| \cdot O(\ell) = O(n^{T+o(1)})$. For all $1 \leq \mathtt k \leq t$, calculating $\mathfrak R_{H}(\bm X,\bm Y,\xi_{\mathtt k})$ takes time $n^{T+o(1)}$. Thus, the total running time of Algorithm~\ref{alg:cal-widetilde-f} is 
    \begin{equation*}
        O(n^{T}) + O(n^{T+o(1)}) + 2t \cdot O(n^{T+o(1)}) = O(n^{T+o(1)}) \,. \qedhere
    \end{equation*}
\end{proof}

We finish this subsection by pointing out that Theorem~\ref{MAIN-THM-recovery-algorithmic-cov} directly follows from Proposition~\ref{prop-approximate-recovery-statistics-cov} and Proposition~\ref{prop-running-time-recovery-cov}.

\section{Low-degree hardness in the subcritical regime}{\label{sec:lower-bound}}

In this section we formally state and prove the results in Theorem~\ref{MAIN-THM-lower-bound-informal}. Inspired by \cite{KWB22}, when studying the correlated spiked Wigner model \eqref{eq-def-correlated-spike-specific} in this section we will instead focus on its slightly modified variant:
\begin{equation}{\label{eq-def-correlated-spiked-model-modify}}
    \widehat{\bm X} = \tfrac{\lambda}{\sqrt{2n}} xx^{\top} + \widehat{\bm W}, \quad \widehat{\bm Y} = \tfrac{\mu}{\sqrt{2n}} yy^{\top} + \widehat{\bm Z} \,,
\end{equation}
where $(x,y) \sim \pi$ for some prior $\pi$ and $\{ \widehat{\bm W}_{i,j}, \widehat{\bm Z}_{i,j}: 1 \leq i,j \leq n \}$ are i.i.d.\ standard normal variables. We let $\widehat{\Pb}$ be the law of $(\widehat{\bm X},\widehat{\bm Y})$ under \eqref{eq-def-correlated-spiked-model-modify} and let $\widehat{\Qb}$ be the law of $(\widehat{\bm W},\widehat{\bm Z})$. It is clear that the detection/recovery problem with respect to $(\widehat{\bm X},\widehat{\bm Y})$ are \emph{computationally easier than} the detection/recovery problem with respect to $(\bm X,\bm Y)$, since we have\footnote{In fact, we can use similar arguments as in \cite[Section~A.2]{KWB22} to show these two models are computationally equivalent.}
\begin{align*}
    \Big( \tfrac{1}{\sqrt{2}}( \widehat{\bm X}+\widehat{\bm X}^{\top} ), \tfrac{1}{\sqrt{2}}( \widehat{\bm Y}+\widehat{\bm Y}^{\top} ) \big) \overset{d}{=} \big( \bm X,\bm Y \big) \,.
\end{align*}
Thus, we only need to show the hardness of the detection/recovery problem in the modified model \eqref{eq-def-correlated-spiked-model-modify}. Now we state our assumptions of the prior $\pi$ in the upper bound, which will be assumed throughout the rest of this section.
\begin{assum}{\label{assum-lower-bound}}
    Suppose that $\pi=\pi_*^{\otimes n}$, where $\pi_*$ is the law of a pair of correlated variables $(\mathbb X,\mathbb Y)$ satisfying the following conditions: 
    \begin{enumerate}
        \item[(1)] $\mathbb E[\mathbb X]=\mathbb E[\mathbb Y]=0$, $\mathbb E[\mathbb X^2]=\mathbb E[\mathbb Y^2]=1$, and $\mathbb E[\mathbb X\mathbb Y]=\rho^2$. 
        \item[(2)] $\mathbb X$ and $\mathbb Y$ are $C'$-subgaussian for some constant $C'$.
        \item[(3)] Suppose that $(\mathbb X,\mathbb Y),(\mathbb X',\mathbb Y') \sim \pi$ with $(\mathbb X,\mathbb Y) \perp (\mathbb X',\mathbb Y')$. Then 
        \begin{align*}
            \mathbb E\big[ (\mathbb X\mathbb X')^a (\mathbb Y\mathbb Y')^b \big] = \mathbb E\big[ \mathbb X^a \mathbb Y^b \big]^2 \geq \mathbb E\big[ \mathbb X^a \big]^2 \mathbb E\big[ \mathbb Y^b \big]^2
        \end{align*}
        for all $a,b \in \mathbb N$.
    \end{enumerate}
    Note that condition~(1) above is consistent with \eqref{eq-moment-pi}.
\end{assum}
\begin{remark}
    We remark that Items~(1) and (2) are standard assumptions for proving low-degree lower bound on PCA problems (see, e.g., \cite{KWB22} and \cite{DK25+}). In contrast, Item~(3) is a technical assumption which will be useful in our interpolation arguments. Nevertheless, it is straightforward to check that our assumption still captures some most canonical priors $\pi$, including the correlated Gaussian prior, the correlated Rademacher prior, and the correlated sparse Rademacher prior introduced in Section~\ref{sec:intro}. We do believe that it is possible to replace Item~(3) with some mild moment conditions, although it would require a more careful analysis in our interpolation argument.
\end{remark}

The rest of this section is organized as follows. In Section~\ref{subsec:LDP-framework}, we introduce the low-degree polynomial framework and rigorously state our computational lower bound (see Theorems~\ref{Main-thm-lower-bound} and \ref{Main-thm-lower-bound-covariance}). In Sections~\ref{subsec:add-Gaussian-model} and \ref{subsec:cov-Gaussian-model}, we use the known results in Gaussian additive model and covariance model to derive a explicit formula of the low-degree advantage. In Sections~\ref{subsec:proof-thm-5.4} and \ref{subsec:proof-thm-5.5} we present the proof of Theorems~\ref{Main-thm-lower-bound} and \ref{Main-thm-lower-bound-covariance}, respectively.

\subsection{Low-degree polynomial framework}{\label{subsec:LDP-framework}}

Inspired by the sum-of-squares hierarchy, the low-degree polynomial method offers a promising framework for establishing computational lower bounds in high-dimensional inference problems \cite{BHK+19, HS17, HKP+17, Hop18}. Roughly speaking, to study the computational efficiency of the hypothesis testing problem between two probability measures $\Pb$ and $\Qb$, the idea of this framework is to find a low-degree polynomial $f$ (usually with degree $O(\log n)$ where $n$ is the size of data) that separates $\Pb$ and $\Qb$ in a certain sense. More precisely, let $\mathbb R[\bm X,\bm Y]_{\leq D}$ denote the set of multivariate polynomials in the entries of $(\bm X,\bm Y)$ with degree at most $D$. With a slight abuse of notation, we will often say ``a polynomial'' to mean a sequence of polynomials $f=f_n \in \mathbb R[\bm X,\bm Y]_{\leq D}$, one for each problem size $n$; the degree $D=D_n$ of such a polynomial may scale with $n$. To probe the computational threshold for testing between two sequences of probability measures $\Pb$ and $\Qb$, we consider the following notions of strong separation and weak separation defined in \cite[Definition~1.6]{BEH+22}.

\begin{DEF}{\label{def-strong-separation}}
    Let $f \in \mathbb{R}[\bm X,\bm Y]_{\leq D}$ be a polynomial.
    \begin{itemize}
        \item We say $f$ strongly separates $\Pb_n$ and $\Qb_n$ if as $n \to \infty$ 
        \begin{align*}
            \sqrt{ \max\big\{ \operatorname{Var}_{\Pb}(f({\bm X},{\bm Y})), \operatorname{Var}_{\Qb}(f({\bm X},{\bm Y})) \big\} } = o\big( \big| \mathbb{E}_{\Pb}[f({\bm X},{\bm Y})] - \mathbb{E}_{\Qb}[f({\bm X},{\bm Y})] \big| \big) \,;
        \end{align*}
        \item We say $f$ weakly separates $\Pb_n$ and $\Qb_n$ if as $n \to \infty$ 
        \begin{align*}
            \sqrt{ \max\big\{ \operatorname{Var}_{\Pb}(f({\bm X},{\bm Y})), \operatorname{Var}_{\Qb}(f({\bm X},{\bm Y})) \big\} } = O\big( \big| \mathbb{E}_{\Pb}[f({\bm X},{\bm Y})] - \mathbb{E}_{\Qb}[f({\bm X},{\bm Y})] \big| \big) \,.
        \end{align*}
    \end{itemize}
\end{DEF}
See \cite{BEH+22} for a detailed discussion of why these conditions are natural for hypothesis testing. In particular, according to Chebyshev's inequality, strong separation implies that we can threshold $f(A,B)$ to test $\Pb$ against $\Qb$ with vanishing type-I and type-II errors. The key quantity of showing the impossibility of separation is the low-degree advantage
\begin{equation}{\label{eq-def-low-deg-adv}}
    \chi^2_{\leq D}(\Pb;\Qb) := \sup_{ \operatorname{deg}(f) \leq D } \Bigg\{ \frac{ \mathbb E_{\Pb}[f] }{ \sqrt{\mathbb E_{\Qb}[f^2]} } \Bigg\} \,.
\end{equation}
(If we remove the restriction to low-degree polynomials, then this expression gives the classical chi-square divergence $\chi^2(\Pb;\Qb)$ between $\Pb$ and $\Qb$.)
The low-degree polynomial method suggests that, if $\chi^2_{\leq D}(\Pb;\Qb )=O(1)$, then no algorithm with running time $e^{\widetilde{\Theta}(D)}$ can strongly distinguish $\Pb$ and $\Qb$ (here $\widetilde{\Theta}(\cdot)$ means having the same order up to poly-logarithmic factors).\footnote{To prevent the readers from being overly optimistic for this conjecture, we point out a recent work \cite{BHJK25} that finds a counterexample for this conjecture. We therefore clarify that the low-degree framework is expected to be optimal for a certain, yet imprecisely defined, class of ``high-dimensional'' problems. Despite these important caveats, we still believe that analyzing low-degree polynomials remains highly meaningful for our setting, as it provides a benchmark for robust algorithmic performance. We refer the reader to the survey \cite{Wein25+} for a more detailed discussion of these subtleties.} The appeal of this framework lies in its ability to yield tight computational lower bounds for a wide range of problems, including detection and recovery problems such as planted clique, planted dense subgraph, community detection, sparse-PCA and graph alignment (see \cite{HS17, HKP+17, Hop18, KWB22, Li25+, SW22, DMW23+, MW25, DKW22, BEH+22, DDL23+, CDGL24+}), optimization problems such as maximal independent sets in sparse random graphs \cite{GJW24, Wein22}, refutation problems such as certification of RIP and lift-monotone properties of random regular graphs \cite{WBP16, BKW20, KY24}, and constraint satisfaction problems such as random $k$-SAT \cite{BH22}.

Using this framework, we can now state our computational lower bound as follows.

\begin{thm}{\label{Main-thm-lower-bound}}
    Suppose Assumption~\ref{assum-lower-bound} is satisfied and $F(\lambda,\mu,\rho,1)<1-\epsilon$ for some constant $\epsilon>0$. Then for $\widehat\Pb,\widehat\Qb$ defined in \eqref{eq-def-correlated-spiked-model-modify}, we have $\chi^2_{\leq D}(\widehat\Pb;\widehat\Qb)= O_{\lambda,\mu,\rho}(1)$ for all $D=n^{o(1)}$. In particular, no degree-$D$ polynomial can strongly separate $\widehat\Pb$ and $\widehat\Qb$ when $D=n^{o(1)}$.
\end{thm}
\begin{thm}{\label{Main-thm-lower-bound-covariance}}
    Suppose Assumption~\ref{assum-lower-bound} is satisfied and $F(\lambda,\mu,\rho,\gamma)<1-\epsilon$ for some constant $\epsilon>0$, where $\gamma=\tfrac{n}{N}$. Then for $\overline\Pb,\overline\Qb$ defined in Definition~\ref{def-correlated-spike-covariance-specific}, we have $\chi^2_{\leq D}(\overline\Pb;\overline\Qb)= O_{\lambda,\mu,\rho,\gamma}(1)$ for all $D=n^{o(1)}$. In particular, no degree-$D$ polynomial can strongly separate $\overline\Pb$ and $\overline\Qb$ when $D=n^{o(1)}$.
\end{thm}

\subsection{Additive Gaussian model}{\label{subsec:add-Gaussian-model}}

It turns out that convenient and elegant closed forms for the low-degree advantage is known for the special case of models with additive Gaussian noise. Let us first state a general result for a general such model.

\begin{DEF}[Additive Gaussian noise model]{\label{def-add-Gaussian}}
    Let $X \sim \mathcal X=\mathcal X_M$ for some probability measure $\mathcal X$ over $\mathbb R^M$ and $Z \sim \mathcal N(0,\Sigma)$ for some $\Sigma \in \mathbb R^{M*M}$ with $\Sigma\succ 0$. The associated additive Gaussian model consists of the probability measures $\Pb=\Pb_M$ and $\Qb=\Qb_M$ given by:
    \begin{itemize}
        \item To sample $Y \sim \Pb$, observe $Y=X+Z$.
        \item To sample $Y \sim \Qb$, observe $Y=Z$.
    \end{itemize}
\end{DEF}
\begin{proposition}{\label{prop-low-deg-Adv-add-Gaussian}}
    Suppose in the above setting $\Sigma=\mathbb I_M$. Then
    \begin{equation}{\label{eq-chi-square-add-Gaussian}}
        \chi^2(\Pb;\Qb) = \mathbb E_{ X_1,X_2 \sim \mathcal X, X_1 \perp X_2 } \Big[ \exp( \langle X_1,X_2 \rangle ) \Big]
    \end{equation}
    and 
    \begin{equation}{\label{eq-low-deg-Adv-add-Gaussian}}
        \chi^2_{\leq D}(\Pb;\Qb) = \mathbb E_{ X_1,X_2 \sim \mathcal X, X_1 \perp X_2 } \Big[ \exp_{\leq D}( \langle X_1,X_2 \rangle ) \Big] \,,
    \end{equation}
    where we define 
    \begin{equation}{\label{eq-exp-leq-D}}
        \exp_{\leq D}(x) = \sum_{k=0}^{D} \frac{ x^k }{ k! } \,.
    \end{equation}
\end{proposition}

We now apply these results to our matrix model \eqref{eq-def-correlated-spiked-model-modify}. By taking $X=(\tfrac{\lambda}{\sqrt{2n}} uu^{\top}, \tfrac{\mu}{\sqrt{2n}} vv^{\top})$ and $Z=(\widehat{\bm W}, \widehat{\bm Z})$ in Proposition~\ref{prop-low-deg-Adv-add-Gaussian}, we get that
\begin{align}
    \chi^2_{\leq D}(\widehat\Pb;\widehat\Qb) &= \mathbb E_{ (x,y),(x',y') \sim \pi }\Big[ \exp_{\leq D}\big( \langle \tfrac{\lambda}{\sqrt{2n}} xx^{\top}, \tfrac{\lambda}{\sqrt{2n}} x'(x')^{\top} \rangle + \langle \tfrac{\mu}{\sqrt{2n}} yy^{\top}, \tfrac{\mu}{\sqrt{2n}} y'(y')^{\top} \rangle \big) \Big] \nonumber \\
    &= \mathbb E_{ (x,y),(x',y') \sim \pi }\Big[ \exp_{\leq D}\big( \tfrac{\lambda^2}{2n} \langle x,x' \rangle^2 + \tfrac{\mu^2}{2n} \langle y,y' \rangle^2 \big) \Big] \,. \label{eq-low-deg-Adv-relax-1}
\end{align}

\subsection{Proof of Theorem~\ref{Main-thm-lower-bound}}{\label{subsec:proof-thm-5.4}}

Now we provide the proof of Theorem~\ref{Main-thm-lower-bound}. Clearly, it suffices to show that the right hand side of \eqref{eq-low-deg-Adv-relax-1} is bounded by $O_{\lambda,\mu,\rho,\epsilon}(1)$. Recall Assumption~\ref{assum-lower-bound}. Denote $(X,Y)$ be the law of $(\mathbb X \mathbb X',\mathbb Y \mathbb Y')$ for $(\mathbb X,\mathbb Y),(\mathbb X',\mathbb Y') \sim \pi_*$ and $(\mathbb X,\mathbb Y) \perp (\mathbb X',\mathbb Y')$. Note that if we sample $(X_1,Y_1),\ldots,(X_n,Y_n) \overset{i.i.d.}{\sim} (X,Y)$, then
\begin{align}
    \eqref{eq-low-deg-Adv-relax-1} &= \mathbb E\Bigg\{ \exp_{\leq D} \Bigg( \frac{\lambda^2}{2n} \Big( \sum_{j=1}^{n} X_j \Big)^2 + \frac{\mu^2}{2n} \Big( \sum_{j=1}^{n} Y_j \Big)^2 \Bigg) \Bigg\} \nonumber \\
    &\overset{\eqref{eq-exp-leq-D}}{=} \sum_{k=0}^{D} \frac{1}{2^k k!} \sum_{\ell=0}^{k} \binom{k}{\ell} \lambda^{2\ell} \mu^{2(k-\ell)}  \mathbb E\Bigg\{ \Big( \frac{ \sum_{j=1}^{n} X_j }{ \sqrt{n} }  \Big)^{2\ell} \Big( \frac{ \sum_{j=1}^{n} Y_j }{ \sqrt{n} } \Big)^{2(k-\ell)} \Bigg\} \nonumber \\
    &\leq \sum_{k,\ell=0}^{D} \frac{ \lambda^{2k} \mu^{2\ell} }{ 2^{k+\ell} k!\ell! } \mathbb E\Bigg\{ \Big( \frac{ \sum_{j=1}^{n} X_j }{ \sqrt{n} } \Big)^{2k} \Big( \frac{ \sum_{j=1}^{n} Y_j }{ \sqrt{n} } \Big)^{2\ell} \Bigg\} \,.  \label{eq-low-def-Adv-relax-2}
\end{align}
The basic intuition behind our approach of bounding \eqref{eq-low-def-Adv-relax-2} is that according to Assumption~\ref{assum-lower-bound} and central limit theorem, we should expect that 
\begin{align*}
    \Big( \frac{ \sum_{j=1}^{n} X_j }{ \sqrt{n} }, \frac{ \sum_{j=1}^{n} Y_j }{ \sqrt{n} } \Big) \mbox{ behaves like } (U,V) \,,
\end{align*}
where $(U,V)$ is a pair of standard normal variable with correlation $\rho^2$. Thus, the first step is to instead bound the joint moment of $U,V$, which is incorporated in the following lemma.
\begin{lemma}{\label{lem-correlated-Gaussian-moment}}
    Suppose $(\zeta_1,\eta_1), \ldots, (\zeta_n,\eta_n)$ are i.i.d.\ pair of standard normal variables with $\mathbb E[\zeta_i \eta_i]=\rho^2$. In addition, suppose that 
    \begin{equation}{\label{eq-condition-lower-bound}}
        \frac{ \lambda^2 \rho^2 }{ 2(1-\lambda^2+\lambda^2 \rho^2) } + \frac{ \mu^2 \rho^2 }{ 2(1-\mu^2+\mu^2 \rho^2) } < \frac{1-\epsilon}{2}
    \end{equation}
    holds. Then there exists a constant $\delta=\delta(\epsilon,\lambda,\mu,\rho)>0$ such that
    \begin{align*}
        \frac{ \lambda^{2k} \mu^{2\ell} }{ 2^{k+\ell}k!\ell! } \mathbb E\Bigg\{ \Big( \frac{ \sum_{j=1}^{n} \zeta_j }{ \sqrt{n} }  \Big)^{2k} \Big( \frac{ \sum_{j=1}^{n} \eta_j }{ \sqrt{n} } \Big)^{2\ell} \Bigg\} \leq O_{\epsilon}(1) \cdot (1+\delta)^{-2(k+\ell)} \,.
    \end{align*}
\end{lemma}

The proof of Lemma~\ref{lem-correlated-Gaussian-moment} is postponed to Section~\ref{subsec:proof-lem-5.8} of the appendix. Based on Lemma~\ref{lem-correlated-Gaussian-moment}, it suffices to show that
\begin{align*}
    \frac{ \lambda^{2k} \mu^{2\ell} }{ 2^{k+\ell}k!\ell! } \mathbb E\Bigg\{ \Big( \frac{ \sum_{j=1}^{n} \zeta_j }{ \sqrt{n} } \Big)^{2k} \Big( \frac{ \sum_{j=1}^{n} \eta_j }{ \sqrt{n} } \Big)^{2\ell} \Bigg\} \mbox{ and } \frac{ \lambda^{2k} \mu^{2\ell} }{ 2^{k+\ell}k!\ell! } \mathbb E\Bigg\{ \Big( \frac{ \sum_{j=1}^{n} X_j }{ \sqrt{n} } \Big)^{2k} \Big( \frac{ \sum_{j=1}^{n} Y_j }{ \sqrt{n} } \Big)^{2\ell} \Bigg\} \,.
\end{align*}
are ``close'' in a certain sense, as incorporated in the following lemma.

\begin{lemma}{\label{lem-compare-moments-Gaussian-non-Gaussian}}
    When $k,\ell=n^{o(1)}$, we have
    \begin{align*}
        & \frac{ \lambda^{2k} \mu^{2\ell} }{ 2^{k+\ell}k!\ell! } \Bigg| \mathbb E\Bigg\{ \Big( \frac{ \sum_{j=1}^{n} \zeta_j }{ \sqrt{n} } \Big)^{2k} \Big( \frac{ \sum_{j=1}^{n} \eta_j }{ \sqrt{n} } \Big)^{2\ell} - \Big( \frac{ \sum_{j=1}^{n} X_j }{ \sqrt{n} } \Big)^{2k} \Big( \frac{ \sum_{j=1}^{n} Y_j }{ \sqrt{n} } \Big)^{2\ell} \Bigg\} \Bigg| \\
        =\ & [n^{-\frac{1}{2}+o(1)}] \cdot \frac{ \lambda^{2k} \mu^{2\ell} }{ 2^{k+\ell}k!\ell! } \mathbb E\Bigg\{ \Big( \frac{ \sum_{j=1}^{n} \zeta_j }{ \sqrt{n} } \Big)^{2k} \Big( \frac{ \sum_{j=1}^{n} \eta_j }{ \sqrt{n} } \Big)^{2\ell} \Bigg\} \,.
    \end{align*}
\end{lemma}
The proof of Lemma~\ref{lem-compare-moments-Gaussian-non-Gaussian} is postponed to Section~\ref{subsec:proof-lem-5.9} of the appendix. We finish this section by pointing out that plugging Lemma~\ref{lem-correlated-Gaussian-moment} and Lemma~\ref{lem-compare-moments-Gaussian-non-Gaussian} into \eqref{eq-low-def-Adv-relax-2} yields Theorem~\ref{Main-thm-lower-bound}.

\subsection{Gaussian covariance model}{\label{subsec:cov-Gaussian-model}}

In this subsection we will derive a tractable bound for $\chi^2_{\leq D}(\overline\Pb;\overline\Qb)$ when $\overline\Pb,\overline\Qb$ is defined in Definition~\ref{def-correlated-spike-covariance-specific}. Although it is possible to again use the result for Gaussian additive models, here we adopt an alternative method motivated by \cite{BKW20} which yields a somewhat simpler bound, which only involves the ``replica overlap'' $(\langle x,x' \rangle,\langle y,y' \rangle)$ where $(x,y),(x',y') \sim \pi$ and $(x,y) \perp (x',y')$.

\begin{lemma}{\label{lem-bound-Adv-cov-relax-1}}
    Suppose $\overline\Pb,\overline\Qb$ is defined in Definition~\ref{def-correlated-spike-covariance-specific}. Then
    \begin{equation}{\label{eq-bound-Adv-cov-relax-1}}
        \chi^2_{\leq D}(\overline\Pb;\overline\Qb) \leq \mathbb E_{(x,y),(x',y')\sim\pi}\Big[ \varphi_{D}\big( \tfrac{\lambda^2 \langle x,x' \rangle^2 }{4n^2} \big) \varphi_{D}\big( \tfrac{\mu^2 \langle y,y' \rangle^2 }{4n^2} \big) \Big] \,.
    \end{equation}
    Here
    \begin{equation}{\label{eq-def-varphi-leq-D}}
        \varphi_D(t) := \sum_{0 \leq k \leq D} t^k \sum_{ k_1+\ldots+k_N=k } \prod_{1 \leq i \leq N} \binom{2k_i}{k_i} \,.
    \end{equation}
\end{lemma}

The proof of Lemma~\ref{lem-bound-Adv-cov-relax-1} is postponed to Section~\ref{subsec:proof-lem-5.10} of the appendix.

\subsection{Proof of Theorem~\ref{Main-thm-lower-bound-covariance}}{\label{subsec:proof-thm-5.5}}

In this subsection we give the proof of Theorem~\ref{Main-thm-lower-bound-covariance}. Clear, based on Lemma~\ref{lem-Adv-cov-transform}, it suffices to show the right hand side of \eqref{eq-Adv-cov-transform} is bounded by $O_{\lambda,\mu,\rho,\gamma}(1)$. To this end, we first show the following bound on the coefficient of $\varphi_{D}$.

\begin{lemma}{\label{lem-bound-coeff-varphi-D}}
    For $N=\Theta(n)$ and $k=n^{o(1)}$, we have
    \begin{align*}
        \sum_{ k_1+\ldots+k_N=k } \prod_{1 \leq i \leq N} \binom{2k_i}{k_i} \leq \frac{ [1+o(1)] 2^k N^k }{ k! } \,.
    \end{align*}
\end{lemma}
\begin{proof}
    Denote $\mathtt I=\{ 1 \leq i \leq N: k_i \not\in\{ 0,1 \} \}$ and denote $\mathtt S=\sum_{i \in \mathtt I} k_i \geq 2|\mathtt I|$. Clearly, for any fixed $\mathtt I$ and $\mathtt S \geq 2 |\mathtt I|$, we have
    \begin{align*}
        \prod_{1 \leq i \leq N} \binom{2k_i}{k_i} = \prod_{ i \in \mathtt I } \binom{2k_i}{k_i} \prod_{ i \not\in \mathtt I } \binom{2k_i}{k_i} \leq 2^{ \#\{ i \not\in \mathtt I:k_i =1 \} } \prod_{ i \in \mathtt I } 2^{2k_i} = 2^{k-\mathtt S} \cdot 2^{2\mathtt S} = 2^{k+\mathtt S} 
    \end{align*}
    and
    \begin{align*}
        \#\Big\{ \{ k_i:i \in \mathtt I \}: \sum_{ i \in \mathtt I } k_i = \mathtt S \Big\} \leq \mathtt I^{\mathtt S}, \quad \#\Big\{ \{ k_i:i \not\in \mathtt I \}: \sum_{ i \not\in \mathtt I } k_i = k-\mathtt S \Big\}=\binom{N-|\mathtt I|}{k-\mathtt S} \leq \frac{ N^{k-\mathtt S} k^{\mathtt S} }{ k! } \,.
    \end{align*}
    Thus, we have
    \begin{align*}
        \sum_{ k_1+\ldots+k_N=k } \prod_{1 \leq i \leq N} \binom{2k_i}{k_i} &\leq \sum_{ \mathtt I \subset [N], |\mathtt I|\leq k } \sum_{ 2 |\mathtt I| \leq \mathtt S \leq k } \frac{ 2^{k+\mathtt S} \cdot \mathtt I^{\mathtt S} \cdot N^{k-\mathtt S} k^{\mathtt S} }{ k! } \\
        &\leq \frac{ 2^k N^k }{ k! } \sum_{ \mathtt I \subset [N], |\mathtt I|\leq k } \sum_{ 2 |\mathtt I| \leq \mathtt S \leq k } N^{-\mathtt S} 2^{\mathtt S} |\mathtt I|^{\mathtt S} k^{\mathtt S} \\
        &= \frac{ [1+o(1)] 2^k N^k }{ k! } \sum_{ \mathtt I \subset [N], |\mathtt I|\leq k } N^{-2|\mathtt I|} 2^{2|\mathtt I|} |\mathtt I|^{2|\mathtt I|} k^{2|\mathtt I|} \\
        &= \frac{ [1+o(1)] 2^k N^k }{ k! } \sum_{ 0 \leq t \leq k } \binom{N}{t} \cdot N^{-2t} 2^{2t} t^{2t} k^{2t} = \frac{ [1+o(1)] 2^k N^k }{ k! } \,,
    \end{align*}
    where the last equality follows from $k=N^{o(1)}$.
\end{proof}

Using Lemmas~\ref{lem-Adv-cov-transform} and \ref{lem-bound-coeff-varphi-D}, we get that
\begin{align}
    \chi^2_{\leq D}(\overline\Pb;\overline\Qb) &\leq [1+o(1)] \mathbb E_{(x,y),(x',y')\sim\pi}\Bigg\{ \sum_{ \substack{ 0 \leq k \leq D \\ 0 \leq \ell \leq D } } \frac{2^{k+\ell}N^{k+\ell}}{k!\ell!} \Big( \frac{\lambda^2 \langle x,x' \rangle^2 }{4n^2} \Big)^{k} \Big( \frac{\mu^2 \langle y,y' \rangle^2 }{4n^2} \Big)^{\ell} \Bigg\}  \nonumber  \\
    &= [1+o(1)] \sum_{ \substack{ 0 \leq k \leq D \\ 0 \leq \ell \leq D } } \mathbb E_{(x,y),(x',y')\sim\pi}\Bigg\{ \frac{ \lambda^{2k}\mu^{2\ell} }{ 2^{k+\ell}k!\ell! } \Big( \frac{N}{n} \Big)^{k+\ell} \Big( \frac{ \langle x,x' \rangle }{\sqrt{n}} \Big)^{2k} \Big( \frac{ \langle y,y' \rangle }{\sqrt{n}} \Big)^{2\ell} \Bigg\} \,.
\end{align}
Again, let $(U,V)$ be a pair of standard normal variables with correlation $\rho^2$. Using Lemma~\ref{lem-compare-moments-Gaussian-non-Gaussian}, we have
\begin{align*}
    \mathbb E_{(x,y),(x',y')\sim\pi}\Bigg\{ \Big( \frac{ \langle x,x' \rangle }{\sqrt{n}} \Big)^{2k} \Big( \frac{ \langle y,y' \rangle }{\sqrt{n}} \Big)^{2\ell} \Bigg\} = [1+o(n^{-\frac{1}{2}+o(1)})] \cdot \mathbb E\big[ U^{2k} V^{2\ell} \big] \,.
\end{align*}
Thus, it suffices to show the following lemma.

\begin{lemma}{\label{lem-correlated-Gaussian-cov}}
    Suppose that $F(\lambda,\mu,\rho,\gamma)<1-\epsilon$ for some constant $\epsilon>0$, where $\gamma=\tfrac{n}{N}$. Then
    \begin{equation}{\label{eq-correlated-Gaussian-moment-cov}}
        \sum_{ \substack{ 0 \leq k \leq D \\ 0 \leq \ell \leq D } } \frac{ \lambda^{2k} \mu^{2\ell} \gamma^{k+\ell} }{ 2^{k+\ell}k!\ell! } \mathbb E\Big[ U^{2k} V^{2\ell} \Big] = O_{\lambda,\mu,\rho,\gamma}(1) \,.
    \end{equation}
\end{lemma}
\begin{proof}
    Note that 
    \begin{align*}
        \mathbb E\Bigg[ \sum_{k,\ell=0}^{\infty} \frac{ \lambda^{2k} \mu^{2\ell} \gamma^{k+\ell} }{ 2^{k+\ell}k!\ell! } \cdot U^{2k} V^{2\ell} \Bigg] = \mathbb E\Bigg[ \exp\Big( \frac{\gamma}{2}( \lambda^2 U^2 + \mu^2 V^2 ) \Big) \Bigg] \,.
    \end{align*}
    Similar as the calculation in the proof of Lemma~\ref{lem-correlated-Gaussian-moment}, the above is bounded by $O_{\lambda,\mu,\rho,\gamma}(1)$ when $F(\lambda,\mu,\rho,\gamma)<1-\epsilon$.
\end{proof}

\appendix

\section{Supplementary proofs in Section~\ref{sec:main-alg-results}}{\label{sec:supp-proofs-sec-2}}

\subsection{Proof of Lemma~\ref{lem-bound-beta-mathcal-H}}{\label{subsec:proof-lem-2.2}}

Define $\mathcal W=\mathcal W(\ell)$ to be the set of $\omega=(\omega_1,\ldots,\omega_{\ell}) \in \{ \bullet,\circ \}^{\ell}$. In addition, for $\omega=(\omega_1,\ldots,\omega_{\ell})\in\mathcal W$, we similarly write
\begin{equation}{\label{eq-def-E-i-Dif-string}}
    \begin{aligned} 
        & E_{\bullet}(\omega)=\{ 1\leq i\leq \ell: \omega_i=\bullet \}, \ E_{\circ}(\omega)=\{ 1\leq i\leq \ell: \omega_i=\circ \} \,; \\
        & \mathsf{diff}(\omega)=\#\{ 1 \leq i \leq \ell-1: \omega_i \neq \omega_{i+1} \} \,.
    \end{aligned}
\end{equation}
Define 
\begin{equation}{\label{eq-def-widetilde-mathcal-H}}
    \widetilde{\mathcal H}=\Big\{ (v,[H]): [H] \in \mathcal H, v \in V(H) \Big\} \,.
\end{equation}
Consider the mapping $\varphi:\mathcal H \to \mathcal W$ define as follows: for all $(v,H) \in \widetilde{\mathcal H}$, write $V(H)$ in the counterclockwise order $V(H)=\{ v_1,\ldots,v_{\ell} \}$ with $v_1=v$ and let $\varphi(v,H)=\omega=(\omega_1,\ldots,\omega_\ell)$ with (below we write $v_{\ell+1}=v_1$)
\begin{align*}
    \omega_i = \gamma((v_i,v_{i+1})) = 
    \begin{cases}
        \bullet, & (v_i,v_{i+1}) \in E_{\bullet}(H) \,; \\
        \circ, & (v_i,v_{i+1}) \in E_{\circ}(H) \,.
    \end{cases}
\end{align*}
It is clear that $\varphi$ is a bijection and
\begin{align*}
    |E_{\bullet}(H)|=|E_{\bullet}(\varphi(v,H))|, \ |E_{\circ}(H)|=|E_{\circ}(\varphi(v,H))|, \ |\mathsf{diff}(H)|-1 \leq |\mathsf{diff}(\varphi(v,H))| \leq |\mathsf{diff}(H)| \,.
\end{align*}
In addition, since $1 \leq |\mathsf{Aut}(H)| \leq \ell$ for all $H \in \mathcal H$, we have
\begin{align*}
    \beta_{\mathcal H} &\overset{\eqref{eq-def-beta-mathcal-H}}{=} \sum_{[H] \in \mathcal H} \frac{ \rho^{2|\mathsf{diff}(H)|} \lambda^{2|E_{\bullet}(H)|} \mu^{2|E_{\circ}(H)|} }{ |\mathsf{Aut}(H)| } \leq \sum_{ [H] \in \mathcal H } \rho^{2|\mathsf{diff}(H)|} \lambda^{2|E_{\bullet}(H)|} \mu^{2|E_{\circ}(H)|} \\
    &\leq \sum_{ (v,[H]) \in \widetilde{\mathcal H} } \rho^{2|\mathsf{diff}(H)|} \lambda^{2|E_{\bullet}(H)|} \mu^{2|E_{\circ}(H)|} \leq \sum_{ (v,[H]) \in \widetilde{\mathcal H} } \rho^{2|\mathsf{diff}(\varphi(v,H))|} \lambda^{2|E_{\bullet}(\varphi(v,H))|} \mu^{2|E_{\circ}(\varphi(v,H))|} \\
    &= \sum_{ \omega\in\mathcal W } \rho^{2|\mathsf{diff}(\omega)|} \lambda^{2|E_{\bullet}(\omega)|} \mu^{2|E_{\circ}(\omega)|} 
\end{align*}
and 
\begin{align*}
    \beta_{\mathcal H} &\overset{\eqref{eq-def-beta-mathcal-H}}{=} \sum_{[H] \in \mathcal H} \frac{ \rho^{2|\mathsf{diff}(H)|} \lambda^{2|E_{\bullet}(H)|} \mu^{2|E_{\circ}(H)|} }{ |\mathsf{Aut}(H)| } \geq \frac{1}{\ell} \sum_{ [H] \in \mathcal H } \rho^{2|\mathsf{diff}(H)|} \lambda^{2|E_{\bullet}(H)|} \mu^{2|E_{\circ}(H)|} \\
    &= \frac{1}{\ell^2} \sum_{ (v,[H]) \in \widetilde{\mathcal H} } \rho^{2|\mathsf{diff}(H)|} \lambda^{2|E_{\bullet}(H)|} \mu^{2|E_{\circ}(H)|} \geq \frac{\rho^2}{\ell^2} \sum_{ (v,[H]) \in \widetilde{\mathcal H} } \rho^{2|\mathsf{diff}(\varphi(v,H))|} \lambda^{2|E_{\bullet}(\varphi(v,H))|} \mu^{2|E_{\circ}(\varphi(v,H))|} \\
    &= \frac{\rho^2}{\ell^2} \sum_{ \omega\in\mathcal W } \rho^{2|\mathsf{diff}(\omega)|} \lambda^{2|E_{\bullet}(\omega)|} \mu^{2|E_{\circ}(\omega)|}  \,.
\end{align*}
Thus, it suffices to show that there exists a constant $D=D(\lambda,\mu,\rho,\epsilon)>0$ such that
\begin{equation}{\label{eq-goal-lem-2.1}}
    D^{-1} A_+^{\ell} \leq \sum_{ \omega\in\mathcal W } \rho^{2|\mathsf{diff}(\omega)|} \lambda^{2|E_{\bullet}(\omega)|} \mu^{2|E_{\circ}(\omega)|} \leq D A_+^{\ell} \,.
\end{equation}
We now give an inductive formula of $\sum_{ \omega\in\mathcal W(\ell) } \rho^{2|\mathsf{diff}(\omega)|} \lambda^{2|E_{\bullet}(\omega)|} \mu^{2|E_{\circ}(\omega)|}$. Denote
\begin{equation}{\label{eq-def-mathsf-X-Y}}
\begin{aligned}
    &\mathsf{X}_{\ell} = \sum_{ \omega\in\mathcal W(\ell), \omega_\ell=1 } \rho^{2|\mathsf{diff}(\omega)|} \lambda^{2|E_{\bullet}(\omega)|} \mu^{2|E_{\circ}(\omega)|} \,; \\
    &\mathsf{Y}_{\ell} = \sum_{ \omega\in\mathcal W(\ell), \omega_\ell=2 } \rho^{2|\mathsf{diff}(\omega)|} \lambda^{2|E_{\bullet}(\omega)|} \mu^{2|E_{\circ}(\omega)|} \,.
\end{aligned}
\end{equation}
Since $\omega\in\mathcal W(\ell)$ with $\omega_{\ell}=1$ can be written as $\omega=(\omega',1)$ with $\omega'\in\mathcal W(\ell-1)$, we can check that
\begin{align*}
    \mathsf{X}_{\ell} =\ & \lambda^2 \sum_{ \omega'\in\mathcal W(\ell-1), \omega_{\ell-1}'=1 } \rho^{2|\mathsf{diff}(\omega')|} \lambda^{2|E_{\bullet}(\omega')|} \mu^{2|E_{\circ}(\omega')|} \\
    & + \lambda^2 \rho^2 \sum_{ \omega'\in\mathcal W(\ell-1), \omega_{\ell-1}'=2 } \rho^{2|\mathsf{diff}(\omega')|} \lambda^{2|E_{\bullet}(\omega')|} \mu^{2|E_{\circ}(\omega')|} = \lambda^2 \mathsf{X}_{\ell-1} + \lambda^2 \rho^2 \mathsf{Y}_{\ell-1} \,.
\end{align*}
And similarly
\begin{align*}
    \mathsf{Y}_{\ell} =\ & \mu^2 \sum_{ \omega'\in\mathcal W(\ell-1), \omega_{\ell-1}'=2 } \rho^{2|\mathsf{diff}(\omega')|} \lambda^{2|E_{\bullet}(\omega')|} \mu^{2|E_{\circ}(\omega')|} \\
    & + \mu^2 \rho^2 \sum_{ \omega'\in\mathcal W(\ell-1), \omega_{\ell-1}'=1 } \rho^{2|\mathsf{diff}(\omega')|} \lambda^{2|E_{\bullet}(\omega')|} \mu^{2|E_{\circ}(\omega')|} = \mu^2 \mathsf{Y}_{\ell-1} + \mu^2 \rho^2 \mathsf{X}_{\ell-1} \,.
\end{align*}
Thus, we have
\begin{align*}
    \begin{pmatrix}
        \mathsf{X}_{\ell} \\ \mathsf Y_{\ell}
    \end{pmatrix}
    = 
    \begin{pmatrix}
        \lambda^2 & \lambda^2 \rho^2 \\ \mu^2 \rho^2 & \mu^2 
    \end{pmatrix}
    \begin{pmatrix}
        \mathsf{X}_{\ell-1} \\ \mathsf Y_{\ell-1}
    \end{pmatrix} \,.
\end{align*}
Straightforward calculation yields that $A_+$ and $A_-$ with
\begin{align*}
    A_{ \pm } = \frac{ \lambda^2 + \mu^2 \pm \sqrt{ \lambda^4+\mu^4-(2-4\rho^4) \lambda^2 \mu^2 } }{ 2 } \mbox{ is the eigenvalue of }
    \begin{pmatrix}
        \lambda^2 & \lambda^2 \rho^2 \\ \mu^2 \rho^2 & \mu^2 
    \end{pmatrix} \,,
\end{align*}
where $A_+$ is the same as in \eqref{eq-def-A-+}. Standard results yields that there exists $D_1,D_2,D_3,D_4 \in \mathbb R$ with $D_1,D_3>0$ such that
\begin{align}
    \mathsf{X}_{\ell} = D_1 A_+^{\ell} + D_2 A_{-}^{\ell} \mbox{ and } \mathsf{Y}_{\ell} = D_3 A_+^{\ell} + D_4 A_{-}^{\ell} \,. \label{eq-formula-X-ell-Y-ell}
\end{align}
Since $A_+>A_-$, we immediately have \eqref{eq-goal-lem-2.1} holds. It remains to show that $A_+>\gamma$ provided that $F(\lambda,\mu,\rho,\gamma)>1$ and $\lambda^2,\mu^2 \leq \gamma$, which follows from
\begin{align*}
    F(\lambda,\mu,\rho,\gamma)>1 &\overset{\eqref{eq-def-F(lambda,mu,rho)}}{\Longrightarrow} \tfrac{ \lambda^2 \rho^2 }{ \gamma-\lambda^2+\lambda^2 \rho^2 } + \tfrac{ \mu^2 \rho^2 }{ \gamma-\mu^2+\mu^2 \rho^2 } >1 \\
    &\Longrightarrow \lambda^2 \rho^2 (\gamma-\mu^2+\mu^2 \rho^2) + \mu^2 \rho^2 (\gamma-\lambda^2+\lambda^2 \rho^2) > (\gamma-\lambda^2+\lambda^2 \rho^2) (\gamma-\mu^2+\mu^2 \rho^2) \\
    &\Longrightarrow \lambda^2 \mu^2(1-\rho^4) -\gamma\lambda^2 -\gamma\mu^2 +\gamma^2 <0 \\ & \Longrightarrow (2\gamma-\lambda^2-\mu^2)^2 < (\lambda^2+\mu^2)^2 - 4(1-\rho^4) \lambda^2 \mu^2 \Longrightarrow A_+>\gamma
\end{align*}
provided that $\lambda^2,\mu^2\leq \gamma$. This completes the proof of Lemma~\ref{lem-bound-beta-mathcal-H}.

\subsection{Proof of Lemma~\ref{lem-bound-beta-mathcal-G}}{\label{subsec:proof-lem-2.6}}

Note that for all $[H] \in \mathcal G$, we can write (we use the convention that $v_{\ell+1}=v_1,u_{\ell+1}=u_1$)
\begin{align*}
    & V(H) = \big\{ v_1,u_1,\ldots,v_{\ell},u_{\ell} \big\} \,; \\
    & V^{\mathsf a}(H) = \big\{ v_1,\ldots,v_{\ell} \big\}, \ V^{\mathsf b}(H) = \big\{ u_1,\ldots,u_{\ell} \big\} \,; \\
    & E(H) = \big\{ (v_1,u_1), (v_2,u_1), \ldots, (v_{\ell},u_{\ell}), (v_1,u_{\ell}) \big\} \,.
\end{align*}
Note that $\mathsf{diff}(H) \subset V^{\mathsf a}(H)$, we thus have $\gamma(v_i,u_i)=\gamma(v_{i+1},u_{i})$ for all $1 \leq i \leq \ell$. Consider the following mapping $\psi:\mathcal G(\ell)\to\mathcal H(\ell)$ such that $\psi(H)$ is the graph with
\begin{align*}
    V(\psi(H))= \big\{ v_1,\ldots,v_{\ell} \big\}, \ E(\psi(H))=\big\{ (v_1,v_2), \ldots, (v_{\ell},v_1) \big\}, \ \gamma(v_i,v_{i+1})=\gamma(v_i,u_i) \,.
\end{align*}
It is clear that $\psi$ is a bijection from $\mathcal G(\ell)$ to $\mathcal H(\ell)$ and 
\begin{align*}
    |E_{\bullet}(\psi(H))|=\frac{1}{2}|E_{\bullet}(H)|,\ |E_{\circ}(\psi(H))|=\frac{1}{2}|E_{\circ}(H)|,\ |\mathsf{diff}(\psi(H))|=|\mathsf{diff}(H)| \,.
\end{align*}
Thus, (recalling \eqref{eq-def-Xi-S} and \eqref{eq-def-Upsilon-S}) we have $\Xi(\psi(H))=\Upsilon(H)$. Also, it is clear that $|\mathsf{Aut}(H)|=|\mathsf{Aut}(\psi(H))|$ and thus
\begin{align*}
    \beta_{\mathcal G} = \sum_{ H \in \mathcal G } \frac{ \Upsilon(H)^2 }{ |\mathsf{Aut}(H)| } = \sum_{ H \in \mathcal H } \frac{ \Xi(H)^2 }{ |\mathsf{Aut}(H)| } = \beta_{\mathcal H} \,.
\end{align*}
Lemma~\ref{lem-bound-beta-mathcal-G} now directly follows from Lemma~\ref{lem-bound-beta-mathcal-H}.

\subsection{Proof of Lemma~\ref{lem-bound-beta-mathcal-J}}{\label{subsec:proof-lem-2.10}}

The proof of Lemma~\ref{lem-bound-beta-mathcal-J} is highly similar to the proof of Lemma~\ref{lem-bound-beta-mathcal-H}, so we will only provide an outline with the main differences while adapting arguments from Lemma~\ref{lem-bound-beta-mathcal-H} without presenting full details. Recall that we define $\mathcal W=\mathcal W(\ell)$ to be the set of $\omega=(\omega_1,\ldots,\omega_{\ell}) \in \{ \bullet,\circ \}^{\ell}$. Also denote 
\begin{equation}
    \widetilde{\mathcal J} = \Big\{ (v,[H]): [H] \in \mathcal J, v \in \mathsf L(H) \Big\}, \quad \widetilde{\mathcal J}_{\star} = \Big\{ (v,[H]): [H] \in \mathcal J_{\star}, v \in \mathsf L(H) \Big\} \,.
\end{equation}
It is clear that there exists a bijection 
\begin{align*}
    \psi:\widetilde{\mathcal J} \to \{ \omega \in \mathcal W: \omega_1 = \omega_{\ell}=1 \} \mbox{ and } \psi_{\star}: \widetilde{\mathcal J}_{\star} \to \mathcal W
\end{align*}
that preserves the quantity $|E_{\bullet}(\cdot)|,|E_{\circ}(\cdot)|$ and $|\mathsf{diff}(\cdot)|$. Thus, similarly as \eqref{eq-goal-lem-2.1}, it suffices to show that there exists a constant $D=D(\lambda,\mu,\rho)>0$ such that (recall \eqref{eq-def-Xi-S})
\begin{equation}{\label{eq-goal-lem-2.5}}
    D^{-1} A_+^{\ell} \leq \sum_{ \omega\in\mathcal W: \omega_1 = \omega_\ell=1 } \Xi(\omega)^2 \leq \sum_{ \omega\in\mathcal W } \Xi(\omega)^2 \leq D A_+^{\ell} \,.
\end{equation}
Recall \eqref{eq-def-mathsf-X-Y}. It is straightforward to check that
\begin{align*}
    \rho^4 \lambda^4 (\mathsf X_{\ell-2}+\mathsf{Y_{\ell-2}}) \leq \sum_{ \omega\in\mathcal W: \omega_1 = \omega_\ell=1 } \Xi(\omega)^2 \leq \sum_{ \omega\in\mathcal W } \Xi(\omega)^2 \leq \mathsf{X}_{\ell} + \mathsf Y_{\ell} \,.
\end{align*}
\eqref{eq-goal-lem-2.5} then directly follows from \eqref{eq-formula-X-ell-Y-ell}.

\subsection{Proof of Lemma~\ref{lem-bound-beta-mathcal-I}}{\label{subsec:proof-lem-2.14}}

Clearly, it suffices to show that
\begin{align}
    & \beta_{\mathcal I_{\star\star}} \leq \Theta(1) \cdot A_+^{\ell} \,; \label{eq-beta-I-upper-bound} \\
    & \beta_{\mathcal I} \geq \Theta(1) \cdot A_+^{\ell} \,. \label{eq-beta-I-lower-bound}
\end{align}
We first show \eqref{eq-beta-I-upper-bound}. Note that for all $[H]\in\mathcal I_{\star\star}$, denote $V(H)=\{ v_1,\ldots,v_{m} \}$ where $m \in \{ 2\ell, 2\ell+1 \}$ and denote
\begin{align*}
    \mathtt a = 1+ \mathbf 1_{ \{ v_1 \in V^{\mathsf b}(H) \} }, \mathtt b = m-\mathbf 1_{ \{ v_m \in V^{\mathsf b}(H) \} } \,,
\end{align*}
we then have $H'=\{ v_{\mathtt a}, \ldots, v_{\mathtt b} \} \in \mathcal I_{\star\star}( \frac{ \mathtt b-\mathtt a }{ 2 } )$. Thus, we have
\begin{align*}
    \sum_{ [H] \in \mathcal I_{\star\star}(\ell) } \frac{\Upsilon(H)^2}{|\mathsf{Aut}(H)|} \leq (\lambda+\mu) \Big( \sum_{ [H] \in \mathcal I_{\star}(\ell) } \frac{\Upsilon(H)^2}{|\mathsf{Aut}(H)|} + \sum_{ [H] \in \mathcal I_{\star}(\ell-1) } \frac{\Upsilon(H)^2}{|\mathsf{Aut}(H)|} \Big) \,.
\end{align*}
Thus, to show \eqref{eq-beta-I-upper-bound} suffices to show that 
\begin{align}
    \sum_{ [H] \in \mathcal I_{\star}(\ell) } \frac{\Upsilon(H)^2}{|\mathsf{Aut}(H)|} \leq D \cdot A_+^{\ell} \,. \label{eq-beta-I-upper-bound-relax}
\end{align}
Consider the following mapping $\psi:\mathcal I_{\star}(\ell)\to\mathcal J_{\star}(\ell)$: for all $[H] \in \mathcal I_{\star}(\ell)$, we can write
\begin{align*}
    & V(H) = \big\{ v_1,u_1,\ldots,v_{\ell},u_{\ell},v_{\ell+1} \big\} \,; \\
    & V^{\mathsf a}(H) = \big\{ v_1,\ldots,v_{\ell+1} \big\}, \ V^{\mathsf b}(H) = \big\{ u_1,\ldots,u_{\ell} \big\} \,; \\
    & E(H) = \big\{ (v_1,u_1), (v_2,u_1), \ldots, (v_{\ell},u_{\ell}), (v_{\ell+1},u_{\ell}) \big\} \,.
\end{align*}
Note that $\mathsf{diff}(H) \subset V^{\mathsf a}(H)$, we thus have $\gamma(v_i,u_i)=\gamma(v_{i+1},u_{i})$ for all $1 \leq i \leq \ell$. We let $\psi(H)$ be the graph with
\begin{align*}
    V(\psi(H))= \big\{ v_1,\ldots,v_{\ell+1} \big\}, \ E(\psi(H))=\big\{ (v_1,v_2), \ldots, (v_{\ell},v_{\ell+1}) \big\}, \ \gamma(v_i,v_{i+1})=\gamma(v_i,u_i) \,.
\end{align*}
It is clear that $\psi$ is a bijection from $\mathcal I_{\star}(\ell)$ to $\mathcal J_{\star}(\ell)$ and 
\begin{align*}
    |E_{\bullet}(\psi(H))|=\frac{1}{2}|E_{\bullet}(H)|,\ |E_{\circ}(\psi(H))|=\frac{1}{2}|E_{\circ}(H)|,\ |\mathsf{diff}(\psi(H))|=|\mathsf{diff}(H)| \,.
\end{align*}
Thus, (recalling \eqref{eq-def-Xi-S} and \eqref{eq-def-Upsilon-S}) we have $\Xi(\psi(H))=\Upsilon(H)$. Also, it is clear that $|\mathsf{Aut}(H)|=|\mathsf{Aut}(\psi(H))|$ and thus
\begin{align*}
    \beta_{\mathcal I_{\star}} = \sum_{ H \in \mathcal G } \frac{ \Upsilon(H)^2 }{ |\mathsf{Aut}(H)| } = \sum_{ H \in \mathcal J } \frac{ \Xi(H)^2 }{ |\mathsf{Aut}(H)| } = \beta_{\mathcal J_{\star}} \,.
\end{align*}
\eqref{eq-beta-I-upper-bound} then directly follows from Lemma~\ref{lem-bound-beta-mathcal-J}. We now prove \eqref{eq-beta-I-lower-bound}. It is straightforward to check that the mapping $\psi$ above is also a bijection from $\mathcal I(\ell)$ to $\mathcal J(\ell)$. Thus, we also have 
\begin{align*}
    \beta_{\mathcal I} = \sum_{ H \in \mathcal G } \frac{ \Upsilon(H)^2 }{ |\mathsf{Aut}(H)| } = \sum_{ H \in \mathcal J } \frac{ \Xi(H)^2 }{ |\mathsf{Aut}(H)| } = \beta_{\mathcal J} \,.
\end{align*}
\eqref{eq-beta-I-lower-bound} now also directly follows from Lemma~\ref{lem-bound-beta-mathcal-J}.

\section{Supplementary proofs in Section~\ref{sec:stat-analysis}}{\label{sec:supp-proofs-sec-3}}

\subsection{Proof of Lemma~\ref{lem-est-cov-f-S-f-K}}{\label{subsec:proof-lem-3.3}}

Note that by first averaging over $\bm W$ and $\bm Z$, we have
\begin{align}
    \mathbb E_{\Pb}[ f_S f_{K} ] &\overset{\eqref{eq-def-correlated-spike-specific},\eqref{eq-def-f-mathcal-H}}{=} \mathbb E_{x,y} \mathbb E_{\Pb}\Bigg[ \prod_{ (i,j) \in E_{\bullet}(S) \cap E_{\bullet}(K) } \big( \tfrac{\lambda}{\sqrt{n}} x_i x_j+\bm W_{i,j} \big)^2 \prod_{ (i,j) \in E_{\bullet}(S) \triangle E_{\bullet}(K) } \big( \tfrac{\lambda}{\sqrt{n}} x_i x_j+\bm W_{i,j} \big)  \nonumber \\
    & \prod_{ (i,j) \in E_{\circ}(S) \cap E_{\circ}(K) } \big( \tfrac{\mu}{\sqrt{n}} y_i y_j+\bm Z_{i,j} \big)^2 \prod_{ (i,j) \in E_{\circ}(S) \triangle E_{\circ}(K) } \big( \tfrac{\mu}{\sqrt{n}} y_i y_j+\bm Z_{i,j} \big) \mid x,y \Bigg] \nonumber \\
    &= \mathbb E\Bigg[ \prod_{ (i,j) \in E_{\bullet}(S) \cap E_{\bullet}(K) } \big( \tfrac{\lambda^2 x_i^2 x_j^2}{n}+1 \big) \prod_{ (i,j) \in E_{\bullet}(S) \triangle E_{\bullet}(K) } \big( \tfrac{\lambda}{\sqrt{n}} x_i x_j \big)  \nonumber \\
    & \prod_{ (i,j) \in E_{\circ}(S) \cap E_{\circ}(K) } \big( \tfrac{\mu^2 y_i^2 y_j^2}{n} + 1 \big)^2 \prod_{ (i,j) \in E_{\circ}(S) \triangle E_{\circ}(K) } \big( \tfrac{\mu}{\sqrt{n}} y_i y_j \big) \Bigg] \,. \label{eq-cov-f-S-f-K-simplify-1}
\end{align}
Thus, when $V(S) \cap V(K)=\emptyset$, we have $E_i(S) \cap E_i(K)=\emptyset$ and thus
\begin{align*}
    \eqref{eq-cov-f-S-f-K-simplify-1} &= \mathbb E\Bigg[ \prod_{ (i,j) \in E_{\bullet}(S) \cup E_{\bullet}(K) } \big( \tfrac{\lambda}{\sqrt{n}} x_i x_j \big) \prod_{ (i,j) \in E_{\circ}(S) \cup E_{\circ}(K) } \big( \tfrac{\mu}{\sqrt{n}} y_i y_j \big) \Bigg] \\
    &= \frac{ \lambda^{|E_{\bullet}(S)|+|E_{\bullet}(K)|} \mu^{|E_{\circ}(S)|+|E_{\circ}(K)|} \rho^{|\mathsf{diff}(S)|+|\mathsf{diff}(K)|} }{ n^{\ell} } = \frac{\Xi(S)\Xi(K)}{n^{\ell}} \,.
\end{align*}
Combined with \eqref{eq-mean-Pb-f-S}, this shows that $\operatorname{Cov}_{\Pb}(f_S,f_K)=0$ when $V(S) \cap V(K)=\emptyset$. For the case $V(S) \cap V(K)\neq \emptyset$, since we have
\begin{align*}
    |E_{\bullet}(S) \triangle E_{\bullet}(K)| + |E_{\circ}(S) \triangle E_{\circ}(K)| = 2\ell-2|E_{\bullet}(S) \cap E_{\bullet}(K)| - 2|E_{\circ}(S) \cap E_{\circ}(K)| \,, 
\end{align*}
it suffices to show that
\begin{align}
    & \mathbb E\Bigg[ \prod_{ (i,j) \in E_{\bullet}(S) \cap E_{\bullet}(K) } \big( \tfrac{\lambda^2 x_i^2 x_j^2}{n}+1 \big) \prod_{ (i,j) \in E_{\circ}(S) \cap E_{\circ}(K) } \big( \tfrac{\mu^2 y_i^2 y_j^2}{n} + 1 \big)^2 \nonumber \\
    &\quad \prod_{ (i,j) \in E_{\bullet}(S) \triangle E_{\bullet}(K) } \big( x_i x_j \big) \prod_{ (i,j) \in E_{\circ}(S) \triangle E_{\circ}(K) } \big( y_i y_j \big) \Bigg] \nonumber \\
    \leq\ & \rho^{|\mathsf{diff}(S) \setminus V(K)|+|\mathsf{diff}(K) \setminus V(S)|} C^{|V(S) \cap V(K)|-|V(S_{\bullet} \cap K_{\bullet}) \cup V(S_{\circ} \cap K_{\circ})|} \,.  \label{eq-goal-cov-f-S-f-K}
\end{align}
Note that the left-hand side of \eqref{eq-goal-cov-f-S-f-K} equals
\begin{align*}
    \sum_{ \substack{ \mathtt E \subset E_{\bullet}(S) \triangle E_{\bullet}(K) \\ \mathtt F \subset E_{\circ}(S) \triangle E_{\circ}(K) } } \frac{ \lambda^{2|\mathtt E|} \mu^{2|\mathtt F|} }{ n^{|\mathtt E|+|\mathtt F|} } \mathbb E\Bigg[ \prod_{ (i,j) \in \mathtt E } x_i^2 x_j^2 \prod_{ (i,j) \in \mathtt F } y_i^2 y_j^2 \prod_{ (i,j) \in E_{\bullet}(S) \triangle E_{\bullet}(K) } x_i x_j \prod_{ (i,j) \in E_{\circ}(S) \triangle E_{\circ}(K) } y_i y_j \Bigg] \,,
\end{align*}
which can be re-written as
\begin{align}
    \sum_{ \substack{ \mathtt E \subset E_{\bullet}(S) \triangle E_{\bullet}(K) \\ \mathtt F \subset E_{\circ}(S) \triangle E_{\circ}(K) } } \big( \tfrac{\lambda^2}{n} \big)^{|\mathtt E|} \big( \tfrac{\mu^2}{n} \big)^{|\mathtt F|} \mathbb E\Bigg[ \prod_{ i \in V(S) \cup V(K) } x_i^{ \mathtt h(i;\mathtt E,\mathtt F;S,K) } y_i^{ \mathtt k(i;\mathtt E,\mathtt F;S,K) } \Bigg] \,, \label{eq-goal-cov-f-S-f-K-transfer-1}
\end{align}
where (below we use $i \sim e$ to denote that the vertex $i$ is incident to the edge $e$)
\begin{equation}{\label{eq-def-mathtt-h-k}}
    \begin{aligned}
        & \mathtt h(i;\mathtt E,\mathtt F;S,K) = 2 \#\{ e \in \mathtt E: i \sim e \} + \#\{ e \in E_{\bullet}(S) \triangle E_{\bullet}(K): i \sim e \} \,; \\
        & \mathtt k(i;\mathtt E,\mathtt F;S,K) = 2 \#\{ e \in \mathtt F: i \sim e \} + \#\{ e \in E_{\circ}(S) \triangle E_{\circ}(K): i \sim e \} \,.
    \end{aligned}
\end{equation}
In addition, denote
\begin{align*}
    V_{\mathsf{bad}} = \{ i \in V(S) \triangle V(K): i \sim e \mbox{ for some } e \in \mathtt E \cup \mathtt F \} \,.
\end{align*}
We will calculate $(\mathtt h(i;\mathtt E,\mathtt F;S,K),\mathtt k(i;\mathtt E,\mathtt F;S,K))$ by considering the following cases separately:
\begin{itemize}
    \item If $i \in \mathsf{diff}(S) \setminus V(K)$ and $i \not\in V_{\mathsf{bad}}$, then 
    \begin{align*}
        \Big( \mathtt h(i;\mathtt E,\mathtt F;S,K), \mathtt k(i;\mathtt E,\mathtt F;S,K) \Big) = (1,1) \Longrightarrow \mathbb E\Big[ x_i^{ \mathtt h(i;\mathtt E,\mathtt F;S,K) } y_i^{ \mathtt k(i;\mathtt E,\mathtt F;S,K) } \Big] =\rho \,.
    \end{align*}
    Similarly, if $i \in \mathsf{diff}(K) \setminus V(S)$ and $i \not\in V_{\mathsf{bad}}$, then
    \begin{align*}
        \Big( \mathtt h(i;\mathtt E,\mathtt F;S,K), \mathtt k(i;\mathtt E,\mathtt F;S,K) \Big) = (1,1) \Longrightarrow \mathbb E\Big[ x_i^{ \mathtt h(i;\mathtt E,\mathtt F;S,K) } y_i^{ \mathtt k(i;\mathtt E,\mathtt F;S,K) } \Big] =\rho \,.
    \end{align*}
    \item If $i \in V(S) \setminus (\mathsf{diff}(S) \cup V(K))$ and $i \not\in V_{\mathsf{bad}}$, then
    \begin{align*}
        \Big( \mathtt h(i;\mathtt E,\mathtt F;S,K), \mathtt k(i;\mathtt E,\mathtt F;S,K) \Big) = (2,0) \mbox{ or } (0,2) \Longrightarrow \mathbb E\Big[ x_i^{ \mathtt h(i;\mathtt E,\mathtt F;S,K) } y_i^{ \mathtt k(i;\mathtt E,\mathtt F;S,K) } \Big] =1 \,.
    \end{align*}
    Similarly, if $i \in V(K) \setminus (\mathsf{diff}(K) \cup V(S))$ and $i \not\in V_{\mathsf{bad}}$, then
    \begin{align*}
        \Big( \mathtt h(i;\mathtt E,\mathtt F;S,K), \mathtt k(i;\mathtt E,\mathtt F;S,K) \Big) = (2,0) \mbox{ or } (0,2) \Longrightarrow \mathbb E_{(x,y)\sim\pi}\Big[ x_i^{ \mathtt h(i;\mathtt E,\mathtt F;S,K) } y_i^{ \mathtt k(i;\mathtt E,\mathtt F;S,K) } \Big] =1 \,.
    \end{align*}
    \item If $i \in V(S_{\bullet} \cap K_{\bullet})$ and $i \not\in V_{\mathsf{bad}}$, then
    \begin{align*}
        \Big( \mathtt h(i;\mathtt E,\mathtt F;S,K)+\mathtt k(i;\mathtt E,\mathtt F;S,K) \Big) \leq 2 \Longrightarrow \mathbb E_{(x,y)\sim\pi}\Big[ x_i^{ \mathtt h(i;\mathtt E,\mathtt F;S,K) } y_i^{ \mathtt k(i;\mathtt E,\mathtt F;S,K) } \Big] \leq 1 \,.
    \end{align*}
    Similarly, if $i \in V(S_{\circ} \cap K_{\circ})$ and $i \not\in V_{\mathsf{bad}}$, then
    \begin{align*}
        \Big( \mathtt h(i;\mathtt E,\mathtt F;S,K)+\mathtt k(i;\mathtt E,\mathtt F;S,K) \Big) \leq 2 \Longrightarrow \mathbb E_{(x,y) \sim\mu}\Big[ x_i^{ \mathtt h(i;\mathtt E,\mathtt F;S,K) } y_i^{ \mathtt k(i;\mathtt E,\mathtt F;S,K) } \Big] \leq 1 \,.
    \end{align*} 
    \item If $i \in (V(S) \cap V(K)) \setminus (V(S_{\bullet} \cap K_{\bullet}) \cup V(S_{\circ} \cap K_{\circ}))$ or $i \not\in V_{\mathsf{bad}}$, then since the degree of each $i$ in $S\cup K$ is bounded by $4$, from Assumption~\ref{assum-upper-bound} we have
    \begin{align*}
        \Big( \mathtt h(i;\mathtt E,\mathtt F;S,K)+\mathtt k(i;\mathtt E,\mathtt F;S,K) \Big) \leq 4 \Longrightarrow \mathbb E_{(x,y) \sim\mu}\Big[ x_i^{ \mathtt h(i;\mathtt E,\mathtt F;S,K) } y_i^{ \mathtt k(i;\mathtt E,\mathtt F;S,K) } \Big] \leq C \,.
    \end{align*}
\end{itemize}
Combining the above cases together, we get that
\begin{align}
    \eqref{eq-cov-f-S-f-K-simplify-1} \leq \sum_{ \substack{ \mathtt E \subset E_{\bullet}(S) \triangle E_{\bullet}(K) \\ \mathtt F \subset E_{\circ}(S) \triangle E_{\circ}(K) } } & \big( \tfrac{\lambda^2}{n} \big)^{|\mathtt E|} \big( \tfrac{\mu^2}{n} \big)^{|\mathtt F|} \rho^{ |\mathsf{diff}(S) \setminus V(K)| + |\mathsf{diff}(K) \setminus V(S)| } \nonumber \\
    & \cdot C^{ |V(S) \cap V(K)| -|V(S_{\bullet} \cap K_{\bullet}) \cup V(S_{\circ} \cap K_{\circ})| + |V_{\mathsf{bad}}| } \,. \label{eq-cov-f-S-f-K-simplify-2} 
\end{align}
Finally, note that $|V_{\mathsf{bad}}| \leq 2|\mathtt E|+2|\mathtt F|$, we have 
\begin{align*}
    \eqref{eq-cov-f-S-f-K-simplify-2} &\leq \sum_{ \substack{ \mathtt E \subset E_{\bullet}(S) \triangle E_{\bullet}(K) \\ \mathtt F \subset E_{\circ}(S) \triangle E_{\circ}(K) } } \big( \tfrac{\lambda^2}{n} \big)^{|\mathtt E|} \big( \tfrac{\mu^2}{n} \big)^{|\mathtt F|} \rho^{ |\mathsf{diff}(S) \setminus V(K)| + |\mathsf{diff}(K) \setminus V(S)| }  \\
    & \quad\quad\quad\quad\quad\quad\quad\quad \cdot C^{ |V(S) \cap V(K)| -|V(S_{\bullet} \cap K_{\bullet}) \cup V(S_{\circ} \cap K_{\circ})| + 2|\mathtt E|+2|\mathtt F| }  \\
    &\leq [1+o(1)] \cdot \rho^{ |\mathsf{diff}(S) \setminus V(K)| + |\mathsf{diff}(K) \setminus V(S)| } C^{ |V(S) \cap V(K)| -|V(S_{\bullet} \cap K_{\bullet}) \cup V(S_{\circ} \cap K_{\circ})| } \,,
\end{align*}
leading to the desired result.

\subsection{Proof of Lemma~\ref{lem-detection-most-technical}}{\label{subsec:proof-lem-3.4}}

We first prove \eqref{eq-bound-var-Pb-f-H-relax-3-Part-1}. Note by \eqref{eq-def-mathtt-M} that $\mathtt M(S,S)=1$. Thus, 
\begin{align*}
    \eqref{eq-var-Pb-f-H-relax-3-Part-1} &= \sum_{ S \in \mathsf K_n: [S] \in \mathcal H } \frac{ \Xi(S)^2 }{ n^{\ell} \beta_{\mathcal H}^2 } = \sum_{ [H] \in \mathcal H } \frac{ \Xi(H)^2 }{ n^{\ell} \beta_{\mathcal H}^2 } \cdot \#\{ S \in \mathsf{K}_n: S \cong H \} \\
    &= [1+o(1)] \cdot \sum_{ [H] \in \mathcal H } \frac{ \Xi(H)^2 }{ n^{\ell} \beta_{\mathcal H}^2 } \cdot \frac{ n^{\ell} }{ |\mathsf{Aut}(H)| } \overset{\eqref{eq-def-beta-mathcal-H}}{=} \frac{1+o(1)}{\beta_{\mathcal H}} \overset{\text{Lemma~\ref{lem-bound-beta-mathcal-H}}}{=} o(1) \,,
\end{align*}
leading to \eqref{eq-bound-var-Pb-f-H-relax-3-Part-1}. Now we prove \eqref{eq-bound-var-Pb-f-H-relax-3-Part-2}. Note that
\begin{align}
    \mathtt M(S,K) \overset{\eqref{eq-def-mathtt-M}}{=}\ & \lambda^{ |E_{\bullet}(S)|+|E_{\bullet}(K)|-2|E_{\bullet}(S) \cap E_{\bullet}(K)| } \mu^{ |E_{\circ}(S)|+|E_{\circ}(K)|-2|E_{\circ}(S) \cap E_{\circ}(K)| }  \nonumber \\
    & \cdot \rho^{|\mathsf{diff}(S) \setminus V(K)|+|\mathsf{diff}(K) \setminus V(S)|} C^{|V(S) \cap V(K)|-|V(S_{\bullet} \cap K_{\bullet}) \cup V(S_{\circ} \cap K_{\circ})|}  \nonumber \\
    \leq\ & \lambda^{|E_{\bullet}(S)|+|E_{\bullet}(K)|} \mu^{|E_{\circ}(S)|+|E_{\circ}(K)|} \rho^{|\mathsf{diff}(S)|+|\mathsf{diff}(K)|} \big( C \lambda^{-2} \mu^{-2} \rho^{-2} \big)^{ |V(S) \cap V(K)| } \nonumber \\
    =\ & n^{o(1)} \cdot \lambda^{|E_{\bullet}(S)|+|E_{\bullet}(K)|} \mu^{|E_{\circ}(S)|+|E_{\circ}(K)|} \rho^{|\mathsf{diff}(S)|+|\mathsf{diff}(K)|} \nonumber \\
    \overset{\eqref{eq-def-Xi-S}}{=}\ & n^{o(1)} \cdot \Xi(S) \Xi(K) \,,  \label{eq-trivial-bound-mathtt-M}
\end{align}
where the second equality follows from $|V(S) \cap V(K)| \leq \ell=o(\tfrac{\log n}{\log\log n})$ in \eqref{eq-condition-strong-detection}, and the inequality follows from the assumptions $\lambda,\mu,\rho\in [0,1]$ and 
\begin{align*}
    & |E_{\bullet}(S) \cap E_{\bullet}(K)| \leq |V(S) \cap V(K)|, \quad |\mathsf{diff}(S)|- |\mathsf{diff}(S) \setminus V(K)| \leq |V(S) \cap V(K)| \,; \\
    & |E_{\circ}(S) \cap E_{\circ}(K)| \leq |V(S) \cap V(K)|, \quad |\mathsf{diff}(K)|- |\mathsf{diff}(K) \setminus V(S)| \leq |V(S) \cap V(K)| \,.
\end{align*}
Thus, we have
\begin{align}
    \eqref{eq-var-Pb-f-H-relax-3-Part-2} \leq \sum_{ \substack{ S,K \subset \mathsf K_n: [S],[K] \in \mathcal H \\ V(S) \cap V(K) \neq \emptyset, S \neq K } } \frac{ \Xi(S)^2 \Xi(K)^2 }{ n^{2\ell-|E_{\bullet}(S)\cap E_{\bullet}(K)|-|E_{\circ}(S) \cap E_{\circ}(K)|-o(1)} \beta_{\mathcal H}^2 } \,.  \label{eq-var-Pb-f-H-relax-3-Part-2-simplify-1}
\end{align}
Note that for $S \neq K$ and $V(S) \cap V(K) \neq \emptyset$, we must have
\begin{align*}
    |E_{\bullet}(S) \cap E_{\bullet}(K)|+|E_{\circ}(S) \cap E_{\circ}(K)| \leq |E(S) \cap E(K)| \leq |V(S) \cap V(K)|-1 \,.
\end{align*}
Thus, we have
\begin{align}
    \eqref{eq-var-Pb-f-H-relax-3-Part-2-simplify-1} = \sum_{ \substack{ S \in \mathsf{K}_n: [S] \in \mathcal H \\ 0 \leq p,q \leq \ell \\ p+q+1 \leq t \leq \ell } } \sum_{ [H] \in \mathcal H } \frac{ \Xi(S)^2 \Xi(H)^2 }{ n^{2\ell-p-q-o(1)} \beta_{\mathcal H}^2 } \cdot \mathsf{ENUM}(S,[H];t) \,,  \label{eq-var-Pb-f-H-relax-3-Part-2-simplify-2}
\end{align}
where 
\begin{equation}{\label{eq-def-ENUM-detection}}
    \mathsf{ENUM}(S,[H];t) = \{ K \in \mathsf{K}_n: K \cong H, |V(K) \cap V(S)|=t \}
\end{equation}
Note that we have at most $\binom{\ell}{t}$ ways to choose $V(K) \cap V(S)$ and at most $\binom{n-\ell}{\ell-t}$ ways to choose $V(K) \setminus V(S)$. In addition, given $V(K)$ we have at most $\frac{\ell!}{|\mathsf{Aut}(K)|}$ ways to choose $K$. Thus, 
\begin{align*}
    \mathsf{ENUM}(S,[H];t) \leq \binom{\ell}{t} \binom{n-\ell}{\ell-t} \cdot \frac{\ell!}{|\mathsf{Aut}(K)|} \leq \frac{ \ell^{2t} n^{\ell-t} }{ |\mathsf{Aut}(H)| } \,.
\end{align*}
Plugging this bound into \eqref{eq-var-Pb-f-H-relax-3-Part-2-simplify-2}, we get that
\begin{align}
    \eqref{eq-var-Pb-f-H-relax-3-Part-2-simplify-2} &\leq \sum_{ \substack{ S \in \mathsf{K}_n: [S] \in \mathcal H \\ 0 \leq p,q \leq \ell \\ p+q+1 \leq t \leq \ell } } \sum_{ [H] \in \mathcal H } \frac{ \Xi(S)^2 \Xi(H)^2 }{ n^{2\ell-p-q-o(1)} \beta_{\mathcal H}^2 } \cdot \frac{ \ell^{2t} n^{\ell-t} }{ |\mathsf{Aut}(H)| } \overset{\eqref{eq-def-beta-mathcal-H}}{=} \sum_{ \substack{ S \in \mathsf{K}_n: [S] \in \mathcal H \\ 0 \leq p,q \leq \ell \\ p+q+1 \leq t \leq \ell } } \frac{ \Xi(S)^2 \ell^{2v} n^{\ell-t} }{ n^{2\ell-p-q-o(1)} \beta_{\mathcal H} } \,.  \label{eq-var-Pb-f-H-relax-3-Part-2-simplify-3}
\end{align}
From \eqref{eq-condition-strong-detection} we see that $\ell^{2v}\leq \ell^{2\ell} =n^{o(1)}$, thus
\begin{align*}
    \eqref{eq-var-Pb-f-H-relax-3-Part-2-simplify-3} &= \sum_{ [H] \in \mathcal H } \sum_{ S \in \mathsf{K}_n: S \cong H } \sum_{ \substack{ 0 \leq p,q \leq \ell \\ p+q+1 \leq t \leq \ell } } \frac{ \Xi(S)^2 }{ n^{\ell+t-p-q+o(1)} \beta_{\mathcal H} } \\
    &= \sum_{ [H] \in \mathcal H } \frac{ \Xi(H)^2 }{ n^{\ell+1+o(1)} \beta_{\mathcal H} } \cdot \#\{ S \in \mathsf{K}_n: S \cong H \} \\
    &= \sum_{ [H] \in \mathcal H } \frac{ \Xi(H)^2 }{ n^{\ell+1+o(1)} \beta_{\mathcal H} } \cdot \frac{ n^{\ell} }{ |\mathsf{Aut}(H)| } \overset{\eqref{eq-def-beta-mathcal-H}}{=} n^{-1+o(1)} \,,
\end{align*}
leading to \eqref{eq-bound-var-Pb-f-H-relax-3-Part-2}.

\subsection{Proof of Lemma~\ref{lem-est-cov-f-S-f-K-cov}}{\label{subsec:proof-lem-3.7}}

Recall \eqref{eq-def-correlated-spike-covariance-specific} and \eqref{eq-def-f-mathcal-G}. By first averaging over $\bm W$ and $\bm Z$, we have $\mathbb E_{\overline\Pb}[ h_S h_{K} ]$ equals
\begin{align}
    & \mathbb E_{x,y,\bm u,\bm v} \mathbb E_{\overline\Pb}\Bigg[ \prod_{ (i,j) \in E_{\bullet}(S) \cap E_{\bullet}(K) } \big( \tfrac{\sqrt{\lambda}}{\sqrt{n}} x_i \bm u_j+\bm W_{i,j} \big)^2 \prod_{ (i,j) \in E_{\bullet}(S) \triangle E_{\bullet}(K) } \big( \tfrac{\sqrt{\lambda}}{\sqrt{n}} x_i \bm u_j+\bm W_{i,j} \big)  \nonumber \\
    & \quad\quad\quad\quad\quad \prod_{ (i,j) \in E_{\circ}(S) \cap E_{\circ}(K) } \big( \tfrac{\sqrt{\mu}}{\sqrt{n}} y_i \bm v_j+\bm Z_{i,j} \big)^2 \prod_{ (i,j) \in E_{\circ}(S) \triangle E_{\circ}(K) } \big( \tfrac{\sqrt{\mu}}{\sqrt{n}} y_i \bm v_j+\bm Z_{i,j} \big) \mid x,y,\bm u,\bm v \Bigg] \nonumber \\
    =\ & \mathbb E\Bigg[ \prod_{ (i,j) \in E_{\bullet}(S) \cap E_{\bullet}(K) } \big( \tfrac{\lambda x_i^2 \bm u_j^2}{n}+1 \big) \prod_{ (i,j) \in E_{\bullet}(S) \triangle E_{\bullet}(K) } \big( \tfrac{\sqrt{\lambda}}{\sqrt{n}} x_i \bm u_j \big)  \nonumber \\
    & \quad \prod_{ (i,j) \in E_{\circ}(S) \cap E_{\circ}(K) } \big( \tfrac{\mu y_i^2 \bm v_j^2}{n} + 1 \big)^2 \prod_{ (i,j) \in E_{\circ}(S) \triangle E_{\circ}(K) } \big( \tfrac{\sqrt{\mu}}{\sqrt{n}} y_i \bm v_j \big) \Bigg] \,. \label{eq-cov-f-S-f-K-cov-simplify-1}
\end{align}
Thus, when $V(S) \cap V(K)=\emptyset$, we have $E_i(S) \cap E_i(K)=\emptyset$ and thus
\begin{align*}
    \eqref{eq-cov-f-S-f-K-cov-simplify-1} &= \mathbb E\Bigg[ \prod_{ (i,j) \in E_{\bullet}(S) \cup E_{\bullet}(K) } \big( \tfrac{\sqrt{\lambda}}{\sqrt{n}} x_i \bm u_j \big) \prod_{ (i,j) \in E_{\circ}(S) \cup E_{\circ}(K) } \big( \tfrac{\sqrt{\mu}}{\sqrt{n}} y_i \bm v_j \big) \Bigg] \\
    &= \frac{ \lambda^{\frac{1}{2}(|E_{\bullet}(S)|+|E_{\bullet}(K)|)} \mu^{\frac{1}{2}(|E_{\circ}(S)|+|E_{\circ}(K)|)} \rho^{|\mathsf{diff}(S)|+|\mathsf{diff}(K)|} }{ n^{2\ell} } = \frac{\Upsilon(S)\Upsilon(K)}{n^{2\ell}} \,.
\end{align*}
Combined with \eqref{eq-mean-Pb-f-S-cov}, this shows that $\operatorname{Cov}_{\overline\Pb}(h_S,h_K)=0$ when $V(S) \cap V(K)=\emptyset$. For the case $V(S) \cap V(K)\neq \emptyset$, since we have
\begin{align*}
    |E_{\bullet}(S) \triangle E_{\bullet}(K)| + |E_{\circ}(S) \triangle E_{\circ}(K)| = 4\ell-2|E_{\bullet}(S) \cap E_{\bullet}(K)| - 2|E_{\circ}(S) \cap E_{\circ}(K)| \,, 
\end{align*}
it suffices to show that
\begin{align}
    & \mathbb E\Bigg[ \prod_{ (i,j) \in E_{\bullet}(S) \cap E_{\bullet}(K) } \big( \tfrac{\lambda x_i^2 \bm u_j^2}{n}+1 \big) \prod_{ (i,j) \in E_{\circ}(S) \cap E_{\circ}(K) } \big( \tfrac{\mu y_i^2 \bm v_j^2}{n} + 1 \big)^2 \nonumber \\
    &\quad \prod_{ (i,j) \in E_{\bullet}(S) \triangle E_{\bullet}(K) } \big( x_i \bm u_j \big) \prod_{ (i,j) \in E_{\circ}(S) \triangle E_{\circ}(K) } \big( y_i \bm v_j \big) \Bigg] \nonumber \\
    \leq\ & \rho^{|\mathsf{diff}(S) \setminus V(K)|+|\mathsf{diff}(K) \setminus V(S)|} (2C)^{|V(S) \cap V(K)|-|V(S_{\bullet} \cap K_{\bullet}) \cup V(S_{\circ} \cap K_{\circ})|} \,.  \label{eq-goal-cov-f-S-f-K-cov}
\end{align}
Note that the left-hand side of \eqref{eq-goal-cov-f-S-f-K-cov} equals
\begin{align*}
    \sum_{ \substack{ \mathtt E \subset E_{\bullet}(S) \triangle E_{\bullet}(K) \\ \mathtt F \subset E_{\circ}(S) \triangle E_{\circ}(K) } } \frac{ \lambda^{2|\mathtt E|} \mu^{2|\mathtt F|} }{ n^{|\mathtt E|+|\mathtt F|} } \mathbb E\Bigg[ & \prod_{ (i,j) \in \mathtt E } x_i^2 \prod_{ (i,j) \in \mathtt F } y_i^2 \prod_{ (i,j) \in E_{\bullet}(S) \triangle E_{\bullet}(K) } x_i \prod_{ (i,j) \in E_{\circ}(S) \triangle E_{\circ}(K) } y_i \\
    & \prod_{ (i,j) \in \mathtt E } \bm u_j^2 \prod_{ (i,j) \in \mathtt F } \bm v_j^2 \prod_{ (i,j) \in E_{\bullet}(S) \triangle E_{\bullet}(K) } \bm u_j \prod_{ (i,j) \in E_{\circ}(S) \triangle E_{\circ}(K) } \bm v_j \Bigg] \,,
\end{align*}
which can be re-written as
\begin{align}
    \sum_{ \substack{ \mathtt E \subset E_{\bullet}(S) \triangle E_{\bullet}(K) \\ \mathtt F \subset E_{\circ}(S) \triangle E_{\circ}(K) } } \big( \tfrac{\lambda^2}{n} \big)^{|\mathtt E|} \big( \tfrac{\mu^2}{n} \big)^{|\mathtt F|} & \mathbb E\Bigg[ \prod_{ i \in V^{\mathsf a}(S) \cup V^{\mathsf a}(K) } x_i^{ \mathtt h(i;\mathtt E,\mathtt F;S,K) } y_i^{ \mathtt k(i;\mathtt E,\mathtt F;S,K) } \Bigg] \nonumber \\
    & \mathbb E\Bigg[ \prod_{ j \in V^{\mathsf b}(S) \cup V^{\mathsf b}(K) } \bm u_j^{ \mathtt h(j;\mathtt E,\mathtt F;S,K) } \bm v_j^{ \mathtt k(i;\mathtt E,\mathtt F;S,K) } \Bigg] \,, \label{eq-goal-cov-f-S-f-K-cov-transfer-1}
\end{align}
where $\mathtt h(i;\mathtt E,\mathtt F;S,K),\mathtt k(i;\mathtt E,\mathtt F;S,K)$ is defined in \eqref{eq-def-mathtt-h-k}. Following the same calculation as in Section~\ref{subsec:proof-lem-3.3}, we see that
\begin{align*}
    & \mathbb E\Bigg[ \prod_{ i \in V^{\mathsf a}(S) \cup V^{\mathsf a}(K) } x_i^{ \mathtt h(i;\mathtt E,\mathtt F;S,K) } y_i^{ \mathtt k(i;\mathtt E,\mathtt F;S,K) } \Bigg] \\
    \leq\ & \rho^{ |\mathsf{diff}(S) \setminus V(K)| + |\mathsf{diff}(K) \setminus V(S)| } C^{ |V^{\mathsf a}(S) \cap V^{\mathsf a}(K)| -|V^{\mathsf a}(S_{\bullet} \cap K_{\bullet}) \cup V^{\mathsf a}(S_{\circ} \cap K_{\circ})| + 2|\mathtt E|+2|\mathtt F| } 
\end{align*}
and (note that $V^{\mathsf b}(S) \cap \mathsf{diff}(S)=V^{\mathsf b}(K) \cap \mathsf{diff}(K)=\emptyset$)
\begin{align*}
    & \mathbb E\Bigg[ \prod_{ j \in V^{\mathsf b}(S) \cup V^{\mathsf b}(K) } \bm u_j^{ \mathtt h(j;\mathtt E,\mathtt F;S,K) } \bm v_j^{ \mathtt k(i;\mathtt E,\mathtt F;S,K) } \Bigg] \\
    \leq\ & 2^{ |V^{\mathsf b}(S) \cap V^{\mathsf b}(K)| -|V^{\mathsf b}(S_{\bullet} \cap K_{\bullet}) \cup V^{\mathsf b}(S_{\circ} \cap K_{\circ})| + 2|\mathtt E|+2|\mathtt F| }  \,.
\end{align*}
Thus, we have that
\begin{align*}
    \eqref{eq-goal-cov-f-S-f-K-cov-transfer-1} &= \sum_{ \substack{ \mathtt E \subset E_{\bullet}(S) \triangle E_{\bullet}(K) \\ \mathtt F \subset E_{\circ}(S) \triangle E_{\circ}(K) } } \big( \tfrac{\lambda^2}{n} \big)^{|\mathtt E|} \big( \tfrac{\mu^2}{n} \big)^{|\mathtt F|} \rho^{ |\mathsf{diff}(S) \setminus V(K)| + |\mathsf{diff}(K) \setminus V(S)| } \\
    & \quad\quad\quad\quad\quad\quad\quad \cdot (2C)^{ |V(S) \cap V(K)| -|V(S_{\bullet} \cap K_{\bullet}) \cup V(S_{\circ} \cap K_{\circ})| + 2|\mathtt E|+2|\mathtt F| }  \\
    &= [1+o(1)] \rho^{ |\mathsf{diff}(S) \setminus V(K)| + |\mathsf{diff}(K) \setminus V(S)| } (2C)^{ |V(S) \cap V(K)| -|V(S_{\bullet} \cap K_{\bullet}) \cup V(S_{\circ} \cap K_{\circ})| }  \,,
\end{align*}
leading to \eqref{eq-goal-cov-f-S-f-K-cov} and thus finishing the proof of Lemma~\ref{lem-est-cov-f-S-f-K-cov}.

\subsection{Proof of Lemma~\ref{lem-detection-cov-most-technical}}{\label{subsec:proof-lem-3.8}}

We first prove \eqref{eq-bound-var-Pb-f-mathcal-G-relax-3-Part-1}. Note by \eqref{eq-def-mathtt-P} that $\mathtt P(S,S)=1$. Thus, 
\begin{align*}
    \eqref{eq-var-Pb-f-mathcal-G-relax-3-Part-1} &= \sum_{ S \in \mathsf K_{N,n}: [S] \in \mathcal G } \frac{ \Upsilon(S)^2 }{ \gamma^{\ell} N^{\ell} n^{\ell} \beta_{\mathcal G}^2 } = \sum_{ [H] \in \mathcal G } \frac{ \Upsilon(H)^2 }{ \gamma^{-\ell} N^{\ell} n^{\ell} \beta_{\mathcal G}^2 } \cdot \#\{ S \in \mathsf{K}_{n,N}: S \cong H \} \\
    &= [1+o(1)] \cdot \sum_{ [H] \in \mathcal G } \frac{ \Upsilon(H)^2 }{ \gamma^{-\ell} N^{\ell} n^{\ell} \beta_{\mathcal G}^2 } \cdot \frac{ N^{\ell} n^{\ell} }{ |\mathsf{Aut}(H)| } \overset{\eqref{eq-def-beta-mathcal-G}}{=} \frac{1+o(1)}{\gamma^{-\ell}\beta_{\mathcal G}} \overset{\text{Lemma~\ref{lem-bound-beta-mathcal-G}}}{=} o(1) \,,
\end{align*}
leading to \eqref{eq-bound-var-Pb-f-mathcal-G-relax-3-Part-1}. Now we prove \eqref{eq-bound-var-Pb-f-mathcal-G-relax-3-Part-2}. Similarly as in \eqref{eq-trivial-bound-mathtt-M}, we can show that 
\begin{align*}
    \mathtt P(S,K)\leq n^{o(1)} \cdot \Upsilon(S)\Upsilon(K) \,.
\end{align*}
Thus, we have
\begin{align}
    \eqref{eq-var-Pb-f-mathcal-G-relax-3-Part-2} \leq \sum_{ \substack{ S,K \subset \mathsf K_{n,N}: [S],[K] \in \mathcal H \\ V(S) \cap V(K) \neq \emptyset, S \neq K } } \frac{ \Upsilon(S)^2 \Upsilon(K)^2 }{ \gamma^{-\ell} N^{\ell} n^{3\ell-|E_{\bullet}(S)\cap E_{\bullet}(K)|-|E_{\circ}(S) \cap E_{\circ}(K)|-o(1)} \beta_{\mathcal G}^2 } \,.  \label{eq-var-Pb-f-mathcal-G-relax-3-Part-2-simplify-1}
\end{align}
Note that for $S \neq K$ and $V(S) \cap V(K) \neq \emptyset$, we must have
\begin{align*}
    & |E_{\bullet}(S) \cap E_{\bullet}(K)|+|E_{\circ}(S) \cap E_{\circ}(K)| \\
    \leq\ & |E(S) \cap E(K)| \leq \min\big\{ |V^{\mathsf a}(S) \cap V^{\mathsf a}(K)|, |V^{\mathsf b}(S) \cap V^{\mathsf b}(K)| \big\} \,.
\end{align*}
Thus, we have
\begin{align}
    \eqref{eq-var-Pb-f-mathcal-G-relax-3-Part-2-simplify-1} = \sum_{ \substack{ S \in \mathsf{K}_{n,N}: [S] \in \mathcal G \\ 0 \leq p,q \leq \ell, 1\leq r+t \leq \ell \\ p+q \leq \min\{ r,t \} } } \sum_{ [H] \in \mathcal H } \frac{ \Upsilon(S)^2 \Upsilon(H)^2 }{ \gamma^{-\ell} N^{\ell} n^{3\ell-p-q-o(1)} \beta_{\mathcal G}^2 } \cdot \mathsf{ENUM}(S,[H];r,t) \,,  \label{eq-var-Pb-f-mathcal-G-relax-3-Part-2-simplify-2}
\end{align}
where 
\begin{equation}{\label{eq-def-ENUM-detection-cov}}
    \mathsf{ENUM}(S,[H];r,t) = \#\{ K \in \mathsf{K}_n: K \cong H, |V^{\mathsf a}(S) \cap V^{\mathsf a}(K)|=r, |V^{\mathsf b}(S) \cap V^{\mathsf b}(K)|=t \} \,.
\end{equation}
Note that we have at most $\binom{\ell}{r}\binom{\ell}{t}$ ways to choose $V^{\mathsf a}(K) \cap V^{\mathsf a}(S),V^{\mathsf b}(K) \cap V^{\mathsf b}(S)$ and at most $\binom{n-\ell}{\ell-r}\binom{N-\ell}{\ell-t}$ ways to choose $V^{\mathsf a}(K) \setminus V^{\mathsf a}(S),V^{\mathsf b}(K) \setminus V^{\mathsf b}(S)$. In addition, given $V(K)$ we have at most $\frac{(\ell!)^2}{|\mathsf{Aut}(H)|}$ ways to choose $K \cong H$. Thus, 
\begin{align*}
    \mathsf{ENUM}(S,[H];r,t) \leq \binom{\ell}{r}\binom{\ell}{t} \binom{n-\ell}{\ell-r}\binom{N-\ell}{\ell-t} \cdot \frac{(\ell!)^2}{|\mathsf{Aut}(H)|} \leq \frac{ \ell^{2t+2r} n^{\ell-r} N^{\ell-t} }{ |\mathsf{Aut}(H)| } \,.
\end{align*}
Plugging this bound into \eqref{eq-var-Pb-f-mathcal-G-relax-3-Part-2-simplify-2}, we get that
\begin{align}
    \eqref{eq-var-Pb-f-mathcal-G-relax-3-Part-2-simplify-2} &\leq \sum_{ \substack{ S \in \mathsf{K}_{n,N}: [S] \in \mathcal G \\ 0 \leq p,q \leq \ell, 1\leq r+t \leq \ell \\ p+q \leq \min\{ r,t \} } } \sum_{ [H] \in \mathcal G } \frac{ \Upsilon(S)^2 \Upsilon(H)^2 }{ \gamma^{-\ell} N^{\ell} n^{3\ell-p-q-o(1)} \beta_{\mathcal G}^2 } \cdot \frac{ \ell^{2t+2r} n^{\ell-r} N^{\ell-t} }{ |\mathsf{Aut}(H)| }  \nonumber \\
    &\overset{\eqref{eq-def-beta-mathcal-G}}{=} \sum_{ \substack{ S \in \mathsf{K}_{n,N}: [S] \in \mathcal G \\ 0 \leq p,q \leq \ell, 1\leq r+t \leq \ell \\ p+q \leq \min\{ r,t \} } } \frac{ \Upsilon(S)^2 \ell^{2r+2t} \gamma^{t} n^{-r-t}  }{ N^{\ell} n^{\ell} n^{-p-q-o(1)} \beta_{\mathcal G} } \,.  \label{eq-var-Pb-f-mathcal-G-relax-3-Part-2-simplify-3}
\end{align}
From \eqref{eq-condition-strong-detection} we see that $\ell^{2r+2t} \gamma^{t} \leq \ell^{4\ell} \gamma^{\ell} =n^{o(1)}$, thus
\begin{align*}
    \eqref{eq-var-Pb-f-mathcal-G-relax-3-Part-2-simplify-3} &= \sum_{ [H] \in \mathcal G } \frac{ \Upsilon(S)^2 }{ N^{\ell} n^{\ell} \beta_{\mathcal G} } \sum_{ S \in \mathsf{K}_{n,N}: S \cong H } \sum_{ \substack{ 0 \leq p,q \leq \ell, 1 \leq r+t \leq 2\ell \\ p+q \leq \min\{ r,t \} } } \frac{ 1 }{ n^{t+r-p-q+o(1)} } \\
    &= \sum_{ [H] \in \mathcal G } \frac{ \Upsilon(H)^2 }{ n^{\ell+1+o(1)} \beta_{\mathcal H} } \cdot \#\{ S \in \mathsf{K}_n: S \cong H \} = \sum_{ [H] \in \mathcal G } \frac{ \Xi(H)^2 }{ N^{\ell} n^{\ell+1-o(1)} \beta_{\mathcal G} } \cdot \frac{ N^{\ell} n^{\ell} }{ |\mathsf{Aut}(H)| } \overset{\eqref{eq-def-beta-mathcal-G}}{=} n^{-1+o(1)} \,,
\end{align*}
leading to \eqref{eq-bound-var-Pb-f-mathcal-G-relax-3-Part-2}.

\subsection{Proof of Lemma~\ref{lem-est-cov-f-S-f-K-chain-case}}{\label{subsec:proof-lem-3.11}}

The proof of Lemma~\ref{lem-est-cov-f-S-f-K-chain-case} is highly similar to the proof of Lemma~\ref{lem-est-cov-f-S-f-K}, so we will only highlight provide an outline with the main differences while adapting arguments from Lemma~\ref{lem-est-cov-f-S-f-K} without presenting full details. Similarly as in \eqref{eq-cov-f-S-f-K-simplify-1}, we have
\begin{align}
    \mathbb E_{\Pb}[ h_S h_K x_u^2 x_v^2 ] =\ & \mathbb E\Bigg[ x_u^2 x_v^2 \cdot \prod_{ (i,j) \in E_{\bullet}(S) \cap E_{\bullet}(K) } \big( \tfrac{\lambda^2 x_i^2 x_j^2}{n}+1 \big) \prod_{ (i,j) \in E_{\bullet}(S) \triangle E_{\bullet}(K) } \big( \tfrac{\lambda}{\sqrt{n}} x_i x_j \big)  \nonumber \\
    & \prod_{ (i,j) \in E_{\circ}(S) \cap E_{\circ}(K) } \big( \tfrac{\mu^2 y_i^2 y_j^2}{n} + 1 \big)^2 \prod_{ (i,j) \in E_{\circ}(S) \triangle E_{\circ}(K) } \big( \tfrac{\mu}{\sqrt{n}} y_i y_j \big) \Bigg] \,. \label{eq-cov-f-S-f-K-chain-case-simplify-1}
\end{align}
Using similar arguments as in \eqref{eq-goal-cov-f-S-f-K-transfer-1}, we can write \eqref{eq-cov-f-S-f-K-chain-case-simplify-1} into
\begin{align}
    \sum_{ \substack{ \mathtt E \subset E_{\bullet}(S) \triangle E_{\bullet}(K) \\ \mathtt F \subset E_{\circ}(S) \triangle E_{\circ}(K) } } \big( \tfrac{\lambda^2}{n} \big)^{|\mathtt E|} \big( \tfrac{\mu^2}{n} \big)^{|\mathtt F|} \mathbb E\Bigg[ \prod_{ i \in V(S) \cup V(K) } x_i^{ \mathtt h(i;\mathtt E,\mathtt F;S,K) } y_i^{ \mathtt k(i;\mathtt E,\mathtt F;S,K) } \Bigg] \,, \label{eq-goal-cov-f-S-f-K-chain-case-transfer-1}
\end{align}
where (below we use $i \sim e$ to denote that the vertex $i$ is incident to the edge $e$)
\begin{align*}
    & \mathtt h(i;\mathtt E,\mathtt F;S,K) = 2 \#\{ e \in \mathtt E: i \sim e \} + \#\{ e \in E_{\bullet}(S) \triangle E_{\bullet}(K): i \sim e \} + 2\cdot \mathbf 1_{i\in\{u,v \}}  \,; \\
    & \mathtt k(i;\mathtt E,\mathtt F;S,K) = 2 \#\{ e \in \mathtt F: i \sim e \} + \#\{ e \in E_{\circ}(S) \triangle E_{\circ}(K): i \sim e \} \,.
\end{align*}
In addition, denote
\begin{align*}
    V_{\mathsf{bad}} = \{ u,v \} \cap \{ i \in V(S) \triangle V(K): i \sim e \mbox{ for some } e \in \mathtt E \cup \mathtt F \} \,.
\end{align*}
We can bound \eqref{eq-goal-cov-f-S-f-K-chain-case-transfer-1} exactly by dividing to the different cases as in the proof of Lemma~\ref{lem-est-cov-f-S-f-K}. The only difference is that here for $i\in\{ u,v \}$, the degree of $i$ in $S \cup K$ equals $2$ (rather than $4$ in the proof of Lemma~\ref{lem-est-cov-f-S-f-K}), so from Assumption~\ref{assum-upper-bound} we still have 
\begin{align*}
    \Big( \mathtt h(i;\mathtt E,\mathtt F;S,K)+\mathtt k(i;\mathtt E,\mathtt F;S,K) \Big) \leq 4 \Longrightarrow \mathbb E\Big[ x_i^{ \mathtt h(i;\mathtt E,\mathtt F;S,K) } y_i^{ \mathtt k(i;\mathtt E,\mathtt F;S,K) } \Big] \leq C \,.
\end{align*}

\subsection{Proof of Lemma~\ref{lem-bound-Xi-mathtt-M}}{\label{subsec:proof-lem-3.12}}

Recall \eqref{eq-def-Xi-S}. It is straightforward to check that 
\begin{align}
    & |E_{\bullet}(S)| = \sum_{ 1 \leq \mathtt t \leq \mathtt T } |E_{\bullet}(S_{[\mathtt p_{\mathtt t},\mathtt q_{\mathtt t}]})| + \sum_{ 1 \leq \mathtt t \leq \mathtt T-1 } |E_{\bullet}(S_{[\mathtt q_{\mathtt t},\mathtt p_{\mathtt t+1}]})| \,, \label{eq-equality-E_1(S)} \\
    & |E_{\circ}(S)| = \sum_{ 1 \leq \mathtt t \leq \mathtt T } |E_{\circ}(S_{[\mathtt p_{\mathtt t},\mathtt q_{\mathtt t}]})| + \sum_{ 1 \leq \mathtt t \leq \mathtt T-1 } |E_{\circ}(S_{[\mathtt q_{\mathtt t},\mathtt p_{\mathtt t+1}]})| \,.  \label{eq-equality-E_2(S)}
\end{align}
In addition, for all $1 \leq \mathtt k \leq \ell+1$ we have 
\begin{align*}
    \sum_{ 1 \leq \mathtt t \leq \mathtt T } \mathbf 1(v_{\mathtt k} \in \mathsf{diff}(S_{[\mathtt p_{\mathtt t},\mathtt q_{\mathtt t}]})) + \sum_{ 1 \leq \mathtt t \leq \mathtt T-1 } \mathbf 1(v_{\mathtt k} \in S_{[\mathtt q_{\mathtt t},\mathtt p_{\mathtt t+1}]}) \leq \mathbf 1(v_{\mathtt k} \in \mathsf{diff}(S)) + 1 \,,
\end{align*}
and the equality can only holds if $\mathtt k \in \{ \mathtt p_1, \mathtt q_1, \ldots, \mathtt p_{\mathtt T}, \mathtt q_{\mathtt T} \}$. Thus, we have
\begin{align}
    |\mathsf{diff}(S)| &= \sum_{ 1 \leq \mathtt k \leq \ell+1 } \mathbf 1(v_{\mathtt k} \in \mathsf{diff}(S)) \nonumber \\
    &\geq \sum_{ 1 \leq \mathtt k \leq \ell+1 } \Big( \sum_{ 1 \leq \mathtt t \leq \mathtt T } \mathbf 1(v_{\mathtt k} \in \mathsf{diff}(S_{[\mathtt p_{\mathtt t},\mathtt q_{\mathtt t}]})) + \sum_{ 1 \leq \mathtt t \leq \mathtt T-1 } \mathbf 1(v_{\mathtt k} \in \mathsf{diff}(S_{[\mathtt q_{\mathtt t},\mathtt p_{\mathtt t+1}]}) ) \Big) - 2 \mathtt T \nonumber \\
    &= \sum_{ 1 \leq \mathtt t \leq \mathtt T } |\mathsf{diff}(S_{[\mathtt p_{\mathtt t},\mathtt q_{\mathtt t}]}))| + \sum_{ 1 \leq \mathtt t \leq \mathtt T-1 } |\mathsf{diff}(S_{[\mathtt q_{\mathtt t},\mathtt p_{\mathtt t+1}]})| - 2 \mathtt T \,. \label{eq-inequality-Dis(S)}
\end{align}
Plugging \eqref{eq-equality-E_1(S)}--\eqref{eq-inequality-Dis(S)} into \eqref{eq-def-Xi-S} yields \eqref{eq-bound-Xi(S)}. Similarly we can show \eqref{eq-bound-Xi(K)}.

Now we turn to \eqref{eq-bound-mathtt-M(S,K)}. Note that
\begin{align*}
    E_{\bullet}( S_{[\mathtt q_{\mathtt t},\mathtt p_{\mathtt t+1}]} ), E_{\bullet}( K_{[\mathtt q_{\mathtt t}',\mathtt p_{\mathtt t+1}']} ) \subset E_{\bullet}(S) \triangle E_{\bullet}(K) \mbox{ and } E_{\bullet}( S_{[\mathtt q_{\mathtt t},\mathtt p_{\mathtt t+1}]} ), E_{\bullet}( K_{[\mathtt q_{\mathtt t}',\mathtt p_{\mathtt t+1}']} ) \mbox{ are disjoint} \,,
\end{align*}
we have
\begin{align}
    |E_{\bullet}(S) \triangle E_{\bullet}(K)| \geq \sum_{1 \leq \mathtt t \leq \mathtt T-1} |E_{\bullet}( S_{[\mathtt q_{\mathtt t},\mathtt p_{\mathtt t+1}]} )| + \sum_{1 \leq \mathtt t \leq \mathtt T-1} |E_{\bullet}( K_{[\mathtt q_{\mathtt t}', \mathtt p_{\mathtt t+1}']} )| \,.  \label{eq-inequality-E_1(S)-triangle-E_1(K)}
\end{align}
Similarly, we have
\begin{align}
    |E_{\circ}(S) \triangle E_{\circ}(K)| \geq \sum_{1 \leq \mathtt t \leq \mathtt T-1} |E_{\circ}( S_{[\mathtt q_{\mathtt t},\mathtt p_{\mathtt t+1}]} )| + \sum_{1 \leq \mathtt t \leq \mathtt T-1} |E_{\circ}( K_{[\mathtt q_{\mathtt t}', \mathtt p_{\mathtt t+1}']} )| \,.  \label{eq-inequality-E_2(S)-triangle-E_2(K)}
\end{align}
In addition, note that $\mathsf{diff}( S_{[\mathtt q_{\mathtt t},\mathtt p_{\mathtt t+1}]} ) \setminus \{ v_{\mathtt q_{\mathtt t}}, v_{\mathtt p_{\mathtt t+1}} \} \subset \mathsf{diff}(S) \setminus V(K)$ are disjoint (and similarly for $\mathsf{diff}(K) \setminus V(S)$), we have
\begin{align}
    & |\mathsf{diff}(S) \setminus V(K)| \geq \sum_{1 \leq \mathtt t \leq \mathtt T-1} |\mathsf{diff}( S_{[\mathtt q_{\mathtt t},\mathtt p_{\mathtt t+1}]} )|-2\mathtt T \,; \label{eq-inequality-dif(S)-setminus-V(K)} \\
    & |\mathsf{diff}(K) \setminus V(S)| \geq \sum_{1 \leq \mathtt t \leq \mathtt T-1} |\mathsf{diff}( K_{[\mathtt q_{\mathtt t}',\mathtt p_{\mathtt t+1}']} )|-2\mathtt T \,. \label{eq-inequality-dif(K)-setminus-V(S)}
\end{align}
Finally, recall \eqref{eq-def-SEQ(S,K)} and note that $\sigma(\mathtt i)=1$ implies that $v_{\mathtt i} \in V(S_{\bullet} \cap K_{\bullet}) \cup V(S_{\circ} \cap K_{\circ})$, thus
\begin{align}
    & |V(S) \cap V(K)| - |V(S_{\bullet} \cap K_{\bullet}) \cup V(S_{\circ} \cap K_{\circ})| \nonumber \\
    =\ & \sum_{ 1 \leq \mathtt i \leq \ell } \mathbf 1( v_{\mathtt i} \in V(S) \cap V(K), v_{\mathtt i} \not\in V(S_{\bullet} \cap K_{\bullet}) \cup V(S_{\circ} \cap K_{\circ}) ) \nonumber \\
    \leq\ & \sum_{ 1 \leq \mathtt t \leq \mathtt T } \sum_{ \mathtt p_{\mathtt t} \leq \mathtt i \leq \mathtt q_{\mathtt t} } 1-\mathbf 1(\sigma(\mathtt i)=1) \overset{\eqref{eq-def-aleph(sigma)}}{\leq} \aleph(\sigma)+\mathtt T \,. \label{eq-inequality-V(K)-cap-V(S)-setminus-V(S_i-cap-K_i)}
\end{align}
Plugging \eqref{eq-inequality-dif(S)-setminus-V(K)}--\eqref{eq-inequality-V(K)-cap-V(S)-setminus-V(S_i-cap-K_i)} into \eqref{eq-def-mathtt-M} yields \eqref{eq-bound-mathtt-M(S,K)}.

\subsection{Proof of Lemma~\ref{lem-recovery-most-technical}}{\label{subsec:proof-lem-3.13}}

Before proving Lemma~\ref{lem-recovery-most-technical}, we first show the following bound which will be used repeatedly in the later proof.
\begin{lemma}{\label{lem-contribution-sub-chain}}
    For all $\mathtt b \geq \mathtt a$ and given $v_{\mathtt a},v_{\mathtt b}$, we have for some constant $D=D(\lambda,\mu,\rho)$
    \begin{align*}
        \sum_{ S_{[\mathtt a,\mathtt b]}: \mathsf{L}(S_{[\mathtt a,\mathtt b]}) } \Xi(S_{[\mathtt a,\mathtt b]})^2 \leq D n^{\mathtt b-\mathtt a-1+\mathbf 1_{\mathtt b=\mathtt a}} A_+^{\mathtt b-\mathtt a} \,.
    \end{align*}
\end{lemma}
\begin{proof}
    Note that we must have $S_{[\mathtt a,\mathtt b]} \in \mathcal J_{\star}(\mathtt b-\mathtt a)$. Thus, we have
    \begin{align}
        \sum_{ S_{[\mathtt a,\mathtt b]}: \mathsf{L}(S_{[\mathtt a,\mathtt b]}) } \Xi(S_{[\mathtt a,\mathtt b]})^2 = \sum_{ [J] \in \mathcal J_{\star}(\mathtt b-\mathtt a) } \Xi(J)^2 \cdot \#\Big\{ S \in \mathsf K_n: S \cong J, \mathsf{L}(S)=\{ v_{\mathtt a},v_{\mathtt b} \} \Big\} \,.  \label{eq-contribution-sub-chain-relax}
    \end{align}
    Note that to choose $S \in \mathsf K_n$ with $S \cong J$ and $\mathsf{L}(S)=\{ v_{\mathtt a},v_{\mathtt b} \}$, we have at most $\binom{n}{\mathtt b-\mathtt a-1+\mathbf 1_{\mathtt b=\mathtt a}}$ ways to choose $V(S) \setminus \mathsf{L}(S)$; in addition, given $V(S)$ we have at most $\frac{ (\mathtt b-\mathtt a-1+\mathbf 1_{\mathtt b=\mathtt a})! }{ |\mathsf{Aut}(J)| }$ ways to choose $S$. Thus, we have
    \begin{align*}
        \#\Big\{ S \in \mathsf K_n: S \cong J, \mathsf{L}(S)=\{ v_{\mathtt a},v_{\mathtt b} \} \Big\} &= \binom{n}{\mathtt b-\mathtt a-1+\mathbf 1_{\mathtt b=\mathtt a}} \cdot \frac{ (\mathtt b-\mathtt a-1+\mathbf 1_{\mathtt b=\mathtt a})! }{ |\mathsf{Aut}(J)| } \\
        &= \frac{ [1+o(1)] n^{\mathtt b-\mathtt a-1+\mathbf 1_{\mathtt b=\mathtt a}} }{ |\mathsf{Aut}(J)| } \,.
    \end{align*}
    Plugging this estimation into \eqref{eq-contribution-sub-chain-relax}, we get that
    \begin{align*}
        \eqref{eq-contribution-sub-chain-relax} = [1+o(1)] n^{\mathtt b-\mathtt a-1+\mathbf 1_{\mathbb b=\mathtt a}} \cdot \sum_{ [J] \in \mathcal J_{\star}(\mathtt b-\mathtt a) } \frac{ \Xi(J)^2 }{ |\mathsf{Aut}(J)| } \leq D n^{\mathtt b-\mathtt a-1+\mathbf 1_{\mathtt b=\mathtt a}} A_+^{\mathtt b-\mathtt a} \,,
    \end{align*}
    where the last inequality follows from Lemma~\ref{lem-bound-beta-mathcal-J}.
\end{proof}

Now we return to the proof of Lemma~\ref{lem-recovery-most-technical}. We will split \eqref{eq-var-Phi-i,j-relax-3} into several parts and bound each part separately.

\noindent {\bf Part 1: the case $\mathtt T=1$.} In this case, we have $V(S)=V(K)$ and thus
\begin{align*}
    \mathsf{SEQ}=\overline{\mathsf{SEQ}}=(\mathtt p_1,\mathtt q_1)=(1,\ell+1), \quad S_{\setminus}=K_{\setminus}=\emptyset, \quad S_{\cap}=S \,.
\end{align*}
Thus, the contribution of this part in \eqref{eq-var-Phi-i,j-relax-3} is bounded by
\begin{align}
    & \sum_{ [S] \in \mathcal J: \mathsf{L}(S)=\{ i,j \} } \frac{ (\rho^{-8}C) n^{\ell} \Xi(S)^2 }{ n^{2\ell-2} \beta_{\mathcal J}^2 } \leq \frac{ (32\rho^{-8}C) }{ n^{\ell-2} \beta_{\mathcal J}^2 } \sum_{ [S] \in \mathcal J_{\star}(\ell): \mathsf{L}(S)=\{ i,j \} } \Xi(S)^2 \nonumber \\
    \leq\ & \frac{ (\rho^{-8}C) }{ n^{\ell-2} \beta_{\mathcal J}^2 } \cdot D n^{\ell-1} A_{+}^{\ell} \leq \rho^{-8}CD^3 n \cdot A_{+}^{\ell} \overset{\eqref{eq-condition-weak-recovery},\text{Lemma~\ref{lem-bound-beta-mathcal-J}}}{=} o(1) \,,  \label{eq-bound-part-1-contribution-var-Phi-i,j}
\end{align}
where the first inequality follows from $\mathcal J \subset \mathcal J_{\star}$, the second inequality follows from taking $(\mathtt a,\mathtt b)=(1,\ell+1)$ in Lemma~\ref{lem-contribution-sub-chain} and the third inequality follows from Lemma~\ref{lem-bound-beta-mathcal-J}.

\noindent {\bf Part 2: the case $\mathtt T=2$.} In this case, we have
\begin{align*}
    & \mathsf{SEQ}=\overline{\mathsf{SEQ}}=( \mathtt p_1,\mathtt q_1,\mathtt p_2,\mathtt q_2 ) = (1,\mathtt q,\mathtt p,\ell+1) \,; \\
    & S_{\cap} = S_{[1,\mathtt q]} \cup S_{[\mathtt p,\ell+1]}, \quad S_{\setminus}=S_{[\mathtt q,\mathtt p]}, \quad K_{\setminus}=K_{[\mathtt q,\mathtt p]} \,. 
\end{align*}
Thus, the contribution of this part in \eqref{eq-var-Phi-i,j-relax-3} is bounded by
\begin{align}
    & \sum_{ \mathtt q<\mathtt p } \sum_{ \substack{ S_{[1,\mathtt q]}, S_{[\mathtt p,\ell+1]}, S_{[\mathtt q,\mathtt p]},  K_{[\mathtt q,\mathtt p]} } } \frac{ (\rho^{-8}C)^2 n^{(\mathtt q-1)+(\ell+1-\mathtt p)} }{ n^{2\ell-2} \beta_{\mathcal J}^2 } \cdot \Xi(S_{[1,\mathtt q]})^2 \Xi(S_{[\mathtt p,\ell+1]})^2 \Xi(S_{[\mathtt q,\mathtt p]})^2 \Xi(K_{[\mathtt q,\mathtt p]})^2 \nonumber \\
    =\ & \sum_{ 1 \leq \mathtt q<\mathtt p \leq \ell+1 } \sum_{ \substack{ S_{[1,\mathtt q]}, S_{[\mathtt p,\ell+1]}, S_{[\mathtt q,\mathtt p]},  K_{[\mathtt q,\mathtt p]} } } \frac{ (\rho^{-8}C)^2 }{ n^{\ell-2+\mathtt p-\mathtt q} \beta_{\mathcal J}^2 } \cdot \Xi(S_{[1,\mathtt q]})^2 \Xi(S_{[\mathtt p,\ell+1]})^2 \Xi(S_{[\mathtt q,\mathtt p]})^2 \Xi(K_{[\mathtt q,\mathtt p]})^2 \,. \label{eq-bound-part-2-contribution-var-Phi-i,j-simplify-1}
\end{align}
Note that there are at most $n^{ 2-\mathbf 1_{\mathtt q=1}-\mathbf 1_{\mathtt p=\ell+1} }$ ways to choose $v_{\mathtt q},v_{\mathtt p}$. In addition, given $(v_{\mathtt q},v_{\mathtt p})$ we have (note that $v_{1}=i,v_{\ell+1}=j$)
\begin{align*}
    &\sum_{ \substack{ \mathsf{L}(S_{[1,\mathtt q]})=\{ v_1,v_{\mathtt q} \} \\ \mathsf{L}(S_{[\mathtt p,\ell+1]}) =\{ v_{\mathtt p},v_{\ell+1} \} \\ \mathsf{L}(S_{[\mathtt q,\mathtt p]}) = \{ v_{\mathtt q}, v_{\mathtt p} \} \\ \mathsf{L}(K_{[\mathtt q,\mathtt p]})= \{ v_{\mathtt q}, v_{\mathtt p} \} }  } \Xi((S_{[1,\mathtt q]}))^2 \Xi(S_{[\mathtt p,\ell+1]})^2 \Xi(S_{[\mathtt q,\mathtt p]})^2 \Xi(K_{[\mathtt q,\mathtt p]})^2  \\
    \overset{\text{Lemma~\ref{lem-contribution-sub-chain}}}{\leq}\ & D^4 \cdot n^{\mathtt q-2+\mathbf 1_{\mathtt q=1}} A_+^{\mathtt q-1} \cdot n^{\ell-\mathtt p+\mathbf 1_{\mathtt p=\ell+1}} A_+^{\ell+1-\mathtt p} \cdot n^{\mathtt p-\mathtt q-1} A_+^{\mathtt p-\mathtt q} \cdot n^{\mathtt p-\mathtt q-1} A_+^{\mathtt p-\mathtt q}  \\
    \leq\ & D^4 n^{\ell+\mathtt p-\mathtt q-4+\mathbf 1_{\mathtt q=1}+\mathbf 1_{\mathtt p=\ell+1}} A_+^{\ell+1+\mathtt p-\mathtt q}
\end{align*}
Thus, we see that
\begin{align}
    \eqref{eq-bound-part-2-contribution-var-Phi-i,j-simplify-1} &\leq (\rho^{-8}C)^2 \sum_{1 \leq \mathtt q<\mathtt p\leq \ell+1} \frac{ n^{ 2-\mathbf 1_{\mathtt q=1}-\mathbf 1_{\mathtt p=\ell+1} } \cdot D^4 n^{\ell+\mathtt p-\mathtt q-4+\mathbf 1_{\mathtt q=1}+\mathbf 1_{\mathtt p=\ell+1}} A_+^{\ell+1+\mathtt p-\mathtt q} }{ n^{\ell-2+\mathtt p-\mathtt q} \beta_{\mathcal J}^2 } \nonumber \\
    &\leq D^4 (\rho^{-8}C)^2 \sum_{1 \leq \mathtt q<\mathtt p\leq \ell+1} \frac{ A_+^{\ell+1+\mathtt p-\mathtt q} }{ \beta_{\mathcal J}^2 } \overset{\text{Lemma~\ref{lem-bound-beta-mathcal-J}}}{\leq} R^4 (\rho^{-8}C)^2 \sum_{1 \leq \mathtt q<\mathtt p\leq \ell+1} A_+^{ -(\mathtt q-1) - (\ell+1-\mathtt p) } \nonumber \\
    &\leq D^4 (\rho^{-8}C)^2 \frac{ A_+^2 }{ (A_+ -1)^{2} } \overset{\text{Lemma~\ref{lem-bound-beta-mathcal-J}}}{\leq} O_{\lambda,\mu,\rho,\delta}(1) \,.  \label{eq-bound-part-2-contribution-var-Phi-i,j}
\end{align}

\noindent{\bf Part 3: the case $\mathtt T \geq 3$.} In this case we have
\begin{align*}
    & \mathsf{SEQ}=( \mathtt p_1,\mathtt q_1,\ldots,\mathtt p_{\mathtt T}, \mathtt q_{\mathtt T} ), \ \overline{\mathsf{SEQ}}=( \mathtt p_1',\mathtt q_1',\ldots,\mathtt p_{\mathtt T}', \mathtt q_{\mathtt T}' ) \mbox{ with } \mathtt p_1=\mathtt p_1'=1, \mathtt q_{\mathtt T}=\mathtt q_{\mathtt T}'=\ell+1 \,; \\
    & S_{\cap} = \cup_{1 \leq \mathtt t \leq \mathtt T} S_{[\mathtt p_{\mathtt t},\mathtt q_{\mathtt t}]}, \quad S_{\setminus} = \cup_{1 \leq \mathtt t \leq \mathtt T-1} S_{[\mathtt q_{\mathtt t},\mathtt p_{\mathtt t+1}]}, \quad \cup_{1 \leq \mathtt t \leq \mathtt T-1} K_{[\mathtt q_{\mathtt t}',\mathtt p_{\mathtt t+1}']} \,.
\end{align*}
Thus, the contribution of this part in \eqref{eq-var-Phi-i,j-relax-3} is bounded by 
\begin{align}
    \sum_{\mathtt T \geq 3} \sum_{ \substack{ \mathsf{SEQ}, \overline{\mathsf{SEQ}} \\ S_{\cap}, S_{\setminus}, K_{\setminus} } } \frac{ (\rho^{-8}C)^{\mathtt T} n^{ \sum_{ 1 \leq \mathtt t \leq \mathtt T } (\mathtt q_{\mathtt t}-\mathtt p_{\mathtt t}) } }{ n^{2\ell-2} \beta_{\mathcal J}^2 } \cdot \Xi(S_{\cap})^2 \Xi(S_{\setminus})^2 \Xi(K_{\setminus})^2 \,.   \label{eq-bound-part-3-contribution-var-Phi-i,j-simplify-1}
\end{align}
Note that the possible choice of $\mathsf{SEQ},\overline{\mathsf{SEQ}}$ is bounded by 
\begin{align*}
    \ell^{2(\mathtt T-1)}
\end{align*}
and the possible choice of $\mathtt V=(v_{\mathtt p_1}, v_{\mathtt q_1},\ldots, v_{\mathtt p_{\mathtt T}}, v_{\mathtt q_{\mathtt T}})$ is bounded by
\begin{align*}
    n^{ 2(\mathtt T-1) - \#\{ 1 \leq \mathtt t \leq \mathtt T: \mathtt q_{\mathtt t} =\mathtt p_{\mathtt t} \} } \,.
\end{align*}
In addition, given $\mathsf{SEQ},\overline{\mathsf{SEQ}}$ and $(v_{\mathtt p_1}, v_{\mathtt q_1},\ldots, v_{\mathtt p_{\mathtt T}}, v_{\mathtt q_{\mathtt T}})$ we have $v_{\mathtt p_{\mathtt t}'}'=v_{\mathtt p_{\mathtt t}}$ and $v_{\mathtt q_{\mathtt t}'}'=v_{\mathtt q_{\mathtt t}}$. Thus, applying Lemma~\ref{lem-contribution-sub-chain} we see that given $\mathsf{SEQ},\overline{\mathsf{SEQ}}$ and $\mathtt V$, 
\begin{align*}
    & \sum_{ S_{\cap}, S_{\setminus}, K_{\setminus} \text{ consistent with } \mathtt V } \Xi(S_{\cap})^2 \Xi(S_{\setminus})^2 \Xi(K_{\setminus})^2 \\
    \leq\ & \prod_{1 \leq \mathtt t \leq \mathtt T} \big( D n^{\mathtt q_{\mathtt t}-\mathtt p_{\mathtt t}-1+\mathbf 1_{ \mathtt p_{\mathtt t}=\mathtt q_{\mathtt t} }} A_+^{\mathtt q_{\mathtt t}-\mathtt p_{\mathtt t}} \big) \prod_{1 \leq \mathtt t \leq \mathtt T-1} \big( D n^{\mathtt p_{\mathtt t+1}-\mathtt q_{\mathtt t}-1 } A_+^{\mathtt p_{\mathtt t+1}-\mathtt q_{\mathtt t}} \big) \prod_{1 \leq \mathtt t \leq \mathtt T-1} \big( D n^{\mathtt p_{\mathtt t+1}'-\mathtt q_{\mathtt t}'-1 } A_+^{\mathtt p_{\mathtt t+1}-\mathtt q_{\mathtt t}} \big)  \\
    =\ & D^{3\mathtt T-2} A_+^{2\ell} \cdot n^{ -\mathtt T+\#\{ 1 \leq \mathtt t \leq \mathtt T: \mathtt p_{\mathtt t}=\mathtt q_{\mathtt t} \} } \cdot n^{-2(\mathtt T-1)} \cdot \prod_{1 \leq \mathtt t \leq \mathtt T} n^{\mathtt q_{\mathtt t}-\mathtt p_{\mathtt t}} \prod_{1 \leq \mathtt t \leq \mathtt T-1} n^{\mathtt p_{\mathtt t+1}-\mathtt q_{\mathtt t}} \prod_{1 \leq \mathtt t \leq \mathtt T-1} n^{\mathtt p_{\mathtt t+1}'-\mathtt q_{\mathtt t}'} \\
    =\ & D^{3\mathtt T-2} A_+^{2\ell} \cdot n^{ -\mathtt T+\#\{ 1 \leq \mathtt t \leq \mathtt T: \mathtt p_{\mathtt t}=\mathtt q_{\mathtt t} \} } \cdot n^{-2(\mathtt T-1)} \cdot n^{ 2\ell-\sum_{1 \leq \mathtt t \leq \mathtt T}(\mathtt q_{\mathtt t}-\mathtt p_{\mathtt t}) } \,,
\end{align*}
where the last equality follows from 
\begin{align*}
    \{ \mathtt q_1' - \mathtt p_1', \ldots, \mathtt q_{\mathtt T}'-\mathtt p_{\mathtt T}' \} = \{ \mathtt q_1 - \mathtt p_1, \ldots, \mathtt q_{\mathtt T}-p_{\mathtt T} \}
\end{align*}
in \eqref{eq-def-overline-SEQ(S,K)}. Thus, we have
\begin{align}
    \eqref{eq-bound-part-3-contribution-var-Phi-i,j-simplify-1} \leq\ & \sum_{\mathtt T \geq 3} \sum_{ \mathsf{SEQ}, \overline{\mathsf{SEQ}} } \frac{ (\rho^{-8}C)^{\mathtt T} n^{ \sum_{ 1 \leq \mathtt t \leq \mathtt T } (\mathtt q_{\mathtt t}-\mathtt p_{\mathtt t}) } }{ n^{2\ell-2} \beta_{\mathcal J}^2 } \cdot n^{ 2(\mathtt T-1) - \#\{ 1 \leq \mathtt t \leq \mathtt T: \mathtt q_{\mathtt t} =\mathtt p_{\mathtt t} \} } \nonumber \\
    & \cdot D^{3\mathtt T-2} A_+^{2\ell} \cdot n^{ -\mathtt T+\#\{ 1 \leq \mathtt t \leq \mathtt T: \mathtt p_{\mathtt t}=\mathtt q_{\mathtt t} \} } \cdot n^{-2(\mathtt T-1)} \cdot n^{ 2\ell-\sum_{1 \leq \mathtt t \leq \mathtt T}(\mathtt q_{\mathtt t}-\mathtt p_{\mathtt t}) } \nonumber \\
    =\ & \sum_{\mathtt T \geq 3} \sum_{ \mathsf{SEQ}, \overline{\mathsf{SEQ}} } \frac{ D^{3\mathtt T-2} (\rho^{-8}C)^{\mathtt T} A_+^{2\ell} }{ n^{\mathtt T-2} \beta_{\mathcal J}^2 } \overset{\text{Lemma~\ref{lem-bound-beta-mathcal-J}}}{\leq} \sum_{\mathtt T \geq 3} \frac{ D^{3\mathtt T} \ell^{2(\mathtt T-1)} (\rho^{-8}C)^{\mathtt T}  }{ n^{\mathtt T-2} } = n^{-1+o(1)} \,.  \label{eq-bound-part-3-contribution-var-Phi-i,j}
\end{align}
Combining \eqref{eq-bound-part-1-contribution-var-Phi-i,j}, \eqref{eq-bound-part-2-contribution-var-Phi-i,j} and \eqref{eq-bound-part-3-contribution-var-Phi-i,j} leads to Lemma~\ref{lem-recovery-most-technical}.

\subsection{Proof of Lemma~\ref{lem-est-cov-f-S-f-K-chain-case-cov}}{\label{subsec:proof-lem-3.16}}

The proof of Lemma~\ref{lem-est-cov-f-S-f-K-chain-case-cov} is highly similar to the proof of Lemma~\ref{lem-est-cov-f-S-f-K-cov}, so we will only highlight provide an outline with the main differences while adapting arguments from Lemma~\ref{lem-est-cov-f-S-f-K-cov} without presenting full details. Similarly as in \eqref{eq-cov-f-S-f-K-cov-simplify-1}, we have
\begin{align}
    \mathbb E_{\Pb}[ h_S h_K x_u^2 x_v^2 ] =\ & \mathbb E\Bigg[ x_u^2 x_v^2 \cdot \prod_{ (i,j) \in E_{\bullet}(S) \cap E_{\bullet}(K) } \big( \tfrac{\lambda x_i^2 \bm u_j^2}{n}+1 \big) \prod_{ (i,j) \in E_{\bullet}(S) \triangle E_{\bullet}(K) } \big( \tfrac{\sqrt{\lambda}}{\sqrt{n}} x_i \bm u_j \big)  \nonumber \\
    & \prod_{ (i,j) \in E_{\circ}(S) \cap E_{\circ}(K) } \big( \tfrac{\mu y_i^2 \bm v_j^2}{n} + 1 \big)^2 \prod_{ (i,j) \in E_{\circ}(S) \triangle E_{\circ}(K) } \big( \tfrac{\sqrt{\mu}}{\sqrt{n}} y_i \bm v_j \big) \Bigg] \,. \label{eq-cov-f-S-f-K-chain-case-cov-simplify-1}
\end{align}
Using similar arguments as in \eqref{eq-goal-cov-f-S-f-K-cov-transfer-1}, we can write \eqref{eq-cov-f-S-f-K-chain-case-cov-simplify-1} into
\begin{align}
    \sum_{ \substack{ \mathtt E \subset E_{\bullet}(S) \triangle E_{\bullet}(K) \\ \mathtt F \subset E_{\circ}(S) \triangle E_{\circ}(K) } } \big( \tfrac{\lambda^2}{n} \big)^{|\mathtt E|} \big( \tfrac{\mu^2}{n} \big)^{|\mathtt F|} & \mathbb E_{(x,y)}\Bigg[ \prod_{ i \in V^{\mathsf a}(S) \cup V^{\mathsf a}(K) } x_i^{ \mathtt h(i;\mathtt E,\mathtt F;S,K) } y_i^{ \mathtt k(i;\mathtt E,\mathtt F;S,K) } \Bigg] \nonumber \\
    & \mathbb E_{(\bm u,\bm v)}\Bigg[ \prod_{ j \in V^{\mathsf b}(S) \cup V^{\mathsf b}(K) } \bm u_j^{ \mathtt h(j;\mathtt E,\mathtt F;S,K) } \bm v_j^{ \mathtt k(i;\mathtt E,\mathtt F;S,K) } \Bigg] \,, \label{eq-goal-cov-f-S-f-K-chain-case-cov-transfer-1}
\end{align}
where (below we use $i \sim e$ to denote that the vertex $i$ is incident to the edge $e$)
\begin{align*}
    & \mathtt h(i;\mathtt E,\mathtt F;S,K) = 2 \#\{ e \in \mathtt E: i \sim e \} + \#\{ e \in E_{\bullet}(S) \triangle E_{\bullet}(K): i \sim e \} + 2\cdot \mathbf 1_{i\in\{u,v \}}  \,; \\
    & \mathtt k(i;\mathtt E,\mathtt F;S,K) = 2 \#\{ e \in \mathtt F: i \sim e \} + \#\{ e \in E_{\circ}(S) \triangle E_{\circ}(K): i \sim e \} \,.
\end{align*}
We can bound \eqref{eq-goal-cov-f-S-f-K-chain-case-transfer-1} exactly by dividing to the different cases as in the proof of Lemma~\ref{lem-est-cov-f-S-f-K}. The only difference is that here for $i\in\{ u,v \}$, the degree of $i$ in $S \cup K$ equals $2$ (rather than $4$ in the proof of Lemma~\ref{lem-est-cov-f-S-f-K-cov}), so from Assumption~\ref{assum-upper-bound} we still have 
\begin{align*}
    \Big( \mathtt h(i;\mathtt E,\mathtt F;S,K)+\mathtt k(i;\mathtt E,\mathtt F;S,K) \Big) \leq 4 \Longrightarrow \mathbb E_{(x,y) \sim\mu}\Big[ x_i^{ \mathtt h(i;\mathtt E,\mathtt F;S,K) } y_i^{ \mathtt k(i;\mathtt E,\mathtt F;S,K) } \Big] \leq C \,.
\end{align*}

\subsection{Proof of Lemma~\ref{lem-recovery-most-technical-cov}}{\label{subsec:proof-lem-3.18}}

Before proving Lemma~\ref{lem-recovery-most-technical-cov}, we first show the following bound which will be used repeatedly in the later proof.

\begin{lemma}{\label{lem-contribution-sub-chain-cov}}
    There exists a constant $D'=D'(\lambda,\mu,\rho,\gamma)$ such that for all $\mathtt b \geq \mathtt a$ and given $v_{\mathtt a},v_{\mathtt b}$, we have
    \begin{align*}
        \sum_{ \mathsf{L}(S_{[\mathtt a,\mathtt b]})=v_{\mathtt a},v_{\mathtt b} } \Xi(S_{[\mathtt a,\mathtt b]})^2 \leq R' n^{\mathtt b-\mathtt a-1+\mathbf 1_{\mathtt b=\mathtt a}} (A_+/\gamma)^{ \frac{\mathtt b-\mathtt a}{2} } \,.
    \end{align*}
\end{lemma}

Once Lemma~\ref{lem-contribution-sub-chain-cov} has been established, we can complete the proof of Lemma~\ref{lem-recovery-most-technical-cov} in the same manner as we derived Lemma~\ref{lem-recovery-most-technical} from Lemma~\ref{lem-contribution-sub-chain}. The only difference is that we will replace all $A_+$ with $\sqrt{A_+/\gamma}$ so we omit further details here. Now we present the proof of Lemma~\ref{lem-contribution-sub-chain-cov}.

\begin{proof}[Proof of Lemma~\ref{lem-contribution-sub-chain-cov}]
    Denote $\mathtt L=\lfloor \frac{\mathtt b-\mathtt a}{2} \rfloor$. Note that we must have $S_{[\mathtt a,\mathtt b]} \in \mathcal I_{\star\star}(\mathtt L)$. Thus, we have
    \begin{align}
        \sum_{ S_{[\mathtt a,\mathtt b]}: \mathsf{L}(S_{[\mathtt a,\mathtt b]}) } \Xi(S_{[\mathtt a,\mathtt b]})^2 = \sum_{ [H] \in \mathcal I_{\star\star}(\mathtt L) } \Xi(H)^2 \cdot \#\Big\{ S \in \mathsf K_{n,N}: S \cong H, \mathsf{L}(S)=\{ v_{\mathtt a},v_{\mathtt b} \} \Big\} \,.  \label{eq-contribution-sub-chain-cov-relax}
    \end{align}
    Denote
    \begin{align*}
        m_{\mathsf{odd}} = \frac{ \mathtt b-\mathtt a - \mathbf 1_{ \{ \mathtt b \neq \mathtt a \} } + \mathbf 1_{ \{ \mathtt b,\mathtt a \text{ is even} \} } }{ 2 }, \ m_{\mathsf{even}}= \frac{ \mathtt b-\mathtt a - \mathbf 1_{ \{ \mathtt b \neq \mathtt a \} } + \mathbf 1_{ \{ \mathtt b,\mathtt a \text{ is odd} \} } }{ 2 } \,.
    \end{align*}
    Note that to choose $S \in \mathsf K_{n,N}$ with $S \cong H$ and $\mathsf{L}(S)=\{ v_{\mathtt a},v_{\mathtt b} \}$, we have at most 
    \begin{align*}
        \binom{n}{m_{\mathsf{odd}}} \binom{N}{m_{\mathsf{even}}} = \Theta(1) \cdot \frac{ \gamma^{ \frac{\mathtt b-\mathtt a}{2} } n^{ \mathtt b-\mathtt a-1+\mathbf 1_{\mathtt b=\mathtt a} } }{ m_{\mathsf{odd}}! m_{\mathsf{even}}! }  
    \end{align*}
    ways to choose $V(S) \setminus \mathsf{L}(S)$; in addition, given $V(S)$ we have at most $\frac{ m_{\mathsf{odd}}! m_{\mathsf{even}}! }{ |\mathsf{Aut}(J)| }$ ways to choose $S$. Thus, we have
    \begin{align*}
        \#\Big\{ S \in \mathsf K_{n,N}: S \cong H, \mathsf{L}(S)=\{ v_{\mathtt a},v_{\mathtt b} \} \Big\} &= \Theta(1) \cdot \frac{ \gamma^{ \frac{\mathtt b-\mathtt a}{2} } n^{ \mathtt b-\mathtt a-1+\mathbf 1_{\mathtt b=\mathtt a} } }{ m_{\mathsf{odd}}! m_{\mathsf{even}}! } \cdot \frac{ m_{\mathsf{odd}}! m_{\mathsf{even}}! }{ |\mathsf{Aut}(H)| } \\
        &= \frac{ \Theta(1) \cdot n^{\mathtt b-\mathtt a-1+\mathbf 1_{\mathtt b=\mathtt a}} }{ |\mathsf{Aut}(H)| } \,.
    \end{align*}
    Plugging this estimation into \eqref{eq-contribution-sub-chain-cov-relax}, we get that
    \begin{align*}
        \eqref{eq-contribution-sub-chain-cov-relax} = \Theta(1) \cdot n^{\mathtt b-\mathtt a-1+\mathbf 1_{\mathbb b=\mathtt a}} \cdot \sum_{ [H] \in \mathcal I_{\star}(\frac{\mathtt b-\mathtt a}{2}) } \frac{ \Xi(H)^2 }{ |\mathsf{Aut}(H)| } \leq D' n^{\mathtt b-\mathtt a-1+\mathbf 1_{\mathtt b=\mathtt a}} (A_+/\gamma)^{\frac{\mathtt b-\mathtt a}{2}} \,,
    \end{align*}
    where the last inequality follows from Lemma~\ref{lem-bound-beta-mathcal-I}.
\end{proof}

\section{Supplementary proofs in Section~\ref{sec:hypergraph-color-coding}}{\label{sec:supp-proofs-sec-4}}

\subsection{Proof of Proposition~\ref{prop-approximate-detection-statistics}}{\label{subsec:proof-prop-4.1}}

Recall \eqref{eq-def-f-mathcal-H} and \eqref{eq-def-widetilde-f-H}. We then have
\begin{align*}
    &\Big( \widetilde{f}_{\mathcal H} - f_{\mathcal H} \Big)^2 = \Bigg( \frac{1}{\sqrt{ \beta_{\mathcal H} n^{\ell} }} \sum_{ [H] \in \mathcal H } \Xi(H) \sum_{ S \cong H } f_{S} \cdot \Big( \frac{1}{t} \sum_{i=1}^{t} \frac{1}{r} \chi_{\tau_i}(V(S))-1 \Big) \Bigg)^2 \\
    =\ & \frac{ 1 }{ \beta_{\mathcal H} n^{\ell} } \sum_{ [H],[I] \in \mathcal H } \sum_{ \substack{ S \cong H \\ K \cong I } } \Xi(S) \Xi(K) f_{S} f_{K} \cdot \Big( \frac{1}{t} \sum_{\mathtt k=1}^{t} \frac{1}{r} \chi_{\tau_{\mathtt k}}(V(S))-1 \Big) \Big( \frac{1}{t} \sum_{\mathtt h=1}^{t} \frac{1}{r} \chi_{\tau_{\mathtt h}}(V(K))-1 \Big) \,.
\end{align*}
Note that by averaging over the randomness over $\{ \tau_{\mathtt k}:1 \leq \mathtt k \leq t \}$ we have
\begin{align*}
    \mathbb E\Big[ \Big( \frac{1}{t} \sum_{\mathtt k=1}^{t} \frac{1}{r} \chi_{\tau_{\mathtt k}}(V(S))-1 \Big) \Big( \frac{1}{t} \sum_{\mathtt h=1}^{t} \frac{1}{r} \chi_{\tau_{\mathtt h}}(V(K))-1 \Big) \Big]=0 \mbox{ if } V(S) \cap V(K)=\emptyset
\end{align*}
and 
\begin{align*}
    &\mathbb E\Big[ \Big( \frac{1}{t} \sum_{\mathtt k=1}^{t} \frac{1}{r} \chi_{\tau_{\mathtt k}}(V(S))-1 \Big) \Big( \frac{1}{t} \sum_{\mathtt h=1}^{t} \frac{1}{r} \chi_{\tau_{\mathtt h}}(V(K))-1 \Big) \Big] \\
    \leq\ & \mathbb E\Big[  \frac{1}{t^2} \sum_{\mathtt k=1}^{t} \frac{1}{r^2} \chi_{\tau_{\mathtt k}}(V(S)) \chi_{\tau_{\mathtt k}}(V(K)) \Big] \leq \frac{1}{tr} = 1 \,.
\end{align*}
Combining the fact that $\mathbb E_{\Pb}[f_{S}(\bm Y) f_{K}(\bm Y)], \mathbb E_{\Qb}[f_{S}(\bm Y) f_{K}(\bm Y)] \geq 0$, we get that
\begin{align}{\label{eq-approx-error-Qb}}
    \mathbb E_{\Qb}\Big[ ( \widetilde{f}_{\mathcal H} - f_{\mathcal H})^2 \Big] \leq \frac{ 1 }{ n^{\ell}\beta_{\mathcal H} } \sum_{ [H],[I] \in \mathcal H } \sum_{ \substack{ S \cong H, K \cong I \\ V(S) \cap V(K) \neq \emptyset } } \Xi(S) \Xi(K) \mathbb E_{\Qb}[f_{S}(\bm Y) f_{K}(\bm Y)] \,.
\end{align}
and 
\begin{align}{\label{eq-approx-error-Pb}}
    \mathbb E_{\Pb}\Big[ ( \widetilde{f}_{\mathcal H} - f_{\mathcal H})^2 \Big] \leq \frac{ 1 }{ n^{\ell}\beta_{\mathcal H} } \sum_{ [H],[I] \in \mathcal H } \sum_{ \substack{ S \cong H, K \cong I \\ V(S) \cap V(K) \neq \emptyset } } \Xi(S) \Xi(K) \mathbb E_{\Pb}[f_{S}(\bm Y) f_{K}(\bm Y)] \,.
\end{align}
We first bound \eqref{eq-approx-error-Qb}. By \eqref{eq-standard-orthogonal} we have
\begin{align*}
    \eqref{eq-approx-error-Qb} &= \frac{ 1 }{ n^{\ell}\beta_{\mathcal H} } \sum_{ [H],[I] \in \mathcal H } \sum_{ \substack{ S \cong H, K \cong I \\ V(S) \cap V(K) \neq \emptyset } } \Xi(S) \Xi(K) \mathbf 1_{ \{ S=K \} } \\
    &= \frac{ 1 }{ n^{\ell}\beta_{\mathcal H} } \sum_{ [H] \in \mathcal H } \Xi(H)^2 \#\big\{ S \subset \mathsf K_n: S \cong H \big\} = \frac{ 1+o(1) }{ n^{\ell}\beta_{\mathcal H} } \sum_{ [H] \in \mathcal H } \frac{\Xi(H)^2 n^{\ell}}{|\mathsf{Aut}(H)|} \\
    & \overset{\eqref{eq-def-beta-mathcal-H}}{=} 1+o(1) \overset{\text{Lemma~\ref{lem-mean-var-f-H-part-1}}}{=} o(1) \cdot \mathbb E_{\Pb}[f]^2 \,.
\end{align*}
Now we bound \eqref{eq-approx-error-Pb}. By the proof of Lemma~\ref{lem-est-cov-f-S-f-K}, we get that
\begin{align*}
    \mathbb E_{\Pb}\big[ f_{S}(\bm Y) f_{K}(\bm Y) \big] \leq n^{ -\ell+|E_{\bullet}(S)\cap E_{\bullet}(K)| +|E_{\circ}(S) \cap E_{\circ}(K)| }  \mathtt M(S,K)  \,.
\end{align*}
Thus, we have
\begin{align*}
    \eqref{eq-approx-error-Pb} &\leq \frac{ 1 }{ \beta_{\mathcal H} n^{\ell} } \sum_{ [H],[I] \in \mathcal H } \sum_{ \substack{ S \cong H, K \cong I \\ V(S) \cap V(K) \neq \emptyset } } \frac{ \Xi(S) \Xi(K) \mathtt M(S,K) }{ n^{ \ell-|E_{\bullet}(S)\cap E_{\bullet}(K)|-|E_{\circ}(S) \cap E_{\circ}(K)| } } = \eqref{eq-var-Pb-f-H-relax-2} \,, 
\end{align*}
which is already shown to be bounded by $o(1)$ in Section~\ref{subsec:proof-prop-2.3}. This completes the proof of Proposition~\ref{prop-approximate-detection-statistics}.

\subsection{Proof of Lemma~\ref{lem-color-coding}}{\label{subsec:proof-lem-4.2}}

In this subsection, we show the approximate graph counts $\mathfrak F_{H}(\bm X, \bm Y,\tau)$ can be computed efficiently via the following algorithm.
\begin{breakablealgorithm}{\label{alg:dynamic-programming}}
    \caption{Computation of $\mathfrak F_{H}(\bm X,\bm Y,\tau)$}
    \begin{algorithmic}[1]
    \STATE \textbf{Input}: Two symmetric matrices $\bm X,\bm Y$ on $[n]$ with its vertices colored by $\tau$, and an element $[H] \in \mathcal H$. 
    \STATE List the vertices in $E(H)$ in counterclockwise order $e_1,\ldots,e_{\ell}$ with $e_i=(v_i,v_{i+1})$ (we may arbitrarily choose $e_1$). Recall that we use $\gamma(H)$ to denote a mapping $\gamma:E(H) \to \{ \bullet,\circ \}$. Let $\gamma_i = \gamma(e_i)$.
    \STATE Denote $\widehat{H}=e_1 \cup \ldots \cup e_{\ell-1}$ such that $\mathsf{L}(\widehat{H})=\{ v_1,v_\ell \}$.
    \STATE For every color $c_1,c_2 \in [\ell]$, every color subsets $C \subset [\ell]$ such that $|C|=2$, every decoration $\gamma \in\{ \bullet,\circ \}$, and every distinct $u,v \in [n]$, compute $Y_{1}(u,v;c_1,c_2;C)= \mathbf 1_{C=\{ c_1,c_2 \}} \cdot \Lambda(u,v;c_1,c_2;\gamma_1)$, where 
    \begin{align*}
        \Lambda(u,v;c_1,c_2;\gamma) =
        \begin{cases}
            \bm X_{u,v} \cdot \mathbf 1_{ \{ \tau(x)=c_1, \tau(y)=c_2 \} },  & \gamma = \bullet \,; \\
            \bm Y_{u,v} \cdot \mathbf 1_{ \{ \tau(x)=c_1, \tau(y)=c_2 \} },  & \gamma = \circ \,;
        \end{cases}
    \end{align*}
    \STATE For every $2 \leq k \leq \ell-1$ and $c_1,c_2 \in [\ell]$, and for every $C \subset [\ell]$ with $|C|=k+1$ and $c_1,c_2 \in C$, compute recursively   
    \begin{align*}
       Y_{k}(x,y;c_1,c_2;C) = \sum_{c \in C} \sum_{ \substack{ z \in [n] \\ z \neq x,y } } \Lambda(z,y;c,c_2;\gamma_k) Y_{k-1}(x,z;c_1,c;C \setminus \{ c_2 \} ) \,.
    \end{align*}
    \STATE Compute
    \begin{align*}
        \mathfrak F_H(\bm X,\bm Y;\tau) =\ & \frac{1}{|\mathsf{Aut}(H)|} \sum_{ \substack{ c_1,c_2 \in [\ell] \\ c_1 \neq c_2 } } \sum_{ \substack{ x,y \in [n] \\ x \neq y } } \Lambda(x,y;c_1,c_2;\gamma_{\ell}) Y_{\ell-1}(y,x;c_2,c_1;[\ell]) \,.
    \end{align*}
    \textbf{Output}: $\mathfrak F_H(\bm X,\bm Y;\tau)$.
    \end{algorithmic}
\end{breakablealgorithm}
Now we can finish the proof of Lemma~\ref{lem-color-coding}. We first bound the total time complexity of Algorithm~\ref{alg:dynamic-programming}. Fixing a color set $C$ with $|C|=k+1$ and $c \in C$, the total number of $\{ c_1,c_2 \}$ with $c_1,c_2\in C$ is bounded by $k(k+1)$. Thus according to Algorithm~\ref{alg:dynamic-programming}, the total time complexity of computing $Y_{\ell-k}(x,y;c_1,c_2;C)$ is bounded by
\begin{align*}
    \sum_{1 \leq k \leq \ell-1} k(k+1) \cdot n^2 \overset{\eqref{eq-condition-strong-detection}}{=} n^{2+o(1)} \,.
\end{align*}
Finally, computing $\mathfrak F_H(\bm X,\bm Y,\tau)$ according to $Y_{2}(y,x;c_2,c_1;U_2 \cup \ldots \cup U_{\ell})$ takes time
\begin{align*}
    2^{2\ell} \cdot n^2 \overset{\eqref{eq-condition-strong-detection}}{=} n^{2+o(1)} \,.
\end{align*}
Thus, the total running time of Algorithm~\ref{alg:dynamic-programming} is $O(n^{2+o(1)})$.

Next, we prove the correctness of Algorithm~\ref{alg:dynamic-programming}. For any $1 \leq k \leq \ell-1$ and $|C|=k+1$, define
\begin{align*}
    & Z_{k}(x,y;c_1,c_2,C) \\
    =\ & \sum_{ \substack{ \psi: \{ v_{1}, \ldots v_{k+1} \} \to [n] \\ \psi(v_{1})=x, \psi(v_{k+1})=y }  } \mathbf 1_{ \{ \tau(v_{1} \cup \ldots \cup v_{k+1})=C, \tau(x)=c_1,\tau(y)=c_2 \} } \prod_{1 \leq i \leq k} \big( \mathbf 1_{\gamma_i=\bullet} \bm X_{v_i,v_{i+1}} + \mathbf 1_{\gamma_i=\circ} \bm Y_{v_i,v_{i+1}} \big) \,.
\end{align*}
It is clear that
\begin{align*}
    Z_{1}(x,y;c_1,c_2,C)=\mathbf 1_{C=\{c_1,c_2\}} \Lambda(x,y;c_1,c_2;\gamma_1)
\end{align*}
and 
\begin{align*}
    Z_{k}(x,y;c_1,c_2,C) =  \sum_{c \in C} \sum_{ \substack{ z \in [n] \\ z \neq x,y } } \Lambda(z,y;c,c_2;\gamma_k) Z_{k-1}(x,z;c_1,c,C \setminus \{ c_2 \}) \,.
\end{align*}
Thus, we have $Z_{k}(x,y;c_1,c_2;C)=Y_{k}(x,y;c_1,c_2;C)$ for all $1 \leq k \leq \ell-1$. This yields that
\begin{align*}
    & \sum_{ \substack{ c_1,c_2 \in [\ell] \\ c_1 \neq c_2 } } \sum_{ \substack{ x,y \in [n] \\ x \neq y } } \Lambda(x,y;c_1,c_2;\gamma_{\ell}) Y_{\ell-1}(y,x;c_2,c_1;[\ell]) \\
    =\ & \sum_{ \substack{ \psi: \{ v_{1}, \ldots v_{\ell} \} \to [n] } } \chi(\psi(v_1),\ldots,\psi(v_{\ell})) \prod_{1 \leq i \leq \ell} \big( \mathbf 1_{\gamma_i=\bullet} \bm X_{v_i,v_{i+1}} + \mathbf 1_{\gamma_i=\circ} \bm Y_{v_i,v_{i+1}} \big) = |\mathsf{Aut}(H)| \cdot \mathfrak F_H(\bm X,\bm Y,\tau) \,.
\end{align*}

\subsection{Proof of Proposition~\ref{prop-approximate-detection-statistics-cov}}{\label{subsec:proof-prop-4.4}}

Recall \eqref{eq-def-f-mathcal-G} and \eqref{eq-def-widetilde-f-mathcal-G}. We then have
\begin{align*}
    &\Big( \widetilde{h}_{\mathcal G} - h_{\mathcal G} \Big)^2 = \Bigg( \frac{1}{\sqrt{ \beta_{\mathcal H} N^{\ell} n^{\ell} }} \sum_{ [H] \in \mathcal G } \Upsilon(H) \sum_{ S \cong H } h_{S} \cdot \Big( \frac{1}{t} \sum_{i=1}^{t} \frac{1}{r} \chi_{\tau_i}(V(S))-1 \Big) \Bigg)^2 \\
    =\ & \frac{ 1 }{ \beta_{\mathcal H} n^{\ell} N^{\ell} } \sum_{ [H],[I] \in \mathcal G } \sum_{ \substack{ S \cong H \\ K \cong I } } \Upsilon(S) \Upsilon(K) h_{S} h_{K} \cdot \Big( \frac{1}{t} \sum_{\mathtt k=1}^{t} \frac{1}{r} \chi_{\tau_{\mathtt k}}(V(S))-1 \Big) \Big( \frac{1}{t} \sum_{\mathtt h=1}^{t} \frac{1}{r} \chi_{\tau_{\mathtt h}}(V(K))-1 \Big) \,.
\end{align*}
Note that by averaging over the randomness over $\{ \tau_{\mathtt k},\tau_{\mathtt k}':1 \leq \mathtt k \leq t \}$ we have
\begin{align*}
    \mathbb E\Big[ \Big( \frac{1}{t} \sum_{\mathtt k=1}^{t} \frac{1}{r} \chi_{\tau_{\mathtt k}}(V(S))-1 \Big) \Big( \frac{1}{t} \sum_{\mathtt h=1}^{t} \frac{1}{r} \chi_{\tau_{\mathtt h}}(V(K))-1 \Big) \Big]=0 \mbox{ if } V(S) \cap V(K)=\emptyset
\end{align*}
and 
\begin{align*}
    &\mathbb E\Big[ \Big( \frac{1}{t} \sum_{\mathtt k=1}^{t} \frac{1}{r} \chi_{\tau_{\mathtt k}}(V(S))-1 \Big) \Big( \frac{1}{t} \sum_{\mathtt h=1}^{t} \frac{1}{r} \chi_{\tau_{\mathtt h}}(V(K))-1 \Big) \Big] \\
    \leq\ & \mathbb E\Big[  \frac{1}{t^2} \sum_{\mathtt k=1}^{t} \frac{1}{r^2} \chi_{\tau_{\mathtt k}}(V(S)) \chi_{\tau_{\mathtt k}}(V(K)) \Big] \leq \frac{1}{tr} = 1 \,.
\end{align*}
Thus, similar as in \eqref{eq-approx-error-Qb} and \eqref{eq-approx-error-Pb}, we have 
\begin{align*}
    \mathbb E_{\overline\Qb}\Big[ ( \widetilde{h}_{\mathcal G} - h_{\mathcal G})^2 \Big] \leq 1+o(1) \mbox{ and } \mathbb E_{\overline\Pb}\Big[ ( \widetilde{h}_{\mathcal G} - h_{\mathcal G})^2 \Big] \leq \eqref{eq-var-Pb-f-mathcal-G-relax-2} = o(1) \cdot \mathbb E_{\overline\Pb}\big[ h_{\mathcal G} \big]^2 \,,
\end{align*}
Completing the proof of Proposition~\ref{prop-approximate-detection-statistics-cov}.

\subsection{Proof of Lemma~\ref{lem-color-coding-cov}}{\label{subsec:proof-lem-4.5}}

In this subsection, we show the approximate hypergraph counts $\mathfrak G_{H}(\bm X, \bm Y,\tau)$ can be computed efficiently via the following algorithm.
\begin{breakablealgorithm}{\label{alg:dynamic-programming-cov}}
    \caption{Computation of $\mathfrak G_{H}(\bm X,\bm Y,\tau)$}
    \begin{algorithmic}[1]
    \STATE \textbf{Input}: Two $n*N$ matrices $\bm X,\bm Y$ with its vertices colored by $\tau$, and an element $[H] \in \mathcal G$. 
    \STATE List the vertices in $E(H)$ in counterclockwise order $e_1,\ldots,e_{2\ell}$ with $e_{2i-1}=(v_i,u_i),e_{2i}=(v_{i+1},u_i)$ (we may arbitrarily choose $e_1$). Recall that we use $\gamma(H)$ to denote a mapping $\gamma:E(H) \to \{ \bullet,\circ \}$. Let $\gamma_i = \gamma(e_{2i-1})=\gamma(e_{2i})$.
    \STATE Denote $\widehat{H}=e_1 \cup \ldots \cup e_{2\ell-2}$ such that $\mathsf{L}(\widehat{H})=\{ v_1,v_\ell \}$.
    \STATE For every color $c_1,c_2 \in [2\ell]$, every decoration $\gamma \in\{ \bullet,\circ \}$, and every distinct $u \in [n], v \in [N]$, compute  
    \begin{align*}
        \Lambda(u,v;c_1,c_2;\gamma) =
        \begin{cases}
            \bm X_{u,v} \cdot \mathbf 1_{ \{ \tau(x)=c_1, \tau(y)=c_2 \} },  & \gamma = \bullet \,; \\
            \bm Y_{u,v} \cdot \mathbf 1_{ \{ \tau(x)=c_1, \tau(y)=c_2 \} },  & \gamma = \circ \,.
        \end{cases}
    \end{align*}
    \STATE For every color $c_1,c_2 \in [2\ell]$, every color subsets $C \subset [2\ell]$ such that $C=\{ c_1,c_2,c_3 \}$, and every distinct $x,y \in [n]$,  
    \begin{align*}
        Y_1(x,y;c_1,c_2;C) = \sum_{ u \in [N] } \Lambda(x,u;c_1,c_3;\gamma_1) \Lambda(y,u;c_2,c_3;\gamma_1) \,.
    \end{align*}
    \STATE For every $2 \leq k \leq \ell-1$ and $c_1,c_2 \in [\ell]$, and for every $C \subset [\ell]$ with $|C|=2k+1$ and $c_1,c_2 \in C$, compute recursively   
    \begin{align*}
       & Y_{k}(x,y;c_1,c_2;C) \\
       =\ & \sum_{ c,c' \in C \setminus \{ c_1,c_2 \} } \sum_{ w \in [N] } \sum_{ \substack{ z \in [n] \\ z \neq x,y } } \Lambda(y,w;c_2,c';\gamma_k) \Lambda(z,w;c,c';\gamma_k) Y_{k-1}(x,z;c_1,c;C \setminus \{ c',c_2 \} ) \,.
    \end{align*}
    \STATE Compute
    \begin{align*}
        & \mathfrak G_H(\bm X,\bm Y;\tau) \\
        =\ & \frac{1}{|\mathsf{Aut}(H)|} \sum_{ \substack{ c_1,c_2,c_3 \in [2\ell] } } \sum_{ \substack{ x,y \in [n] \\ x \neq y } } \sum_{ z \in [N] }  \Lambda(y,w;c_2,c;\gamma_\ell) \Lambda(z,w;c,c';\gamma_\ell) Y_{\ell-1}(y,x;c_2,c_1;[\ell]) \,.
    \end{align*}
    \textbf{Output}: $\mathfrak G_H(\bm X,\bm Y;\tau)$.
    \end{algorithmic}
\end{breakablealgorithm}

The running time analysis and the correctness of Algorithm~\ref{alg:dynamic-programming-cov} is almost identical to that of Algorithm~\ref{alg:dynamic-programming}, so we will omit further details here.

\subsection{Proof of Proposition~\ref{prop-approximate-recovery-statistics}}{\label{subsec:proof-prop-4.7}}

Recall \eqref{eq-averaging-over-coloring-recovery}, by first averaging over the random coloring $\{ \xi_{\mathtt k}: 1 \leq \mathtt k \leq t\}$ we have
\begin{align*}
    \mathbb E\Big[ \widetilde{\Phi}_{u,v}^{\mathcal J} \mid \bm X,\bm Y \Big] &\overset{\eqref{eq-def-widetilde-Phi-H}}{=} \frac{1}{n^{\frac{\ell}{2}-1} \beta_{\mathcal J}} \sum_{[J] \in \mathcal J} \Xi(J) \Big( \frac{1}{t} \sum_{ \mathtt k=1 }^{t} \sum_{ S \cong J,\mathsf{L}(S)=\{ u,v \} } f_S(\bm X,\bm Y) \Big) \\
    &= \frac{1}{n^{\frac{\ell}{2}-1} \beta_{\mathcal J}} \sum_{[J] \in \mathcal J} \Xi(J) \sum_{ S \cong J,\mathsf{L}(S)=\{ u,v \} } f_S(\bm X,\bm Y) = \Phi_{u,v}^{\mathcal J} \,.
\end{align*}
Thus, we have
\begin{align}
    \mathbb E\Big[ \widetilde{\Phi}_{u,v}^{\mathcal J} \cdot x_u x_v \Big] = \mathbb E\Big[ x_u x_v \cdot \mathbb E\big[ \widetilde{\Phi}_{u,v}^{\mathcal J} \mid \bm X, \bm Y \big] \Big] = \mathbb E\Big[ \Phi_{u,v}^{\mathcal J} \cdot x_u x_v \Big] = 1+o(1) \,, \label{eq-prop-approx-Phi-part-1}
\end{align}
where the last equality follows from Proposition~\ref{main-prop-recovery}. In addition, by \eqref{eq-def-widetilde-Phi-H} we then have
\begin{align*}
    & \Big( \widetilde{\Phi}^{\mathcal J}_{u,v} \Big)^2 = \Bigg( \frac{ 1 }{ \beta_{\mathcal J} n^{\frac{\ell}{2}-1} } \sum_{ \substack{ S \subset \mathsf K_n: [S] \in \mathcal J \\ \mathsf{L}(S)=\{ u,v \} } } \Xi(S) f_{S}(\bm X,\bm Y) \cdot \Big( \frac{1}{t} \sum_{\mathtt k=1}^{t} \frac{1}{\varkappa} \chi_{\xi_{\mathtt k}}(V(S)) \Big)  \Bigg)^2 \\
    =\ & \sum_{ \substack{ S,K \subset \mathsf K_n: [S],[K] \in \mathcal J \\ \mathsf{L}(S)=\mathsf{L}(K)=\{ u,v \} } } \frac{ \Xi(S) \Xi(K) f_{S}(\bm X,\bm Y) f_{K}(\bm X,\bm Y) }{ \beta_{\mathcal J}^2 n^{2\ell-2} }  \Big( \frac{1}{t} \sum_{\mathtt k=1}^{t} \frac{1}{\varkappa} \chi_{\xi_{\mathtt k}}(V(S)) \Big) \Big( \frac{1}{t} \sum_{\mathtt k=1}^{t} \frac{1}{\varkappa} \chi_{\xi_{\mathtt k}}(V(K)) \Big) \,.
\end{align*}
Note that 
\begin{align*}
    &\mathbb E\Big[ \Big( \frac{1}{t} \sum_{\mathtt k=1}^{t} \frac{1}{r} \chi_{\xi_{\mathtt k}}(V(S))-1 \Big) \Big( \frac{1}{t} \sum_{\mathtt h=1}^{t} \frac{1}{r} \chi_{\xi_{\mathtt h}}(V(K))-1 \Big) \Big] \\
    \leq\ & \mathbb E\Big[  \frac{1}{t^2} \sum_{\mathtt k=1}^{t} \frac{1}{r^2} \chi_{\xi_{\mathtt k}}(V(S)) \chi_{\xi_{\mathtt k}}(V(K)) \Big] \leq \frac{1}{tr} = 1 \,.
\end{align*}
We then have
\begin{align}
    \mathbb E\Big[ \big( \widetilde{\Phi}^{\mathcal J}_{u,v} \big)^2 \Big] \leq \sum_{ \substack{ S,K \subset \mathsf K_n: [S],[K] \in \mathcal J \\ \mathsf{L}(S)=\mathsf{L}(K)=\{ u,v \} } } \frac{ \Xi(S) \Xi(K)  }{ \beta_{\mathcal J}^2 n^{2\ell-2} } \cdot \Big| \mathbb E_{\Pb}\big[ f_{S}(\bm X,\bm Y) f_{K}(\bm X,\bm Y) \big] \Big| \leq R \,,  \label{eq-prop-approx-Phi-part-2}
\end{align}
where the second inequality holds since the middle term is exactly \eqref{eq-var-Phi-i,j-relax-1}, which is shown to be bounded by some $R=O_{\lambda,\mu,\rho,\delta}(1)$ in Section~\ref{subsec:proof-prop-2.7}. Similarly we can show that
\begin{align}
    \mathbb E\Big[ \big( \widetilde{\Phi}^{\mathcal J}_{u,v} \big)^2 \cdot x_u^2 x_v^2 \Big] \leq R \,.  \label{eq-prop-approx-Phi-part-3}
\end{align}
Combined with \eqref{eq-prop-approx-Phi-part-2} and \eqref{eq-prop-approx-Phi-part-1}, this completes the proof of Proposition~\ref{prop-approximate-recovery-statistics}.

\subsection{Proof of Lemma~\ref{lem-color-coding-recovery}}{\label{subsec:proof-lem-4.8}}

In this subsection, we show the approximate hypergraph counts $\mathfrak L_{H}(\bm X,\bm Y,\xi)$ can be computed efficiently via the following algorithm.
\begin{breakablealgorithm}{\label{alg:dynamic-programming-recovery}}
    \caption{Computation of $\mathfrak L_{H}(\bm X,\bm Y,\xi)$}
    \begin{algorithmic}[1]
    \STATE \textbf{Input}: Symmetric matrices $\bm X,\bm Y$ on $[n]$ with its vertices colored by $\xi$, and an element $[H] \in \mathcal J$.  
    \STATE List the vertices in $E(H)$ in counterclockwise order $e_1,\ldots,e_{\ell}$ with $e_i=(v_i,v_{i+1})$ such that $v_1=i$ and $v_{\ell+1}=j$. Recall that we use $\gamma(H)$ to denote a mapping $\gamma:E(H) \to \{ \bullet,\circ \}$. Let $\gamma_i = \gamma(e_i)$.
    \STATE For every color $c_1,c_2 \in [\ell]$, every color subsets $C \subset [\ell]$ such that $|C|=2$, every decoration $\gamma \in\{ \bullet,\circ \}$, and every distinct $u,v \in [n]$, compute $Y_{1}(u,v;c_1,c_2;C)= \mathbf 1_{C=\{ c_1,c_2 \}} \cdot \Lambda(u,v;c_1,c_2;\gamma_1)$, where 
    \begin{align*}
        \Lambda(u,v;c_1,c_2;\gamma) =
        \begin{cases}
            \bm X_{u,v} \cdot \mathbf 1_{ \{ \xi(x)=c_1, \xi(y)=c_2 \} },  & \gamma = \bullet \,; \\
            \bm Y_{u,v} \cdot \mathbf 1_{ \{ \xi(x)=c_1, \xi(y)=c_2 \} },  & \gamma = \circ \,;
        \end{cases}
    \end{align*}
    \STATE For every $2 \leq k \leq \ell-1$ and $c_1,c_2 \in [\ell]$, and for every $C \subset [\ell]$ with $|C|=k+1$ and $c_1,c_2 \in C$, compute recursively   
    \begin{align*}
       Y_{k}(x,y;c_1,c_2;C) = \sum_{c \in C} \sum_{ \substack{ z \in [n] \\ z \neq x,y } } \Lambda(z,y;c,c_2;\gamma_k) Y_{k-1}(x,z;c_1,c;C \setminus \{ c_2 \} ) \,.
    \end{align*}
    \STATE Compute
    \begin{align*}
        \mathfrak L_H(\bm X,\bm Y,\xi) = \frac{1}{|\mathsf{Aut}(H)|} \sum_{ c_1\neq c_2 } Y_{1}(u,v;c_1,c_2,C)
    \end{align*}
    \textbf{Output}: $\mathfrak L_H(\bm X,\bm Y,\xi)$.
    \end{algorithmic}
\end{breakablealgorithm}
Now we can finish the proof of Lemma~\ref{lem-color-coding-recovery}. We first bound the total time complexity of Algorithm~\ref{alg:dynamic-programming-recovery}. Fixing a color set $C$ with $|C|=k+1$ and $c \in C$, the total number of $C_1,C_2$ with $C_1 \cup C_2=C$ and $C_1 \cap C_2 = \{ c \}$ is bounded by $2^{k+1}$. Thus according to Algorithm~\ref{alg:dynamic-programming-recovery}, the total time complexity of computing $Y_{\ell-k}(x,y;c_1,c_2;C)$ is bounded by
\begin{align*}
    \sum_{1 \leq k \leq \ell-1} 2^{k+1} \cdot n^2 \overset{\eqref{eq-condition-weak-recovery}}{=} n^{T+o(1)} \,.
\end{align*}
Thus, the total running time of Algorithm~\ref{alg:dynamic-programming} is $O(n^{T+o(1)})$. The correctness of Algorithm~\ref{alg:dynamic-programming-recovery} is almost identical to the correctness of Algorithm~\ref{alg:dynamic-programming}, so we omit further details here for simplicity.

\subsection{Proof of Proposition~\ref{prop-approximate-recovery-statistics-cov}}{\label{subsec:proof-prop-4.10}}

Recall \eqref{eq-averaging-over-coloring-recovery-cov}, by first averaging over the random coloring $\{ \xi_{\mathtt k}: 1 \leq \mathtt k \leq t\}$ we have
\begin{align*}
    \mathbb E\Big[ \widetilde{\Phi}_{u,v}^{\mathcal I} \mid \bm X,\bm Y \Big] &\overset{\eqref{eq-def-widetilde-Phi-mathcal-J}}{=} \frac{1}{ N^{\ell} n^{-1} \beta_{\mathcal I}} \sum_{[H] \in \mathcal I} \Upsilon(H) \Big( \frac{1}{t} \sum_{ \mathtt k=1 }^{t} \sum_{ S \cong H,\mathsf{L}(S)=\{ u,v \} } h_S(\bm X,\bm Y) \Big) \\
    &= \frac{1}{N^{\ell}n^{-1} \beta_{\mathcal I}} \sum_{[H] \in \mathcal I} \Upsilon(H) \sum_{ S \cong H,\mathsf{L}(S)=\{ u,v \} } h_S(\bm X,\bm Y) = \Phi_{u,v}^{\mathcal I} \,.
\end{align*}
Thus, we have
\begin{align}
    \mathbb E\Big[ \widetilde{\Phi}_{u,v}^{\mathcal I} \cdot x_u x_v \Big] = \mathbb E\Big[ x_u x_v \cdot \mathbb E\big[ \widetilde{\Phi}_{u,v}^{\mathcal I} \mid \bm X, \bm Y \big] \Big] = \mathbb E\Big[ \Phi_{u,v}^{\mathcal I} \cdot x_u x_v \Big] = 1+o(1) \,, \label{eq-prop-approx-Phi-mathcal-I-part-1}
\end{align}
where the last equality follows from Proposition~\ref{main-prop-recovery-cov}. In addition, by \eqref{eq-def-widetilde-Phi-mathcal-J} we then have
\begin{align*}
    & \Big( \widetilde{\Phi}^{\mathcal I}_{u,v} \Big)^2 = \Bigg( \frac{ 1 }{ N^{\ell}n^{-1} \beta_{\mathcal I} } \sum_{ \substack{ S \subset \mathsf K_n: [S] \in \mathcal I \\ \mathsf{L}(S)=\{ u,v \} } } \Upsilon(S) h_{S}(\bm X,\bm Y) \cdot \Big( \frac{1}{t} \sum_{\mathtt k=1}^{t} \frac{1}{\varkappa} \chi_{\xi_{\mathtt k}}(V(S)) \Big)  \Bigg)^2 \\
    =\ & \sum_{ \substack{ S,K \subset \mathsf K_n: [S],[K] \in \mathcal I \\ \mathsf{L}(S)=\mathsf{L}(K)=\{ u,v \} } } \frac{ \Upsilon(S) \Upsilon(K) h_{S}(\bm X,\bm Y) h_{K}(\bm X,\bm Y) }{ \beta_{\mathcal I}^2 N^{2\ell}n^{-2} }  \Big( \frac{1}{t} \sum_{\mathtt k=1}^{t} \frac{1}{\varkappa} \chi_{\xi_{\mathtt k}}(V(S)) \Big) \Big( \frac{1}{t} \sum_{\mathtt k=1}^{t} \frac{1}{\varkappa} \chi_{\xi_{\mathtt k}}(V(K)) \Big) \,.
\end{align*}
Similarly as in \eqref{eq-prop-approx-Phi-part-2} and \eqref{eq-prop-approx-Phi-part-3}, we can show that
\begin{align*}
    \mathbb E\Big[ \big( \widetilde{\Phi}^{\mathcal I}_{u,v} \big)^2 \cdot x_u^2 x_v^2 \Big],\ \mathbb E\Big[ \big( \widetilde{\Phi}^{\mathcal I}_{u,v} \big)^2 \Big] \leq R \,,  
\end{align*}
completing the proof of Proposition~\ref{prop-approximate-recovery-statistics-cov}.

\subsection{Proof of Lemma~\ref{lem-color-coding-recovery-cov}}{\label{subsec:proof-lem-4.11}}

In this subsection, we show the approximate hypergraph counts $\mathfrak R_{H}(\bm X,\bm Y,\xi)$ can be computed efficiently via the following algorithm.
\begin{breakablealgorithm}{\label{alg:dynamic-programming-recovery-cov}}
    \caption{Computation of $\mathfrak R_{H}(\bm X,\bm Y,\xi)$}
    \begin{algorithmic}[1]
    \STATE \textbf{Input}: Two $n*N$ matrices $\bm X,\bm Y$ with its vertices colored by $\tau$, and an element $[H] \in \mathcal I$. 
    \STATE List the vertices in $E(H)$ in counterclockwise order $e_1,\ldots,e_{2\ell}$ with $e_{2i-1}=(v_i,u_i),e_{2i}=(v_{i+1},u_i)$ (we may arbitrarily choose $e_1$). Recall that we use $\gamma(H)$ to denote a mapping $\gamma:E(H) \to \{ \bullet,\circ \}$. Let $\gamma_i = \gamma(e_{2i-1})=\gamma(e_{2i})$.
    \STATE For every color $c_1,c_2 \in [2\ell]$, every decoration $\gamma \in\{ \bullet,\circ \}$, and every distinct $u \in [n], v \in [N]$, compute  
    \begin{align*}
        \Lambda(u,v;c_1,c_2;\gamma) =
        \begin{cases}
            \bm X_{u,v} \cdot \mathbf 1_{ \{ \xi(x)=c_1, \xi(y)=c_2 \} },  & \gamma = \bullet \,; \\
            \bm Y_{u,v} \cdot \mathbf 1_{ \{ \xi(x)=c_1, \xi(y)=c_2 \} },  & \gamma = \circ \,.
        \end{cases}
    \end{align*}
    \STATE For every color $c_1,c_2 \in [2\ell]$, every color subsets $C \subset [2\ell]$ such that $C=\{ c_1,c_2,c_3 \}$, and every distinct $u,v \in [n]$,  
    \begin{align*}
        Y_1(u,v;c_1,c_2;C) = \sum_{ z \in [N] } \Lambda(x,z;c_1,c_3;\gamma_1) \Lambda(y,z;c_2,c_3;\gamma_1) \,.
    \end{align*}
    \STATE For every $2 \leq k \leq \ell-1$ and $c_1,c_2 \in [\ell]$, and for every $C \subset [\ell]$ with $|C|=2k+1$ and $c_1,c_2 \in C$, compute recursively   
    \begin{align*}
       & Y_{k}(x,y;c_1,c_2;C) \\
       =\ & \sum_{ c,c' \in C \setminus \{ c_1,c_2 \} } \sum_{ w \in [N] } \sum_{ \substack{ z \in [n] \\ z \neq x,y } } \Lambda(y,w;c_2,c';\gamma_k) \Lambda(z,w;c,c';\gamma_k) Y_{k-1}(x,z;c_1,c;C \setminus \{ c',c_2 \} ) \,.
    \end{align*}
    \STATE Compute
    \begin{align*}
        \mathfrak R_H(\bm X,\bm Y,\xi) = \frac{1}{|\mathsf{Aut}(H)|} \sum_{ c_1\neq c_2 } Y_{1}(u,v;c_1,c_2,[2\ell+1])
    \end{align*}
    \textbf{Output}: $\mathfrak R_H(\bm X,\bm Y,\xi)$.
    \end{algorithmic}
\end{breakablealgorithm}

The running time analysis and the correctness of Algorithm~\ref{alg:dynamic-programming-recovery-cov} is almost identical to that of Algorithm~\ref{alg:dynamic-programming-recovery}, so we will omit further details here.

\section{Supplementary proofs in Section~\ref{sec:lower-bound}}{\label{sec:supp-proofs-sec-5}}

\subsection{Proof of Lemma~\ref{lem-correlated-Gaussian-moment}}{\label{subsec:proof-lem-5.8}}

Denote $\lambda_*=(1+\delta)\lambda$ and $\mu_*=(1+\delta)\mu$. Note that since \eqref{eq-condition-lower-bound} holds, we can choose a sufficiently small $\delta>0$ such that
    \begin{equation}{\label{eq-condition-lower-bound-relax}}
        \frac{ \lambda_*^2 \rho^2 }{ 2(1-\lambda_*^2+\lambda_*^2 \rho^2) } + \frac{ \mu_*^2 \rho^2 }{ 2(1-\mu_*^2+\mu_*^2 \rho^2) } < \frac{1-0.5\epsilon}{2} \,.
    \end{equation}
    Thus, we have 
    \begin{align*}
        & \sum_{ k,\ell=0 }^{D} \frac{ \lambda_*^{2k} \mu_*^{2\ell} }{ 2^{k+\ell}k!\ell! } \mathbb E\Bigg\{ \Big( \frac{ \sum_{j=1}^{n} \zeta_j }{ \sqrt{n} } \Big)^{2k} \Big( \frac{ \sum_{j=1}^{n} \eta_j }{ \sqrt{n} } \Big)^{2\ell} \Bigg\}  \\
        \leq\ & \mathbb E\Bigg\{ \exp \Bigg( \frac{\lambda_*^2}{2} \Big( \frac{ \sum_{j=1}^{n} \zeta_j }{ \sqrt{n} } \Big)^2 + \frac{\mu_*^2}{2} \Big( \frac{ \sum_{j=1}^{n} \eta_j }{ \sqrt{n} } \Big)^2 \Bigg\} = \mathbb E\Big\{ e^{ \frac{1}{2} (\lambda_*^2 U^2 + \mu_*^2 V^2) }  \Big\} \,. 
    \end{align*}
    Since we can write $U=\rho Z+ \sqrt{1-\rho^2} X, V = \rho Z+\sqrt{1-\rho^2} Y$ where $X,Y,Z$ are independent standard normal variables, thus
    \begin{align*}
        &\mathbb E\Big\{ e^{ \frac{1}{2} (\lambda_*^2 U^2 + \mu_*^2 V^2) } \Big\} = \mathbb E\Big\{ \exp\big( \tfrac{1}{2} \lambda_*^2(\rho Z+ \sqrt{1-\rho^2} X)^2 + \tfrac{1}{2} \mu_*^2 (\rho Z+\sqrt{1-\rho^2} Y)^2 \big) \Big\} \\
        =\ & \mathbb E\Big\{ \exp\big( \tfrac{(\lambda_*^2+\mu_*^2)\rho^2}{2} Z^2 + \lambda_*^2 \rho\sqrt{1-\rho^2} ZX + \mu_*^2 \rho\sqrt{1-\rho^2} ZY + \tfrac{\lambda_*^2(1-\rho^2)}{2} X^2 + \tfrac{\mu_*^2(1-\rho^2)}{2} Y^2 \big) \Big\} \\
        =\ & \mathbb E\Big\{ \exp\big( \tfrac{(\lambda_*^2+\mu_*^2)\rho^2}{2} Z^2 + \tfrac{ \lambda_*^4 \rho^2(1-\rho^2) }{ 2(1-\lambda_*^2+\lambda_*^2 \rho^2) } Z^2 + \tfrac{ \mu_*^4 \rho^2(1-\rho^2) }{ 2(1-\mu_*^2+\mu_*^2 \rho^2) } Z^2 \big) \Big\} \\
        =\ & \mathbb E\Big\{ \exp\big( ( \tfrac{ \lambda_*^2 \rho^2 }{ 2(1-\lambda_*^2+\lambda_*^2 \rho^2) }  + \tfrac{ \mu_*^2 \rho^2 }{ 2(1-\mu_*^2+\mu_*^2 \rho^2) } ) Z^2 \big) \Big\} \overset{\eqref{eq-condition-lower-bound-relax}}{=} O_{\epsilon}(1) \,.
    \end{align*}
    Thus, we have
    \begin{align*}
        \frac{ \lambda^{2k} \mu^{2\ell} }{ 2^{k+\ell}k!\ell! } \mathbb E\Bigg\{ \Big( \frac{ \sum_{j=1}^{n} \zeta_j }{ \sqrt{n} } \Big)^{2k} \Big( \frac{ \sum_{j=1}^{n} \eta_j }{ \sqrt{n} } \Big)^{2\ell} \Bigg\} &\leq (1+\delta)^{-2(k+\ell)} \mathbb E\Big\{ e^{ \frac{1}{2} (\lambda_*^2 U^2 + \mu_*^2 V^2) } \Big\} \\
        &= O_{\epsilon}(1) \cdot (1+\delta)^{-2(k+\ell)} \,.  
    \end{align*}

\subsection{Proof of Lemma~\ref{lem-compare-moments-Gaussian-non-Gaussian}}{\label{subsec:proof-lem-5.9}}

We will use Lindeberg's interpolation to bound the difference between
\begin{align*}
    \frac{ \lambda^{2k} \mu^{2\ell} }{ 2^{k+\ell}k!\ell! } \mathbb E\Bigg\{ \Big( \frac{ \sum_{j=1}^{n} \zeta_j }{ \sqrt{n} } \Big)^{2k} \Big( \frac{ \sum_{j=1}^{n} \eta_j }{ \sqrt{n} } \Big)^{2\ell} \Bigg\} \mbox{ and } \frac{ \lambda^{2k} \mu^{2\ell} }{ 2^{k+\ell}k!\ell! } \mathbb E\Bigg\{ \Big( \frac{ \sum_{j=1}^{n} X_j }{ \sqrt{n} } \Big)^{2k} \Big( \frac{ \sum_{j=1}^{n} Y_j }{ \sqrt{n} } \Big)^{2\ell} \Bigg\} \,.
\end{align*}
To this end, define 
\begin{align}
    U=U_0=\frac{ \sum_{j=1}^{n} \zeta_j }{ \sqrt{n} } \mbox{ and } V=V_1= \frac{ \sum_{j=1}^{n} \eta_j }{ \sqrt{n} } \,.  \label{eq-def-U-0V-0}
\end{align}
In addition, for each $1 \leq t \leq n$ define
\begin{align}
    U_t = \frac{1}{\sqrt{n}}\Big( \sum_{j=1}^{t} X_j + \sum_{j=t+1}^{n} \zeta_j \Big) \mbox{ and } V_t = \frac{1}{\sqrt{n}}\Big( \sum_{j=1}^{t} Y_j + \sum_{j=t+1}^{n} \eta_j \Big) \,.  \label{eq-def-U-t-V-t}
\end{align}
Finally, define
\begin{align}
    U_{t,*} = \frac{1}{\sqrt{n}}\Big( \sum_{j=1}^{t-1} X_j + \sum_{j=t+1}^{n} \zeta_j \Big) \mbox{ and } V_{t,*} = \frac{1}{\sqrt{n}}\Big( \sum_{j=1}^{t-1} Y_j + \sum_{j=t+1}^{n} \eta_j \Big) \,.  \label{eq-def-U-t-*-V-t-*}
\end{align}
Then, we have (note that $\{ X_j,Y_j \}$ is independent with $\{ X_i,Y_i: i \neq j \}$)
\begin{align}
    \mathbb E\Big[ U_t^{2k} V_t^{2\ell} \Big] &= \mathbb E\Big[ (U_{t,*}+\tfrac{X_t}{\sqrt{n}})^{2k} (V_{t,*}+\tfrac{Y_t}{\sqrt{n}})^{2\ell} \Big] = \sum_{a=0}^{2k} \sum_{b=0}^{2\ell} \binom{2k}{a} \binom{2\ell}{b} \frac{1}{n^{(a+b)/2}} \mathbb E\Big[ U_{t,*}^{2k-a} V_{t,*}^{2\ell-b} X_t^{a} Y_t^{b} \Big] \nonumber \\
    &= \sum_{a=0}^{2k} \sum_{b=0}^{2\ell} \binom{2k}{a} \binom{2\ell}{b} \frac{1}{n^{(a+b)/2}} \mathbb E\big[ X_t^{a} Y_t^{b} \big] \mathbb E\Big[ U_{t,*}^{2k-a} V_{t,*}^{2\ell-b} \Big] \,.  \label{eq-moment-U-V}
\end{align}
Now we make several claims regarding the joint moments of $X_t,Y_t$ and $U_{t,*},V_{t,*}$.

\begin{claim}{\label{claim-moment-u-t-v-t}}
    There exists an absolute constant $C''>0$ such that
    \begin{align*}
        \mathbb E\big[ X_t^{a} Y_t^{b} \big] \leq (C'')^{a+b} (2a+2b)! \,.
    \end{align*}
\end{claim}
\begin{proof}
    Note that by our assumption we have $X_t$ are products of two $C'$-subgaussian variables (and so does $Y_t$), thus using \cite[Proposition~3.14]{DK25+} we have $\mathbb E[|X_t|^{a}],\mathbb E[|Y_t|^a] \leq (4C')^{2a} \cdot (2a)!$ for all $a \geq 0$. Thus, using Holder's inequality we have
    \begin{equation*}
        \mathbb E\big[ X_t^{a} Y_t^{b} \big] \leq \mathbb E\big[ |X_t|^{a+b} \big]^{ \frac{a}{a+b} } \cdot \mathbb E\big[ |Y_t|^{a+b} \big]^{\frac{b}{a+b}} \leq (4C')^{a+b} (2a+2b)! \leq (C'')^{a+b} (2a+2b)! \,. \qedhere
    \end{equation*}
\end{proof}
\begin{claim}{\label{claim-moment-U-t-*-V-t-*}}
    For $k,\ell=n^{o(1)}$ we have
    \begin{align*}
        \mathbb E\Big[ U_{t,*}^{2k-a} V_{t,*}^{2\ell-b} \Big] \leq (1-o(1)) \mathbb E\Big[ U_{t,*}^{2k} V_{t,*}^{2\ell} \Big] \,.
    \end{align*}
\end{claim}
\begin{proof}
    It suffices to show the case that $a,b$ are even, i.e.,
    \begin{align*}
        \mathbb E\Big[ U_{t,*}^{2k-2a} V_{t,*}^{2\ell-2b} \Big] \leq (1-o(1)) \mathbb E\Big[ U_{t,*}^{2k} V_{t,*}^{2\ell} \Big] \,,
    \end{align*}
    since we can deal with the odd case using
    \begin{align*}
        U_{t,*}^{2k-2a+1} \leq \frac{1}{2}\big( U_{t,*}^{2k-2a+2} + U_{t,*}^{2k-2a} \big) \mbox{ and } V_{t,*}^{2k-2a+1} \leq \frac{1}{2} \big( V_{t,*}^{2k-2a+2} + V_{t,*}^{2k-2a} \big) \,.
    \end{align*}
    Note that $\mathbb E[U_{t,*}^2]=\mathbb E[V_{t,*}^2]=\frac{n-1}{n}$. In addition, from Assumption~\ref{assum-lower-bound} we have
    \begin{align*}
        \mathbb E[X_i^a Y_i^b] \geq \mathbb E[X_i^a] \mathbb E[Y_i^b] \geq 0 \mbox{ for all } 1 \leq i \leq n \mbox{ and } a,b \in \mathbb N \,.
    \end{align*}
    Also, it is clear that
    \begin{align*}
        \mathbb E[\zeta_i^a \eta_i^b] \geq \mathbb E[\zeta_i^a] \mathbb E[\eta_i^b] \geq 0 \mbox{ for all } 1 \leq i \leq n \mbox{ and } a,b \in \mathbb N \,.
    \end{align*}
    Thus, we see that $U_{t,*}^a,V_{t,*}^b$ are positively correlated for all $a,b\in\mathbb N$. This yields that
    \begin{align*}
        \mathbb E\big[ U_{t,*}^{2k-2a+2} V_{t,*}^{2\ell-2b} \big] \geq \mathbb E\big[ U_{t,*}^{2k-2a} V_{t,*}^{2\ell-2b} \big] \mathbb E\big[ U_{t,*}^{2} \big] = (1-\tfrac{1}{n}) \mathbb E\big[ U_{t,*}^{2k-2a} V_{t,*}^{2\ell-2b} \big] \,.
    \end{align*}
    A simple induction argument then yields the desired bound.
\end{proof}
To this end, based on Claims~\ref{claim-moment-u-t-v-t} and \ref{claim-moment-U-t-*-V-t-*}, we then have
\begin{align*}
    & \sum_{ \substack{ (a,b) \neq (0,0) \\ a \leq 2k, b \leq 2\ell } } \binom{2k}{a} \binom{2\ell}{b} \frac{1}{n^{(a+b)/2}} \mathbb E\big[ X_t^{a} Y_t^{b} \big] \mathbb E\Big[ U_{t,*}^{2k-a} V_{t,*}^{2\ell-b} \Big]  \\
    \leq\ & [1-o(1)] \cdot \sum_{ \substack{ (a,b) \neq (0,0) \\ a \leq 2k, b \leq 2\ell } } \binom{2k}{a} \binom{2\ell}{b} \frac{ (2a+2b)! (C'')^{a+b} }{n^{(a+b)/2}} \mathbb E\Big[ U_{t,*}^{2k} V_{t,*}^{2\ell} \Big] = n^{-\frac{1}{2}+o(1)} \cdot \mathbb E\Big[ U_{t,*}^{2k} V_{t,*}^{2\ell} \Big] 
\end{align*}
when $k,\ell=n^{o(1)}$. Thus, we have
\begin{align}
    \mathbb E\Big[ U_t^{2k} V_t^{2\ell} \Big] = (1+o(1)) \mathbb E\Big[ U_{t,*}^{2k} V_{t,*}^{2\ell} \Big] \,.  \label{eq-compare-U-t-*-V-t-*-and-U-t-V-t}
\end{align}
Similarly, we have
\begin{align*}
    \mathbb E\Big[ U_{t+1}^{2k} V_{t+1}^{2\ell} \Big] &= \mathbb E\Big[ (U_{t,*}+\tfrac{\zeta_t}{\sqrt{n}})^{2k} (V_{t,*}+\tfrac{\eta_t}{\sqrt{n}})^{2\ell} \Big] = \sum_{a=0}^{2k} \sum_{b=0}^{2\ell} \binom{2k}{a} \binom{2\ell}{b} \frac{1}{n^{(a+b)/2}} \mathbb E\Big[ U_{t,*}^{2k-a} V_{t,*}^{2\ell-b} \zeta_t^{a} \eta_t^{b} \Big] \\
    &= \sum_{a=0}^{2k} \sum_{b=0}^{2\ell} \binom{2k}{a} \binom{2\ell}{b} \frac{1}{n^{(a+b)/2}} \mathbb E\big[ \zeta_t^{a} \eta_t^{b} \big] \mathbb E\Big[ U_{t,*}^{2k-a} V_{t,*}^{2\ell-b} \Big] \,.
\end{align*}
Thus, we have (recall that $\mathbb E[X_t^a Y_t^b]=\mathbb E[\zeta_t^a \eta_t^b]$ for $a+b \leq 2$)
\begin{align*}
    & \mathbb E\big[ U_{t+1}^{2k} V_{t+1}^{2\ell} \big] - \mathbb E\big[ U_{t}^{2k} V_{t}^{2\ell} \big] \\
    =\ & \sum_{a=0}^{2k} \sum_{b=0}^{2\ell} \binom{2k}{a} \binom{2\ell}{b} \frac{1}{n^{(a+b)/2}} \Big( \mathbb E\big[ \zeta_t^{a}\eta_t^{b} \big] - \mathbb E\big[ X_t^{a} Y_t^{b} \big] \Big) \mathbb E\Big[ U_{t,*}^{2k-a} V_{t,*}^{2\ell-b} \Big] \\
    =\ & \sum_{ \substack{ a \leq 2k, b \leq 2\ell \\ a+b \geq 3 } } \binom{2k}{a} \binom{2\ell}{b} \frac{1}{n^{(a+b)/2}} \Big( \mathbb E\big[ \zeta_t^{a}\eta_t^{b} \big] - \mathbb E\big[ X_t^{a} Y_t^{b} \big] \Big) \mathbb E\Big[ U_{t,*}^{2k-a} V_{t,*}^{2\ell-b} \Big] \\
    \leq\ & \sum_{ \substack{ a \leq 2k, b \leq 2\ell \\ a+b \geq 3 } } \binom{2k}{a} \binom{2\ell}{b} \frac{(2a+2b)!(C'')^{a+b}}{n^{(a+b)/2}} \mathbb E\Big[ U_{t,*}^{2k} V_{t,*}^{2\ell} \Big] \leq n^{-\frac{3}{2}+o(1)} \mathbb E\Big[ U_{t}^{2k} V_{t}^{2\ell} \Big]  \,,
\end{align*}
where the first inequality follows from Claims~\ref{claim-moment-u-t-v-t} and \ref{claim-moment-U-t-*-V-t-*}, and the second inequality follows from \eqref{eq-compare-U-t-*-V-t-*-and-U-t-V-t}. In conclusion, we have
\begin{align*}
    \mathbb E\big[ U_{t+1}^{2k} V_{t+1}^{2\ell} \big] = (1+n^{-\frac{3}{2}+o(1)}) \mathbb E\big[ U_{t}^{2k} V_{t}^{2\ell} \big] \,.
\end{align*}
Thus, we have
\begin{align*}
    \mathbb E\big[ U_{n}^{2k} V_{n}^{2\ell} \big] = (1+n^{-\frac{3}{2}+o(1)})^n \mathbb E\big[ U_{0}^{2k} V_{0}^{2\ell} \big] = (1+n^{-\frac{1}{2}+o(1)}) \mathbb E\big[ U_{0}^{2k} V_{0}^{2\ell} \big] \,.
\end{align*}
Recall \eqref{eq-def-U-0V-0} and \eqref{eq-def-U-t-V-t}, this completes the proof of Lemma~\ref{lem-compare-moments-Gaussian-non-Gaussian}.

\subsection{Proof of Lemma~\ref{lem-bound-Adv-cov-relax-1}}{\label{subsec:proof-lem-5.10}}

The goal of this section is to prove Lemma~\ref{lem-bound-Adv-cov-relax-1}. The following polynomials will play a fundamental role in our analysis.

\begin{DEF}{\label{def-Hermite-poly}}
    For all $m \in \mathbb N$, define the Hermite polynomials by
    \begin{equation}{\label{eq-def-Hermite-poly}}
        H_0(z)=1 \,, \quad H_1(z)=z \,, \quad H_{m+1}(z)= zH_m(z)-mH_{m-1}(z) \,.
    \end{equation}
    In addition, define $\mathcal H_m(z)=\frac{1}{\sqrt{m!}} H_m(z)$, and for all $x \in \mathbb R^n$ and $\alpha\in \mathbb N^{n}$ define
    \begin{equation}{\label{eq-def-mathcal-H}}
        \mathcal H_{\alpha}(x) := \prod_{i=1}^{d} \mathcal H_{\alpha_i}(x_i) \,.
    \end{equation}
    For all $\bm{\alpha}=(\alpha_i: i \in [N]) \in (\mathbb N^n)^{\otimes N}$ and $\bm{\beta}=(\beta_j:j \in [N]) \in (\mathbb N^n)^{\otimes N}$, define
    \begin{equation}{\label{eq-def-phi-alpha,beta}}
        \mathcal H_{\bm{\alpha},\bm{\beta}}(\bm X,\bm Y) = \prod_{1\leq i \leq N} \mathcal H_{\alpha_i}(\bm X_i) \prod_{1 \leq j \leq N} \mathcal H_{\beta_j}(\bm Y_j) \,. 
    \end{equation}
\end{DEF}

It is well known (see, e.g., \cite{Sze39}) that $\{ \phi_{\bm{\alpha},\bm \beta}: |\bm\alpha|+|\bm\beta| \leq D \}$ forms a standard orthogonal basis of $\mathbb R[X,Y]_{\leq D}$ under the measure $\overline\Qb$, i.e., we have
\begin{equation}{\label{eq-standard-orthogonal-basis}}
    \mathbb E_{\overline\Qb}\big[ \phi_{\bm \alpha,\bm \beta} \phi_{\bm \alpha',\bm \beta'} \big] = \mathbf 1_{ \{ (\bm\alpha,\bm\beta)=(\bm\alpha',\bm\beta') \} } \,.
\end{equation}

\begin{lemma}{\label{lem-Adv-cov-transform}}
    For any $n,N,D\geq 1$, it holds that
    \begin{equation}{\label{eq-Adv-cov-transform}}
        \mathsf{Adv}_{\leq D}(\overline\Pb,\overline\Qb) = \Bigg( \sum_{ (\bm \alpha,\bm \beta): |\bm \alpha|+|\bm \beta|\leq D } \mathbb E_{\overline\Pb}\Big[ \phi_{\bm \alpha,\bm \beta}(X,Y) \Big]^2 \Bigg)^{1/2} \,.
    \end{equation}
\end{lemma}
\begin{proof}
    For any $f \in \mathbb R[\bm X,\bm Y]_{\leq D}$, it can be uniquely expressed as
    \begin{align*}
        f=\sum_{ (\bm \alpha,\bm \beta): |\bm \alpha|+|\bm \beta|\leq D } C_{\bm \alpha,\bm \beta} \phi_{\bm \alpha,\bm \beta} \,,
    \end{align*}
    where $C_{\bm \alpha,\bm \beta}$'s are real constants. Applying Cauchy-Schwartz inequality one gets
    \begin{align*}
        \frac{ \mathbb{E}_{\overline\Pb}[f] }{ \sqrt{\mathbb{E}_{\overline\Qb}[f^2]} } = \frac{ \sum_{ (\bm \alpha,\bm \beta): |\bm \alpha|+|\bm \beta|\leq D } C_{\bm \alpha,\bm \beta} \mathbb{E}_{\overline\Pb}[\phi_{\bm \alpha,\bm \beta}] }{ \sqrt{ \sum_{ (\bm \alpha,\bm \beta): |\bm \alpha|+|\bm \beta|\leq D } C_{\bm \alpha,\bm \beta}^2} } \leq \Bigg( \sum_{ (\bm \alpha,\bm \beta): |\bm \alpha|+|\bm \beta|\leq D }\big(\mathbb E_{\overline\Pb} [\phi_{\bm \alpha,\bm \beta}]\big)^2 \Bigg)^{1/2} \,,
    \end{align*}
    with equality holds if and only if $C_{\bm \alpha,\bm \beta} \propto \mathbb{E}_{\overline\Pb}[\phi_{\bm \alpha,\bm \beta}]$. 
\end{proof}

Now we can finish the proof of Lemma~\ref{lem-bound-Adv-cov-relax-1}.

\begin{proof}
    Denote $\overline\Pb_{x,y}=\overline\Pb(\cdot\mid x,y)$. Note that
    \begin{align*}
        \mathbb L(\bm X,\bm Y)= \frac{ \mathrm{d}\overline\Pb(\bm X,\bm Y) }{ \mathrm{d}\overline\Qb } = \mathbb E_{(x,y) \sim\pi}\Bigg[ \frac{ \mathrm{d}\overline\Pb_{x,y}(\bm X,\bm Y) }{ \mathrm{d}\overline\Qb } \Bigg] = \mathbb E_{(x,y) \sim\pi}\Big[ \mathbb L_x(\bm X) \mathbb L_y(\bm Y) \Big] \,,
    \end{align*}
    where 
    \begin{align*}
        &\mathbb L_x(\bm X) = \frac{ 1 }{ (1+\frac{\lambda}{n}\|x\|^2) } \exp\Big( \frac{ \frac{\lambda}{n}( \langle x,\bm X_1 \rangle^2 + \ldots + \langle x,\bm X_N \rangle^2 ) }{ 2(1+\frac{\lambda}{n}\|x\|^2) } \Big) \,; \\
        &\mathbb L_y(\bm Y) = \frac{ 1 }{ (1+\frac{\lambda}{n}\|y\|^2) } \exp\Big( \frac{ \frac{\lambda}{n}( \langle y,\bm Y_1 \rangle^2 + \ldots + \langle y,\bm Y_N \rangle^2 ) }{ 2(1+\frac{\lambda}{n}\|y\|^2) } \Big) \,.
    \end{align*}
    Thus, we have
    \begin{align}
        \mathbb E_{\overline\Pb}\big[ \mathcal H_{\bm\alpha,\bm\beta}(\bm X,\bm Y) \big] &= \mathbb E_{\overline\Qb}\Big[ \mathbb L(\bm X,\bm Y) \cdot \mathcal H_{\bm\alpha,\bm\beta}(\bm X,\bm Y) \Big] \nonumber \\
        &= \mathbb E_{ (x,y)\sim\pi } \Bigg\{ \mathbb E_{\overline\Qb}\Big[ \mathbb L_x(\bm X) \mathcal H_{\bm\alpha,\emptyset}(\bm X) \Big] \mathbb E_{\overline\Qb}\Big[ \mathbb L_y(\bm Y) \mathcal H_{\emptyset,\bm\beta}(\bm Y) \Big] \Bigg\} \,.  \label{eq-separate-X,Y-in-Adv-cov}
    \end{align}
    In addition, following the calculation in \cite[Section~A.2]{BKW20}, we get that
    \begin{align*}
        &\mathbb E_{\overline\Qb}\Big[ \mathbb L_x(\bm X) \mathcal H_{\bm\alpha,\emptyset}(\bm X) \Big] = \mathbf 1_{ \{ |\alpha_i| \text{ is even for all } i \} } \big(\tfrac{\lambda}{n}\big)^{|\bm \alpha|/2} x^{\alpha_1+\ldots+\alpha_N} \prod_{i=1}^{N} \frac{ (|\alpha_i|-1)!! }{ \alpha_i! } \,;  \\
        &\mathbb E_{\overline\Qb}\Big[ \mathbb L_y(\bm Y) \mathcal H_{\emptyset,\bm\beta}(\bm Y) \Big] = \mathbf 1_{ \{ |\beta_i| \text{ is even for all } i \} } \big(\tfrac{\mu}{n}\big)^{|\bm \beta|/2} y^{\beta_1+\ldots+\beta_N} \prod_{i=1}^{N} \frac{ (|\beta_i|-1)!! }{ \beta_i! } \,.
    \end{align*}
    Thus, we have $\mathbb E_{\overline\Pb}[\mathcal H_{\bm\alpha,\bm\beta}(\bm X,\bm Y)]$ equals
    \begin{align*}
         \mathbb E_{(x,y)\sim\pi}\Bigg\{ \mathbf 1_{ \{ |\alpha_i|,|\beta_i| \text{ is even for all } i \} } \big(\tfrac{\lambda}{n}\big)^{|\bm \alpha|/2} \big(\tfrac{\mu}{n}\big)^{|\bm \beta|/2} x^{\alpha_1+\ldots+\alpha_N} y^{\beta_1+\ldots+\beta_N} \prod_{i=1}^{N} \frac{ (|\alpha_i|-1)!!(|\beta_i|-1)!! }{ \alpha_i!\beta_i! } \Bigg\} \,.
    \end{align*}
    Plugging this equation into Lemma~\ref{lem-Adv-cov-transform}, we get that $\chi^2_{\leq D}(\overline\Pb;\overline\Qb)$ is bounded by
    \begin{align}
        & \sum_{ \substack{ |\bm\alpha|,|\bm\beta|\leq D \\ |\alpha_i|,|\beta_i| \text{ is even} } } \mathbb E_{(x,y)\sim\pi}\Bigg\{ \Big(\frac{\lambda}{n}\Big)^{|\bm \alpha|/2} \Big(\frac{\mu}{n}\Big)^{|\bm \beta|/2} x^{\alpha_1+\ldots+\alpha_N} y^{\beta_1+\ldots+\beta_N} \prod_{i=1}^{N} \frac{ (|\alpha_i|-1)!!(|\beta_i|-1)!! }{ \alpha_i!\beta_i! } \Bigg\}^2 \nonumber \\
        =\ & \mathbb E_{(x,y),(x',y')\sim\pi}\Bigg\{ \sum_{ \substack{ |\bm\alpha|,|\bm\beta|\leq D \\ |\alpha_i|,|\beta_i| \text{ is even} } } \Big(\frac{\lambda}{n}\Big)^{|\bm \alpha|/2} \Big(\frac{\mu}{n}\Big)^{|\bm \beta|/2} \prod_{i=1}^{N} \frac{ ((|\alpha_i|-1)!!)^2 ((|\beta_i|-1)!!)^2 }{ (\alpha_i!)^2 (\beta_i!)^2 } \nonumber \\
        & \quad\quad\quad\quad\quad\quad\quad\quad x^{\alpha_1+\ldots+\alpha_N} y^{\beta_1+\ldots+\beta_N} (x')^{\alpha_1+\ldots+\alpha_N} (y')^{\beta_1+\ldots+\beta_N} \Bigg\} \,.  \label{eq-replica-form-Adv-cov}
    \end{align}
    Again, following the calculation in \cite[Section~A.2]{BKW20}, we have that
    \begin{equation}{\label{eq-combinatorial-identity-varphi-D}}
        \begin{aligned}
            & \sum_{ \substack{ |\bm\alpha| \leq D \\ |\alpha_i| \text{ is even} } } \Big(\frac{\lambda}{n}\Big)^{|\bm \alpha|} x^{\alpha_1+\ldots+\alpha_N} (x')^{\alpha_1+\ldots+\alpha_N} \prod_{i=1}^{N} \frac{ ((|\alpha_i|-1)!!)^2 }{ (\alpha_i!)^2 } = \varphi_{D}\big( \tfrac{\lambda^2 \langle x,x' \rangle^2 }{4n^2} \big) \,;  \\
            & \sum_{ \substack{ |\bm\beta| \leq D \\ |\beta_i| \text{ is even} } } \Big(\frac{\mu}{n}\Big)^{|\bm \beta|} y^{\beta_1+\ldots+\beta_N} (y')^{\beta_1+\ldots+\beta_N} \prod_{i=1}^{N} \frac{ ((|\beta_i|-1)!!)^2 }{ (\beta_i!)^2 } = \varphi_{D}\big( \tfrac{\lambda^2 \langle x,x' \rangle^2 }{4n^2} \big) \,;
        \end{aligned}
    \end{equation}
    Plugging \eqref{eq-combinatorial-identity-varphi-D} into \eqref{eq-replica-form-Adv-cov} yields Lemma~\ref{lem-Adv-cov-transform}.
\end{proof}

\bibliographystyle{alpha}

\end{document}